\documentclass{amsart}

\usepackage[english]{babel}

\usepackage[letterpaper,top=2cm,bottom=2cm,left=3cm,right=3cm,marginparwidth=1.75cm]{geometry}
\usepackage{biblatex} %Imports biblatex package
\usepackage{amsmath, amsfonts, amssymb, mathtools, amsthm, relsize}
\DeclareMathOperator{\coeq}{coeq}
\DeclareMathOperator{\eq}{eq}
\DeclareMathOperator{\diag}{diag}
\DeclareMathOperator*{\colim}{colim}
\DeclareMathOperator*{\hocolim}{hocolim}
\DeclareMathOperator*{\holim}{holim}
\DeclareMathOperator{\spec}{Spec}

\DeclareFieldFormat{pages}{#1}
\DeclareFieldFormat{postnote}{#1}

\newcommand{\completed}{\diamond}
\newcommand{\solid}{{\mathsmaller{\blacksquare}}}
\DeclareMathOperator{\Tot}{Tot}
\usepackage{wrapfig}
\usepackage{enumerate}
\usepackage{lscape}
\usepackage{rotating}
\usepackage{stmaryrd}
\numberwithin{equation}{section}
\newtheorem{conjecture}[equation]{Conjecture}
\newtheorem{remark}[equation]{Remark}
\newtheorem{proposition}[equation]{Proposition}
\newtheorem{theorem}[equation]{Theorem}
\newtheorem{lemma}[equation]{Lemma}
\newtheorem{definition}[equation]{Definition}
\usepackage{svg}
\usepackage{comment}
\usepackage{tikz-cd}
\usepackage{graphicx}
\usepackage[colorlinks=true, allcolors=blue]{hyperref}
\newcommand{\C}{\mathbb{C}}
\newcommand{\EE}{\mathbb{E}}
\newcommand{\Z}{\mathbb{Z}}
\newcommand{\F}{\mathbb{F}}

\newcommand{\Q}{\mathbb{Q}}
\newcommand{\R}{\mathbb{R}}
\newcommand{\bU}{\mathbb{U}}
\newcommand{\Ho}{\textrm{Ho}}

\newcommand{\Grad}{\nabla}

\renewcommand{\EE}{\mathbb{E}}
\renewcommand{\SS}{\mathbf{S}}
\newcommand{\bS}{\mathbf{S}}
\newcommand{\tensor}{\otimes}
\newcommand{\sma}{\wedge}
\newcommand{\CC}{\mathcal{C}}

\newcommand{\Sp}{\textrm{Sp}}
\newcommand{\htp}{\simeq}
\newcommand{\op}{\textrm{op}}
\newcommand{\cyc}{\textrm{cyc}}
\newcommand{\id}{\textrm{id}}
\newcommand{\sd}{\textrm{sd}}
\newcommand{\kk}{\mathbf{k}}
\newcommand{\Alg}{\textrm{Alg}}
\newcommand{\CAlg}{\textrm{CAlg}}
\newcommand{\noloc}{\;{:}\,}
\newcommand{\GTop}{\textrm{GTop}}
\newcommand{\AMod}{\textrm{A-Mod}}
\newcommand{\ZMod}{\mathbb{Z}\textrm{-Mod}}
\newcommand{\FpMod}{\mathbb{F}_p\textrm{-Mod}}
\newcommand{\Mod}{\textrm{Mod}}
\newcommand{\ModC}[1]{#1\textrm{-Mod}}
\newcommand{\Map}{\textrm{Map}}
\newcommand{\Fun}{\textrm{Fun}}
\newcommand{\trianglespace}{\boldsymbol{\Delta}}
\newcommand{\cosimptri}{{\trianglespace}}
\newcommand{\RHom}{\textrm{RHom}}
\newcommand{\Wedge}{\bigwedge\nolimits}

\newcommand{\aC}{\mathcal{C}}
\newcommand{\aD}{\mathcal{D}}
\newcommand{\varSS}{{\SS'}} 
\newcommand{\Hom}{\textrm{Hom}}
\newcommand{\hofib}{\textrm{hofib}}
\makeatletter
\newcommand*{\relrelbarsep}{.386ex}
\newcommand*{\relrelbar}{%
  \mathrel{%
    \mathpalette\@relrelbar\relrelbarsep
  }%
}
\newcommand*{\@relrelbar}[2]{%
  \raise#2\hbox to 0pt{$\m@th#1\relbar$\hss}%
  \lower#2\hbox{$\m@th#1\relbar$}%
}
\providecommand*{\rightrightarrowsfill@}{%
  \arrowfill@\relrelbar\relrelbar\rightrightarrows
}
\providecommand*{\leftleftarrowsfill@}{%
  \arrowfill@\leftleftarrows\relrelbar\relrelbar
}
\providecommand*{\xrightrightarrows}[2][]{%
  \ext@arrow 0359\rightrightarrowsfill@{#1}{#2}%
}
\providecommand*{\xleftleftarrows}[2][]{%
  \ext@arrow 3095\leftleftarrowsfill@{#1}{#2}%
}
\makeatother

\title{Noncommutative Cartier Formulae}
\author{Semon Rezchikov}

\begin{document}

\begin{abstract}
We prove, for every $\EE_1$ algebra $A$, a formula describing the interaction of the action of the
cap product on topological Hochschild homology of $A$ with the cyclotomic structure map, as well as
a variant of this result relative to a ring $R$. Specializing to $R = \F_p$ gives a noncommutative
analog of a formula of Cartier which describes the conjugation of interior product action on
differential forms by the Cartier isomorphism, and which computes the $p$-curvature of the
Getzler-Gauss-Manin connection in terms of an equivariant cap product. The motivation for this
formula comes from symplectic geometry, where (in the case $R=\F_p$ or a Novikov analog) the
symplectic analog of this formula explains the interaction between the cyclotomic structure on
symplectic cohomology and the quantum Steenrod operations. We prove, under standard transversality
and nondegeneracy assumptions on the Fukaya category, that for a Calabi-Yau symplectic manifold with
rational symplectic form, the $p$-curvature of the quantum connection computes the Quantum Steenrod
operations. In particular, the $p$-curvature of the quantum connections of projective Calabi-Yau
hypersurfaces, and many other examples in mirror symmetry, can be interpreted in terms of
$\Z/p\Z$-equivariant genus zero Gromov-Witten invariants. 
\end{abstract}

\maketitle

\section{Introduction}

The purpose of this paper is to establish a noncommutative analog of a formula of Cartier connecting
the interior product and the Cartier isomorphism, in the form of an analogous result for all
$\EE_1$-ring spectra $A$ and a relative variant for all $\EE_1$-ring spectra $A$ over an
$\EE_\infty$-ring spectrum $R$. The arguments and results require only homotopical algebra; however,
the motivation arises from symplectic geometry, and the author discovered the analog of these
formulae first in that setting, where they admit a simple pictorial interpretation (and `picture
proof'). The formulae describe the relation between the cap product action of Hochschild cohomology
$THC$ on Hochschild homology $THH$, and the cyclotomic structure on topological Hochschild homology.
One can formulate an expected generalization of these formulae, articulating the compatibility
conditions between the $\EE_2$-structure on $THC$ and the cyclotomic structure on $THH$; we refer to
this as the \emph{cyclotomic Deligne conjecture} (Appendix \ref{app:cyclotomic-deligne-conjecture}). 
In the setting of symplectic topology, this formula corresponds to a result describing the
relationship between the quantum product, the equivariant quantum product (or `quantum Steenrod
operation'), and the cyclotomic structure on symplectic cohomology, thus clarifying the
algebro-geometric interpretation of the quantum Steenrod operations under mirror symmetry, and
giving an \emph{explanation} of the relationship between the quantum product and the quantum
Steenrod operation that has been established in various special cases in previous works.

We now state our results. For any associative ring spectrum $A$, we can form the topological
Hochschild homology spectrum $THH(A) \htp A \sma_{A \sma A^{\op}} A$ and the topological Hochschild
cohomology spectrum $THC(A) \htp F_{A \sma A^{\op}}(A,A)$.  Here $F$ denotes the mapping spectrum in
the category $A \sma A^{\op}$-modules (i.e., $A$-$A$ bimodules).  In this introduction, we are
implicitly working homotopically and so throughout we mean the derived smash product and derived
mapping spectrum.

There is a cap product map~\cite[\S 4.3.1]{malm2010string},\cite[\S5]{angeltveit-hill-lawson}
\begin{align}
\label{eq:cap-product-THH}
    \cap\colon &THC(A) \sma THH(A) \to THH(A) \\
    &F_{A \sma A^{\op}}(A,A) \sma (A \sma_{A \sma A^{\op}} A) \to A \sma_{A \sma A^{\op}} A 
\end{align}
of spectra inducing on homotopy groups the cap product
\begin{equation}
\cap\colon THC^{p}(A) \sma THH_q(A) \to THH_{q-p}(A).
\end{equation}
The cap product is defined by applying an endomorphism of $A$ to the first copy of $A$ in the smash
product.
This generalizes the algebraic action of Hochschild cohomology on Hochschild
homology~\cite{loday2013cyclic}. Now, recall that $THH(A)$ has the structure of an $S^1$-spectrum
which is a well-defined homotopy type with respect to the weak equivalences detected on passages to
finite subgroups of $S^1$ (e.g., see~\cite{angeltveit2018topological, nikolaus-scholze}).  Our basic
construction is a $p$-fold cap product map 
\begin{equation}
    \label{eq:p-fold-cap-product}
    \cap^p\colon N_e^{C_p} THC(A) \sma THH(A) \to THH(A)
\end{equation}
in the category of $C_p$-spectra, where we view $THH(A)$ as a genuine $C_p$-spectrum by forgetting
the $S^1$-action, and $N_e^{C_p}$ denotes the Hill-Hopkins-Ravenel norm.  Morally speaking, the
$C_p$ action on the domain is the smash product of the $C_p$ action which permutes the tensor
factors of $THC(A)^{\sma p}$ with the usual $C_p$ action on $THH(A)$.

The category $\Sp_G$ of $G$-equivariant spectra admits three distinct notions of fixed points. 
Unlike the category of $G$-spaces, where there are the strict fixed points $X^{G}$ and the homotopy
fixed points $(EG \sma X)^{G}$, in $G$-spectra we also have the geometric fixed points $\Phi^{G} X$. 
These are characterized homotopically by the requirement that they are symmetric monoidal, commute
with homotopy colimits, and lift the strict fixed points on spaces in the sense that
$\Phi^{G}(\Sigma_+^{\infty} X) \htp \Sigma_+^{\infty} X^{G}$.  For $H$ a normal subgroup of $G$, the
geometric fixed points $\Phi^H X$ for a $G$-spectrum $X$ are naturally a $G/H$-spectrum.

When we lift $HH$ to $THH$, there is an additional structure that appears, the {\em cyclotomic}
structure map.  This models the self-equivalence of the free loop space $(LX)^{C_p} \htp LX$, and
can be described in terms of a map $\phi\colon THH(A) \to \Phi^{C_p} THH(A)$ that is
$S^1$-equivariant when the target is given the $S^1$-action induced by the isomorphism $S^1 / C_p
\cong S^1$.

Our first main theorem expresses the compatibility of the cap product with the cyclotomic structure
map, using the $p$-fold cap product map from equation~\eqref{eq:p-fold-cap-product}.

\begin{theorem}
\label{thm:absolute-bk-formula}
    There is a commutative square in the $\infty$-category of spectra
    \begin{equation}
\label{eq:genuine-bk-formula}
    \begin{tikzcd}
        THC(A) \sma THH(A) \ar[rr, "\cap"] \ar[d, "\Delta \sma \phi", swap] & & THH(A) \ar[d,
        "\phi"] \\
        \Phi^{C_p} (N_e^{C_p} THC(A) \sma THH(A)) \ar[rr, "\Phi^{C_p} (\cap^p)", swap] & &
        \Phi^{C_p} THH(A). 
    \end{tikzcd}
\end{equation}
Here $\Delta \colon THC(A) \to \Phi^{C_p} N_e^{C_p} THC(A)$ denotes the isomorphism given by the
Hill-Hopkins-Ravenel diagonal \cite{hill2016nonexistence} (which we will henceforth refer to as the
\emph{diagonal}), and the lefthand vertical arrow is the composite
\begin{equation}
THC(A) \sma THH(A) \to (\Phi^{C_p} N_e^{C_p} THC(A)) \sma \Phi^{C_p} THH(A) \to \Phi^{C_p}
(N_e^{C_p} THC(A) \sma THH(A)).
\end{equation}
\end{theorem}

We now consider the analogue of this story relative to a base commutative ring spectrum $R$.  That
is, let $A$ be an associative $R$-algebra.  In this setting, there is a notion of $THH$ of $A$
relative to $R$.  This is sometimes written $THH_R(A)$, but we will denote it by $THH(A/R)$,
following~\cite{nikolaus-scholze}.  By construction, $THH(A/R)$ has the natural structure of an
$R$-module spectrum.  Although $THH(A/R)$ is by definition $A \sma_{A \sma_R A^{\op}} A$, there is a
convenient formula describing the relative theory:
\begin{equation}
THH(A/R) \htp THH(A) \sma_{THH(R)} R.
\end{equation}
We can also construct a relative version of $THC$, which we write as $THC(A/R) \htp F_{A \sma_R
A^{\op}}(A,A)$, and there is a relative cap product
\begin{equation}
\cap_R \colon THC(A/R) \sma_R THH(A/R) \to THH(A/R).
\end{equation}

\begin{remark}
In the case when $R$ is a commutative ring and $A$ is a differential graded algebra over $R$, the
relative spectrum $THH(HA/HR)$ is equivalent to the Hochschild chain complex of $A$ relative to
$R$~\cite[IX.1.7]{EKMM}.
\end{remark}

Just as in the absolute theory where $THH(A)$ is an $S^1$-equivariant
orthogonal spectrum, the relative construction $THH(A/R)$ is an
$S^1$-equivariant $R$-module.  However, in contrast to the absolute
setting, $THH(A/R)$ is not always a cyclotomic spectrum, but when $R$
has the structure of a {\em cyclotomic
  base}~\cite{blumberg-mandell-yuan}, then $THH(A/R)$ can be endowed
with a cyclotomic structure.  In this case, we have the following
relative version of the cap product formula. 

\begin{theorem}
\label{thm:relative-bk-formula-cyclotomic-base}
There is a commutative square in the $\infty$-category of spectra
\begin{equation}
\label{eq:genuine-bk-formula-relative}
    \begin{tikzcd}
        THC(A/R) \sma_R THH(A/R) \ar[rr, "\cap_R"] \ar[d, "\Delta \sma_R \phi", swap] & & THH(A/R)
        \ar[d, "\phi"] \\
        _R \Phi^{C_p} (_R N_e^{C_p} THC(A/R) \sma_R THH(A/R)) \ar[rr, "_R \Phi^{C_p} (\cap^p_R)",
        swap] & & _R \Phi^{C_p} THH(A/R). 
    \end{tikzcd}
\end{equation}
\end{theorem}

Even when $R$ is not a cyclotomic base, there is a very interesting formula relating the cap product
to the cyclotomic structure on $A$.  Observe that there is an evident collapse map
\begin{equation}
    \label{eq:equivariant-collapse}
    THH(A) \to THH(A/R)
\end{equation}
of equivariant spectra and an analogous map
\begin{equation}
THC(A/R) \to THC(A)
\end{equation}
induced by the collapse map $A \sma A^{\op} \to A \sma_R A^{\op}$.

Furthermore, $THH(R)$ is an equivariant commutative ring spectrum, $THH(A)$ is an equivariant module
over $THH(R)$, and the cyclotomic structure map is linear relative to the module structure induced
by the map $THH(R) \to THH(R)^{\Phi C_p}$. 

\begin{theorem}
\label{thm:relative-bk-formula}
The following diagram commutes in the $\infty$-category of spectra, and the outer square commutes in
the category of $THH(R)$-modules:
\begin{equation}
\label{eq:relative-bk-formula}
\begin{tikzcd}
    THC(A/R) \sma THH(A) \ar[r] \ar[d, "\Delta \sma \phi"] &THC(A) \sma THH(A) \ar[r, "\cap"] \ar[d,
    "\Delta \sma \phi"] &THH(A) \ar[d, "\phi"] \\
     \Phi^{C_p}(THC(A/R)^{\sma p} \sma THH(A)) \ar[r]\ar[d] & \Phi^{C_p} (THC(A)^{\sma p} \sma
     THH(A)) \ar[r, "\Phi^{C_p}(\cap^p)", swap] & \Phi^{C_p} THH(A) \ar[d] \\
    \Phi^{C_p}(_R N_e^{C_p} THC(A/R) \sma_R THH(A/R)) \ar[rr, "\Phi^{C_p}(\cap^p_R)", swap] &&
    \Phi^{C_p} THH(A/R).
\end{tikzcd}
\end{equation}

Here we give the bottom elements the structure of a $THH(R)$-module via the map $THH(R) \to
(THH(R))^{\Phi C_p} \to R^{\Phi C_p}$, where $R$ is thought of as a genuine $C_p$-spectrum with a
trivial action and the second map is induced by applying $\Phi^{C_p}$ to the map $THH(R) \to R$. 
\end{theorem}

Although the theorems are stated in terms of the $\infty$-category of spectra,
we use point-set models in terms of equivariant orthogonal spectra to reduce the proofs to a
combinatorial verification using simplicial and cosimplicial manipulations.
The technical approach is to give a construction of the cap product~\eqref{eq:cap-product-THH} based
on the prismatic subdivision used in McClure-Smith's proof of the Deligne conjecture
\cite{mcclure1999solution} and study the interaction of this explicit map with the construction of
the cyclotomic structure map in terms of the norm~\cite{angeltveit2018topological}.  Ensuring that
we have homotopical control requires the new equivariant model structures of~\cite{cyc23}.

\begin{remark}
We have phrased our results in terms of the ``neo-classical'' description of the cyclotomic
structure on $THH$ in terms of the norm.  However, replacing $\Phi^{C_p}$ with the Tate fixed points
$tC_p$ everywhere, and interpreting $\phi$ as the cyclotomic structure map of Nikolaus-Scholze and
$\delta$ as the Tate diagonal~\cite{nikolaus-scholze}, we can state the analogous theorems.
\end{remark}

\begin{remark}
Although we have stated our results in terms of a ring spectrum or $R$-algebra $A$, in fact these
theorems hold when $A$ is a small stable $\infty$-category (e.g., a pretriangulated spectral or dg
category).  When this category has a compact generator (e.g., we are studying compact dg modules
over a ring $A$ or compact module spectra or a ring spectrum $A$), the usual Morita theory for $THH$
and $THC$ implies that this just reduces to the case of ring spectra~\cite{blumberg2012localization,
blumberg2019E2}.
\end{remark}

We now explain the application of these algebraic results to symplectic topology. Symplectic
geometry and complex geometry are each, in their own right, a source of interesting linear ordinary
differential equations. In the algebro-geometric setting, these differential equations, called
\emph{Gauss-Manin connections} define the parallel transport of cycles in the cohomology bundle
$H^*(X/S)$ on a smooth proper family of complex varieties $X$ over a base $S$. The underlying
differential equations can be defined directly using algebraic de Rham cohomology, and this makes
sense for an arbitrary base $S$; for example, when $S = \spec R,\; R= \kk[[x]]$,  the construction
produces a linear differential operator 
\begin{equation}
    \label{eq:gauss-manin-connection-formal}
    \Grad^{GM}_{x\partial/\partial_x} = x\partial_x + A, A \in Mat_{n \times n}(\kk[[x]]). 
\end{equation} 
In the symplectic setting, corresponding differential equations are constructed out of the operation
of quantum multiplication on quantum cohomology. Explicitly, given a symplectic manifold $M$, one
defines the \emph{quantum differential equation} 

\begin{equation}
    \label{eq:quantum-connection-formal}
    \Grad^{QM}_{x \partial_x} = ux\partial_x + [\omega] *
\end{equation}
where $*$ denotes the multiplication operator on quantum cohomology by the class of the symplectic
form. (Here $x$ is the `Novikov variable', which keeps track of the areas (or degrees) of holomorphic
curves, and we are implicitly assuming that $[\omega] \in H^2(M, \Q)$.) 

In the phenomenon of \emph{enumerative mirror symmetry}, one sees that associated to certain
symplectic manifolds $M$ there exist mirror complex families $X/R$ such that the Gauss-Manin
connection \eqref{eq:gauss-manin-connection-formal} agrees with the connection
\eqref{eq:quantum-connection-formal} after pulling back along an automorphism of the base $R$ (the
mirror map) as well as performing an appropriate gauge transformation, and then setting $u=1$. The
$u$-variable in \eqref{eq:quantum-connection-formal} ends up corresponding to the Griffiths
transversality of the Gauss-Manin connection \cite{ganatra2015mirror}.

Now, the relation between $\Grad^{GM}$ and $\Grad^{QM}$ is connected to homological mirror symmetry,
which posits (in certain cases) an identification between the stable $R$-linear $\infty$-categories
$D^bCoh(X/R)$ and the Fukaya category $Fuk(M)$ of $M$. In important cases, such as the case where
$M$ is a Calabi-Yau hypersurface in projective space, this conjecture has been established
\cite{sheridan2015homological}, which allows one to make computations in $Fuk(M)$. An important part
of the toolkit is the cyclic open-closed map \cite{ganatra2019cyclic}

\[ HH_\bullet(Fuk(M)) \to QH^{\bullet+n}(M) \]
which intertwines the homotopy-$S^1$-actions on the complexes on either side. Expected properties of
open-closed should prove that homological mirror symmetry implies the enumerative mirror symmetry in
many cases of interest \cite{ganatra2015mirror}.

Now, in the setting of monotone symplectic manifolds, the cyclic open-closed map has been enhanced
to a map intertwining not just the $S^1$-actions on the two sides but also the $C_p$-actions in a
combinatorially convenient manner \cite{chen2024operadic}. Using this modification,
\cite{chen2024quantum} identifies the product 

\[ \cap^p_{R_p}: HH^a(Fuk(M, R_p)) \tensor HH_b(Fuk(M, R_p))^{tC_p} \to HH_{b -pa}(Fuk(M,
R_p))^{tC_p}\]

(which agrees, up to a straightforward comparison, with the definition of our product) with the
\emph{Quantum Steenrod operations} on $M$ (see \cite[Theorem 4.6]{chen2024quantum} and also Section
\ref{sec:comparing-conventions} of this paper for comparisons between symplectic and algebraic
conventions): 

\begin{equation}
Q\Sigma: H^a(M; \F_p) \tensor H^{-b}(M; \F_p) \to H^{pa-b}(M)[[q, u]]\langle \theta \rangle.
\end{equation}

Here, the Novikov ring is taken to be $\F_p[[q]]$, and in cohomological grading,

\[H^*(BC_p, \F_p) = \F_p[[u]]\langle \theta \rangle, |u|= 2, |\theta|=1. \]

The operations $Q\Sigma$ are in terms of counts of $C_p$-equivariant Gromov-Witten invariants of
$M$, with the equivariance given by rotating the domain curve (and thus does not require $M$ to
carry a group action). These invariants have remarkable applications to dynamics
\cite{shelukhin2022hofer}, and have been used \cite{chen2024quantum} to establish the exponential
type conjecture for the quantum differential equation of monotone symplectic manifolds.

Now, recall that for a $dg$ category $\CC$ over a ring $R$ (or equivalently, for an
$A_\infty$-category over $R$ or a stable $\infty$-category over $R$), there is a connection on the
$R$-module $HP(\CC/R)$ \cite{getzler1993cartan, petrov2018gauss} called the Getzler-Gauss-Manin
connection, which specializes to the Gauss-Manin connection when $\CC = D^bCoh(X/R)$ for $X$ a
smooth proper variety over $R$ (and we are in sufficiently high characteristic relative to the
dimension of $X$ over $R$) via a Hochschild-Kostant-Rosenberg Theorem.

In the algebro-geometric setting, when $S$ is of arithmetic nature, e.g. $S = \spec \;\Z_p[[x]]$,
the Gauss-Manin connection acquires important arithmetic properties \cite{berthelot2015notes}. In
that setting, a basic operation one can perform is to reduce the coefficients of the differential
operator $\Grad^{QM}$ modulo $p$, and study its $p$-curvature

\begin{equation}
    \label{eq:p-curvature}
    F^\Grad_v = \Grad_{v^p} - (\Grad_v)^p. 
\end{equation}
Here, this formula uses the fundamental fact that in characteristic $p$, the $p$-th power $v^p$ of a
derivation $v$ acts as another derivation. 

The $p$-curvature, which is a tensorial (rather than differential) operator, is a fundamental
invariant of a differential operator in characteristic $p$. It participates in fundamental
conjectures about the algebraicity of solutions to an ODE with integer coefficients (the
$p$-curvature conjecture \cite{katz1972algebraic}) and can be used to establish that the monodromy
of an algebraic family of varieties is quasi-unipotent (ibid).

Using Theorem \ref{thm:relative-bk-formula}, we establish the following result
\begin{theorem}
\label{thm:algebraic-p-curvature}
    Let $\CC$ be a smooth proper (homologically graded) dg category over $R = \Z[1/N][[x]]$ or $R =
    \Z[1/N]((x))$. Let 
    \[ m = \sup \{k : HH^k(\CC/R) \neq 0\}, n = - \inf\{k: \CC(L, L')_k \neq 0 \},\] \[r =
    \max\{r_+: HH_{R_+}(\mathcal{C}/R)\neq 0\} - \min \{r_-: HH_{R_-}(\mathcal{C}/R)\neq 0\}\] 
    For any ring $\kk$, write
\[ \hat{\Omega}^1_{\kk((x))/\kk} = \kk((x)) \, dx \quad \textrm{and} \quad \hat{\Omega}^1_{\kk[[x]]}
= \kk[[x]] \, dx;\]
let 
\[ \CC_p = \CC \tensor_\Z \F_p, R_p = R \tensor_\Z \F_p.\]
    Then for all 
    \[p > f(m, n, N, r) = \max(N-1,\max(m, n)/2+1, r/2)  \]
    such that $HH_*(\CC/R)$ is $p$-torsion-free, the $p$-curvature of the mod-$p$
    Getzler-Gauss-Manin connection 
    \[ F^{\Grad^{GGM}}: HP_*(\CC_p/R_p) \to HP_{*-2}(\CC_p/R_p) \, \hat{\tensor}_R \,
    \hat{\Omega}^1_{R/\F_p} \]
    is given by the equivariant $p$-fold cap product with the Kodaira-Spencer class corresponding to
    $\Grad$ i.e.  
    \[F^{\Grad^{GGM}}_\xi(v) = (\Delta_{R}(-u^{-1}e_\kappa(\xi)))\cap^p_{R} v.\]
    Here $\Delta_R$ is the relative Tate diagonal (Lemma \ref{lemma:tate-diagonal-comparison}), which is
    equivalent to Kaledin diagonal after an appropriate multiplication by $u$ (Lemma
    \ref{lemma:tate-diagonal-comparison}), and the hat notation denotes $x$-adic completion and the
    corresponding completed tensor product. 
\end{theorem}
\begin{remark}
    In particular given a smooth proper category $\CC/R$, the result always holds for all but
    finitely many primes $p$.  By a base change argument, one can thus derive a variant of the
    result for arbitrary smooth proper categories over an arbitrary base smooth over $\Z[1/N]$.
\end{remark}

From the perspective of symplectic geometry, the $p$-curvature of the quantum connection has been
the subject of computational investigations. Below we state a conjecture, a variant of which has
been stated by Jae Hee Lee \cite{lee2023quantum, chen2024quantum} for monotone symplectic manifolds:

\begin{conjecture}[Symplectic $p$-curvature conjecture.]
\label{conj:symplectic-p-curvature}
Let $M$ be a semipositive symplectic manifold. Let 
\[Q\Sigma_{\omega} = F^{\Grad^{QM}}_{x \partial x}. \]
\end{conjecture}
\begin{remark}
In fact, the claim above immediately computes $Q\Sigma_{\omega}$ for any class in $[\omega]\in
H^2(M, \F_p)$ by standard properties of quantum cohomology (using multiparameter Novikov rings). The
quantum Cartan relation 
\[ Q\Sigma_a Q\Sigma_b = Q\Sigma_{a * b} \]
then allows for the computation of $Q\Sigma_a$ for many other classes $a$ of even degree. 
\end{remark}

Now, in the aforementioned mirror symmetry settings, variants of the mirror symmetry conjecture have
been proven over $R = \Z[1/N][[x]]$, conditional on the same expected properties of the relative
Fukaya category. Theorem \ref{thm:algebraic-p-curvature} allows us to compute the quantum Steenrod
operations, under similar properties of the relative Fukaya category:
\begin{theorem}
\label{thm:symplectic-p-curvature-for-calabi-yau}
    Let $M$ be a Calabi-Yau symplectic manifold. Under Assumptions A-G of Section
    \ref{sec:putting-it-all-together},  the symplectic $p$-curvature conjecture holds for $M$ for
    all $p>f(\dim M, b, N, \dim M)$, where $b$ is the maximal grading of a nonzero Floer cohomology
    group $HF^b(L_0, L_1)$ between a pair of Lagrangians used in defining the Fukaya category.
    \end{theorem}
\begin{remark}
Assumptions A-F can be summarized as: counts of holomorphic disks can be taken to be in $\Z[1/N]$, 
the symplectic manifold is nondegenerate, the appropriate equivariant open-closed map comparison
results have been proven, and there is no $p$-torsion in homology of $M$ for the relevant $p$ for
which we wish to study the quantum Steenrod operations. One expects that all assumptions except
nondegeneracy hold unconditionally, and their inclusion here is a consequence of the state of
Fukaya-categorical foundations. Nondegeneracy, in contrast, is a geometric hypothesis, as it
requires finding a sufficient number of Lagrangian submanifolds in $M$. 
\end{remark}

The argument does not require mirror symmetry to hold, as it only uses intrinsic properties of the
Fukaya category, but the mirror symmetry settings of \cite{sheridan2021homological} are natural
settings where these properties can be verified.

Finally, let us explain the \emph{meaning} of Theorems \ref{thm:absolute-bk-formula} and
\ref{thm:relative-bk-formula} from the perspective of arithmetic and symplectic geometry. The first
motivation comes from differential geometry over $\F_p$; thus, let us specialize Theorem
\ref{thm:relative-bk-formula} to the case where $R = \F_p$ and $A$ is a commutative $\F_p$-algebra.
In that setting, Cartier (\cite{cartier}, see also  \cite{bezrukavnikov2008fedosov}) proved the
analogous formula, namely that the following diagram commutes:
\begin{equation} 
\label{eq:classical-bk-formula}
\begin{tikzcd}
\Wedge^* T_{A/\F_p} \tensor_A \Omega^*_{A/\F_p} \ar[r, "\iota"] \ar[d, "1 \tensor C^{-1}"]&
\Omega^*_{A/\F_p} \ar[d, "C^{-1}"] \\
\bigwedge^* T_{A/\F_p}\tensor_A H^*_{dR}(A/\F_p)   \ar[r, "\iota^{[p]}"] &  H^*_{dR}(A/\F_p).
\end{tikzcd}
\end{equation}
Here $C^{-1}$ is the inverse Cartier isomorphism, $\iota$ is the interior product, and for vector
fields $X \in T_{A/\F_p}$ and differential forms $\alpha \in \Omega^*_{A/\F_p}$, 
\[ \iota^{[p]}(X \tensor \alpha) = \iota_{X^p} \alpha - \mathcal{L}_X^{p-1} \iota_X \alpha\]
where $X^p$ is the derivation on $A$ given by the $p$-th power of the derivation corresponding to
$X$, and $\mathcal{L}_X$ is the Lie derivative with respect to $X$.

Now, famously, for smooth algebras $A$ over $\F_p$, the Hochschild-Kostant-Rosenberg theorem
identifies polyvector fields with relative Hochschild cohomology, and differential forms with
relative Hochschild homology:
\[\Lambda^*T_{A/\F_p} = HH^*(A/\F_p) \quad\textrm{and}\quad \Omega^*_{A/\F_p} = HH_*(A/\F_p). \]
Under this identification, the interior product $\iota$ becomes the
cap product action of Hochschild cohomology on Hochschild homology
\cite{loday2013cyclic}. One might ask to identify the remaining maps
in the diagram \eqref{eq:classical-bk-formula} in terms of operations
on Hochschild (co)homology. It turns out that (at least when
$A=\F_p[x_1, \ldots, x_n]$), the corresponding noncommutative formula
of Theorem \ref{thm:relative-bk-formula} (taking $R=\F_p$) specializes
precisely to the formula of Bezrukavnikov-Kaledin, after replacing
geometric fixed points with Tate fixed points and tensoring the domain
with $R^{tC_p}$ over $THH(R)$ appropriately -- see
Sections~\ref{sec:relative-bk-formula}
and~\ref{eq:p-fold-covers-on-free-loop-space}. 

\begin{figure}[t]
\label{fig:cyclotomic-deligne-conjecture}
\includegraphics[width=\textwidth]{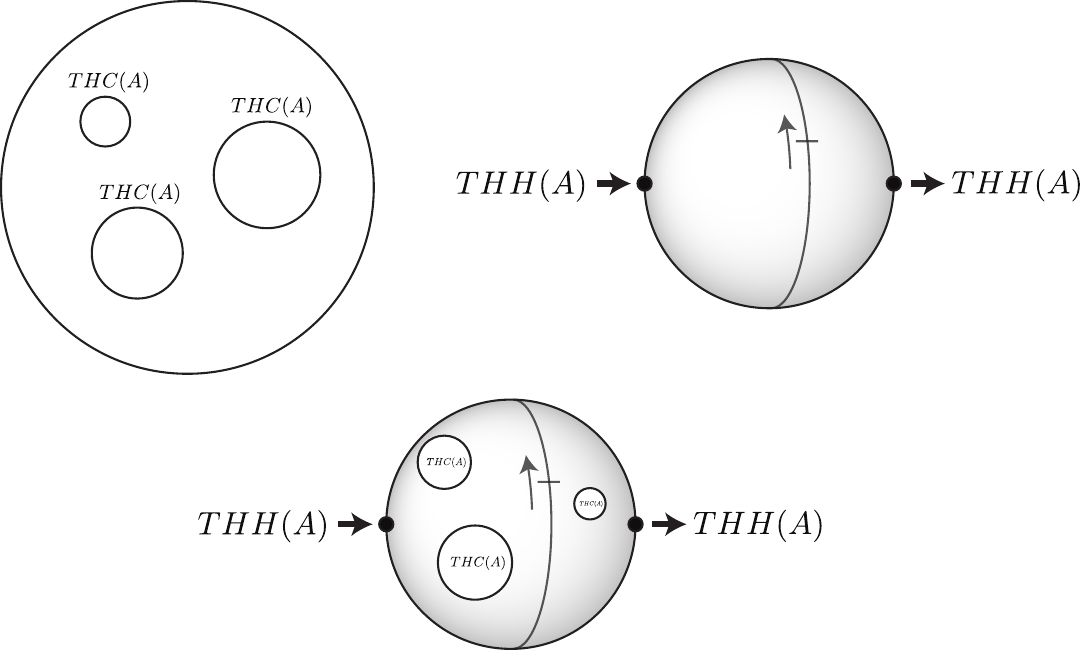}
\caption{\emph{The Cyclotomic Deligne Conjecture.} It is well known that $THC(A)$ is an $\EE_2$
algebra (i.e. an algebra over the little disks operad, upper left), while $THH(A)$ is a genuine
$S^1$-spectrum (i.e. it is acted on by the circle, thought of as a sphere with one input, one
output, and an `angle' marking.) Moreover, $THH(A)$ is a module over $THC(A)$ via the cap product;
the spaces of collections of embedded disks into the sphere with an angle marking (bottom) describe
the module structure of $THH(A)$ over $THC(A)$. Together these spaces define a colored
operad, the rationalization of which is studied in \cite{kontsevich2008notes}. Writing the space on
the bottom as $\mathcal{O}^{\cap}(n)$, there is an additional $S^1$-action on
$\mathcal{O}^{\cap}(n)$ which rotates the embedded disks but fixes the angle marking (thus, this
action is trivial on $\mathcal{O}^{\cap}(0)$, corresponding to the picture on the top-right). There
are $S^1 \to S^1/C_p$-equivariant isomorphisms $\mathcal{O}^{\cap}(k) \to
\mathcal{O}^{\cap}(pk)^{C_p}$, with the fixed points taken with respect to this latter action. We
conjecture that the cyclotomic structure maps commute with these operations. For a formalization,
see Appendix \ref{app:cyclotomic-deligne-conjecture}.   }
\centering
\end{figure}

Thus, Theorem \ref{thm:relative-bk-formula} is a noncommutative generalization of Cartier's
formula \eqref{eq:classical-bk-formula} to an arbitrary base commutative ring spectrum, and we refer
to it as the \emph{noncommutative Cartier formula}. In general, (relative) Topological Hochschild
Cohomology ($THC$) is an $\EE_2$-algebra (over $R$), and $THH$ is a module over $THC$ via the cap
product.  One may ask how the cyclotomic structure on $THH$ interacts with the $\EE_2$-algebra
structure on $THC$. We expect (following Kontsevich-Soibelman, who discuss the rational case
\cite{kontsevich2008notes}) that there is a natural colored operad describing the $\EE_2$ structure
on $THC$, the $S^1$ action on $THH$, the module structure of $THH$ over $THC$, and the compatibility
between these data. Moreover, there are certain equivariant maps between certain spaces associated
to this operad, and a `cyclotomic Deligne conjecture' should hold, describing the compatibility
between the cyclotomic structure on $THH$ and the structure of this colored operad. From this
perspective, the `operadic proof' of Theorem \ref{thm:absolute-bk-formula} should be given by the
picture in Figure \ref{fig:bk-formula-picture-proof}.

\begin{figure}[t]
\label{fig:bk-formula-picture-proof}
\includegraphics[width=\textwidth]{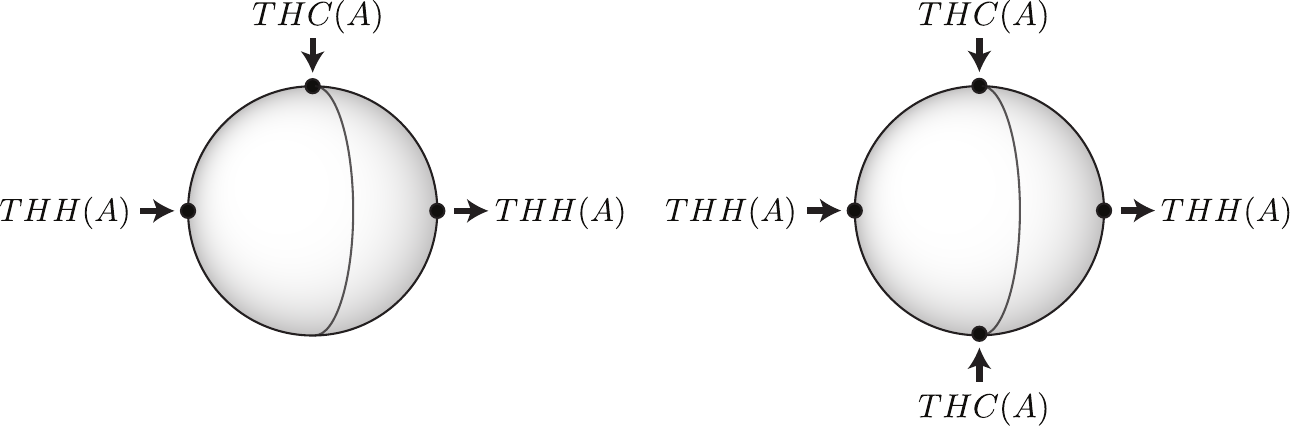}
\caption{\emph{The `geometric proof' of the noncommutative Cartier formula}. The cap product action
of $THC(A)$ on $THH(A)$ corresponds to the picture on the left. The picture on the right is the
$p$-fold cover along the central axis of the picture on the left (here $p=2$). The intuition is that
the `geometric fixed points' of the operation on the right are exactly the picture on the left. More
formally, the cyclotomic Deligne conjecture (or alternatively, in the symplectic setting,
interpreting these pictures as domains of pseudoholomorphic curves) would immediately imply Theorem
\ref{thm:absolute-bk-formula} (or Equation \eqref{eq:sh-diagram} in the symplectic setting, via
\cite{rezchikov-cyclotomic}).  Note that the picture on the right is exactly the domain defining the
equivariant pants product (which becomes the equivariant Gromov-Witten invariants computing the
quantum Steenrod operations). In the arithmetic setting, at least when $A = \F_p[x]$, this becomes
exactly the formula \eqref{eq:classical-bk-formula} of Cartier.  }
\centering
\end{figure}

Indeed, the diagram in Figure \ref{fig:bk-formula-picture-proof} was the original motivation for
this work. In an earlier work, the author constructed, for aspherical symplectic manifolds with
contact type boundary $M$ satisfying a topological condition \cite{rezchikov-cyclotomic}, a genuine
$p$-cyclotomic spectrum $SH(M, \SS)$ which lifts the symplectic cohomology $SH^*(M)$, a basic
invariant of $M$ that agrees with quantum cohomology of $M$ for compact symplectic manifolds $M$.
This cyclotomic spectrum should agree with $THH$ of a lift of the (wrapped) Fukaya category of $M$
to the sphere spectrum, whenever the latter is defined. Using the tools of
\cite{rezchikov-cyclotomic}, Figure \ref{fig:bk-formula-picture-proof} can be straightforwardly
turned into a proof of the following commutative diagram of spectra:

\begin{equation}
\label{eq:sh-diagram}
    \begin{tikzcd}
        SH(M, \SS) \sma SH(M, \SS) \ar[r, "\cup"] \ar[d, "\Delta \sma \phi"] &SH(M, \SS) \ar[d,
        "\phi"] \\
        (SH(M, \SS)^{\sma p} \sma SH(M, \SS))^{\Phi C_p} \ar[r, "(\cup^p)^{\Phi C_p}"] & SH(M,
        \SS)^{\Phi C_p}. 
    \end{tikzcd}
\end{equation}
The top map is (the spectral lift of) the product structure on symplectic cohomology; the bottom map
is the (spectral lift of) the equivariant pants product on symplectic cohomology with
$\F_p$-coefficients, which becomes the quantum Steenrod operation when $M$ is without boundary. More
precisely, the same argument proves the corresponding equation with $SH(M, \SS)$ replaced by $SH(M,
\F_p) = SH(M, \SS) \sma \F_p$ everywhere; after replacing $\Phi C_p$ by $tC_p$ one can show directly 
(\cite{rezchikov-cyclotomic-structure-and-pair-of-pants}, work in progress) that the composition of
the left and bottom maps agrees with the  equivariant pants product on $\F_p$-symplectic cohomology
\cite{seidel2014equivariant, shelukhin-zhao}.  Thus, Theorems \ref{thm:absolute-bk-formula} and
\ref{thm:relative-bk-formula} are the algebraic analog of the relation between the cyclotomic
structure and the equivariant pants product in symplectic topology, so it is no surprise that they
can be utilized to understand the meaning of the equivariant pants product in terms of more standard
structures of quantum cohomology. 

In order to avoid introducing the heavy machinery of symplectic
topology into this algebraic paper, we neither prove nor define
\eqref{eq:sh-diagram} in this paper, but delay proofs of these results
to \cite{rezchikov-cyclotomic-structure-and-pair-of-pants}.
Nonetheless, we hope this connection between equivariant symplectic
topology, noncommutative geometry, and arithmetic is motivating to the
reader of this paper. Moreover, the argument in this paper offers a
blueprint for how to establish Conjecture
\ref{conj:symplectic-p-curvature} via a closed-string proof, for those
symplectic manifolds for which the cyclotomic structure map of
\cite{rezchikov-cyclotomic} can be defined on integral quantum
cohomology. 

More broadly, we expect that the cyclotomic structure on symplectic
cohomology explains many of the arithmetic properties of symplectic
enumerative invariants, leading to a `crystalline realization' of the
`motive' of a symplectic manifold. This `crystalline realization'
should enhance the usual connection between the Gromov-Witten theory
of a symplectic manifold and the variation of Hodge structures of the
mirror, explaining various arithmetic phenomena in mirror symmetry
\cite{kontsevich-vologodsky, schwarz2008frobenius,
  hep-th/0012233}. Combining the ideas of this paper with the methods
of Floer homotopy theory, one expects an algebraic explanation of the
``$p$-curvature conjecture for quantum $K$ theory'', which has been
studied by \cite{koroteev2024quantum, bai2025quantum}, using the
variant of Theorem \ref{thm:relative-bk-formula} when $\tilde{R} =
KU$. A full development of this structure must be left to future work;
we only hope here to convince the reader that this relatively abstract
machinery has applications to concrete enumerative questions. 

\vspace{5 pt}

\paragraph{\bf Summary of paper.} Section \ref{sec:background}
contains various background about equivariant orthogonal spectra and
associated $\infty$-categorical constructions. Section
\ref{sec:topological-cap-product-proof} proves Theorem
\ref{thm:absolute-bk-formula}, and Section
\ref{sec:relative-topological-cap-product} proves Theorem
\ref{thm:relative-bk-formula}. Section
\ref{sec:relative-tate-diagonals} uses relative Tate diagonals to make
the domains of the vertical maps in the relative noncommutative
Cartier formula into classical objects related to Hochschild homology,
such that spectra are no longer needed for the statement. Section
\ref{sec:lifting-to-the-sphere} contains a lemma about automatically
lifting categories to spherical coefficients, which makes the results
of Section \ref{sec:relative-tate-diagonals} often applicable once one
throws away finitely many primes. Section \ref{sec:pvv} proves Theorem
\ref{thm:algebraic-p-curvature}. In fact Theorem
\ref{thm:algebraic-p-curvature} would follow automatically from
Theorem \ref{thm:relative-bk-formula} and Section
\ref{sec:relative-tate-diagonals} together with
Petrov-Vaintrob-Vologodsky \cite{petrov2018gauss}, but we cannot use
the latter because there is no comparison between Kaledin's
noncommutative Cartier map and corresponding spectral constructions
available in the literature. Finally, Section
\ref{sec:comparing-conventions} compares algebraic with symplectic
conventions, and Section \ref{sec:putting-it-all-together} states the
relevant assumptions on the relative Fukaya category and establishes
Theorem \ref{thm:symplectic-p-curvature-for-calabi-yau}. 

\vspace{5 pt}
\paragraph{\bf Acknowledgments.}  The author thanks Vadim Vologodsky for his many elegant
explanations regarding $THH$ and arithmetic aspects of non-commutative Hodge theory and the
Getzler connection. The author also thanks Alexander Petrov for certain clarifying remarks,
especially regarding the argument in Section \ref{sec:lifting-to-the-sphere}.
Additionally, the author thanks Andrew Blumberg for regular discussion related to this circle of ideas,
as well as for many useful references and a great deal of feedback. In particular, Blumberg pointed
out the prismatic subdivision of the simplex, which figures heavily in the proof of Theorem
\ref{thm:absolute-bk-formula}, and explained the relevant results of \cite{cyc23} to the author. 
The author thanks Zihong Chen for his hugely clarifying papers mentioned in the introduction. The
author also thanks Paul Seidel and Mohammed Abouzaid for their encouragement and general interest in
this line of research. The author thanks Dustin Clausen for advice on Lemma
\ref{lemma:perfectness-lifts}.  Finally, the author's earlier work \cite{rezchikov-cyclotomic} was
originally motivated by the natural desire to `collapse the inputs' in the quantum Steenrod
operations, but the author had the `wrong picture' in mind for several years. The author thanks
Trenitalia for the somewhat slow but peaceful and pleasant 2024 train ride from Verona to Venice, on
which the `picture proof' in Figure \ref{fig:bk-formula-picture-proof}, as well as the outline of
the arguments of this paper, was conceived. 

\section{Background and Conventions}
\label{sec:background}

In this section, we give a terse review of the homotopical background
required for this paper, with a focus on the various models of
homotopical categories and derived functors we use as well as
discussion of how we handle various models of the equivariant stable
category and its localizations. 

\subsection{Model categories} 

We will describe the equivariant stable category in terms of
presentations of homotopical categories given by model categories.
This is a convenient formalism for computing derived geometric
fixed-point functors, in particular. 

We will only need formal properties of model categories; see
\cite{hovey2007model, hirschhorn2003model} for textbook references. It
suffices to know that a model category structure on a category $\CC$
is given by specifying classes of morphisms called cofibrations,
fibrations, and weak equivalences, satisfying certain axioms.  The
model structure axioms ensure that we have control on the homotopy
category $\Ho(\CC)$, which depends only on the weak equivalences of
$\CC$.  The cofibrations and fibrations can be thought of as
scaffolding helping us calculate derived functors, notably homotopy
colimits and limits. 

We will work in the context of model categories for which cofibrant
and fibrant replacements are functorial; that is, we will be given
cofibrant and fibrant replacement functors which we denote $c$ and
$f$, respectively.  These are equipped with natural transformations
\[c \to id \to f \text{ with } cX \to X \text{ a cofibration and } X
\to fX \text{ a fibration for any } X \in \CC.\]

Writing $\emptyset$ for the initial object of $\CC$ and $\ast$ for the
terminal object, an object $X$ in $\CC$ is {\em cofibrant} if
$\emptyset \to X$ is a cofibration and {\em fibrant} if $X \to \ast$
is a fibration.  Many of the model categories we work with are
pointed, in the sense that the initial and terminal objects coincide.
An elementary argument shows that cofibrant replacement preserves
fibrant objects, and vice versa.  The model structure provides control
on the homotopy category: if $X$ is cofibrant and $Y$ is fibrant, then
the localization map $\CC(X,Y) \to \Ho(\CC)(X,Y)$ is the universal map
given by quotienting by the equivalence relation of homotopy
equivalence of morphisms.  There is a more precise version of this
statement in terms of the underlying $\infty$-category associated to
$\CC$, which we discuss below. 

Most of the model categories we consider are symmetric monoidal; this
means that there is a symmetric monoidal structure on the category
$\CC$ that is compatible with the model structure in a precise sense
(namely that the unit axiom and the pushout-product axioms are
satisfied).  One consequence of this is that under mild hypotheses the
model structure on $\CC$ lifts to a model structure on $\Alg(\CC)$,
the category of algebra objects in $\CC$.  The situation is slightly
more complicated for commutative algebra objects, but in the cases we
care about there will also be an induced model structure on
$\CAlg(\CC)$.  We give more detail on this below in the discussion on
(equivariant) commutative ring spectra. 

\subsection{$\infty$-categories and higher algebra}

Although we use model categories of equivariant spectra for technical
convenience, our results are naturally formulated at the level of
$\infty$-categories.  In this paper, we mostly work
model-agnostically, but to be concrete we will use quasicategories as
our model for the theory of $(\infty,1)$-categories unless otherwise
specified.  Associated to a model category $\CC$ is an underlying
$\infty$-category $N(\CC)[W^{-1}]$ that lifts the homotopy category
$\Ho(\CC)$.  Here $N$ denotes the nerve functor and $N(\CC)[W^{-1}]$
denotes the localization in quasicategories at the weak equivalences.
The homotopy limits and colimits in $\CC$ model limits and colimits in
$N(\CC)[W^{-1}]$, and derived functors in model categories compute
functors between $\infty$-categories. 

When $\CC$ is a symmetric monoidal model category, a variant of this
construction gives rise to a symmetric monoidal $\infty$-category.
For any symmetric monoidal $\infty$-category $\CC$, there is an
$\infty$-category $\EE_1(\CC)$ of $\EE_1$-algebra in $\CC$, and
similarly an $\infty$-category $\EE_\infty(\CC)$ of
$\EE_\infty$-algebras in $\CC$. These generalize the notion of
commutative/associative algebras to the setting of
$\infty$-categories. Similarly, given $A \in \EE_1(\CC)$ there is an
$\infty$-category of $A$-modules, which we will write either as
$\AMod$ or as $\Mod_A$. When $A \in \EE_\infty(\CC)$, $\Mod_A$ is
symmetric monoidal.  

When considering objects in $\AMod$ for $A \in \EE_\infty(\CC)$ in an
arbitrary symmetric monoidal $\infty$-category $\CC$, we will write $M
\tensor_A N$ for the symmetric monoidal structure on $\ModC{A}(\CC)$.

\subsection{Equivariant Orthogonal Spectra.}

We will work with the $G$-equivariant stable category using the model
category of orthogonal $G$-spectra.  \emph{In this paper, $G$ will
always be a closed Lie subgroup of the circle $S^1$}. We will write
$C_n \subset S^1$ for the subgroup of order $n$.  We now give a very
terse review of the theory of orthogonal $G$-spectra; we
recommend~\cite{mandell2002equivariant} and the appendices
of~\cite{hill2016nonexistence} for comprehensive treatments. 

We write $\GTop$ for the category of compactly generated weakly Hausdorff $G$-spaces and
\emph{non-equivariant} maps, and $\GTop_*$ for the category of pointed $G$-spaces and
non-equivariant maps. Each of these is enriched over itself, with the symmetric monoidal structure
given by the product and the smash product, respectively. We will write $X \mapsto X_+$ for the
functor $\GTop \to \GTop_*$ which adds a disjoint basepoint.

The homotopy theory on orthogonal $G$-spectra encodes which orbits $G/H$ are dualizable using a {\em
universe}.  A $G$-universe $\mathcal{U}$ is an orthogonal $G$-representation of countable dimension,
such that if $V \subset \mathcal{U}$ is a finite-dimensional subrepresentation then there is an
equivariant embedding $V^\infty \subset \mathcal{U}$. A $G$-universe is \emph{complete} if it
contains every irreducible $G$-representation as a subrepresentation. Let $\mathcal{U}_{0} =
\R^{\infty}$ be the trivial $G$-universe (on which the $G$-action is trivial).

To a $G$-universe $\mathcal{U}$ one associates the category $J_G(\mathcal{U})$, enriched in
$(Top_*)_G$, with: 
\begin{itemize}
    \item objects the finite-dimensional $G$-subrepresentations of $\mathcal{U}$, and 
    \item morphisms 
    \[
    J_G(\mathcal{U})(V, W) = O(V,W)^{W-V},
    \]
    where $O(V,W)$ denotes the space of isometric equivariant embeddings $i\colon V \to W$ and the
    superscript indicates the Thom space of the vector bundle over $O(V,W)$ with fiber $Im(i)^\perp$
    over $i$.  
\end{itemize}    
    
We will write $\Sp_G$ for the ($1$-)category of genuine $G$-equivariant orthogonal spectra, which we
will take to be the category $F(J_G(\mathcal{U}_0), \GTop_*)$ of enriched functors. Given any other
$G$-universe $j: \mathcal{U}' \supset \mathcal{U}_0$, there is an adjunction 
\[
j_*\colon Fun(J_G(\mathcal{U}_0), \GTop_*) \leftrightarrow Fun(J_G(\mathcal{U}'), \GTop_*)\noloc
j^*,
\] 
with the right adjoint given by restriction and the left adjoint by Kan extension.  A surprising and
useful observation is that these functors are inverse isomorphisms of point-set
categories~\cite[V.1.5]{mandell2002equivariant}.

Given a universe $U$, we define the stable equivalences on orthogonal $G$-spectra (typically
referred to as weak equivalences) as the maps $f \colon X \to Y$ that induce isomorphisms on stable
homotopy groups $\pi_*^H(-)$
\[
\pi_{q}^{H}X=\colim\limits_{V<U}\colim\limits_{n\geq \max\{0,-q\}}
\pi_{q+n} ((\Omega^{V}(X(\mathbb{R}^{n}\oplus V)))^{H}).
\]
for all proper closed subgroups $H \subseteq G$, where here $V < U$ means that $V$ is a
finite-dimensional $G$-stable subspace of $U$.

The stable model structure on $\Sp_G$ has weak equivalences the stable equivalences and fibrations
detected spacewise.  This model structure encodes the correct homotopical data: the homotopy
category for this model structure is the classical $G$-equivariant stable category, and the
underlying $\infty$-category is the initial equivariant stable category in a precise sense.  We will
use various model structures on $\Sp_G$ in our work, but all of them present the same underlying
$\infty$-category.

The category $\Sp_G$ is a closed symmetric monoidal model category with monoidal structure given by
the smash product $\sma$ and with unit given by the stable sphere $\SS$ \cite{mandell2001model}. In
particular, when $X$ and $Y$ are cofibrant, $X \sma Y$ is cofibrant as well. As always, on functor
categories, limits and colimits are computed levelwise. There is a symmetric monoidal functor
\[
\Sigma^\infty \colon \GTop_* \to \Sp_G
\]
specified by the levelwise smash product.  Note that $\Sp_G$ is tensored and cotensored over
$\GTop_*$, with the tensor given by smashing with the suspension spectrum and the cotensor by the
levelwise mapping space.  Moreover, $\Sp_G$ is enriched over itself with function spectrum
\[ 
F(Y, Z)(V) = map(Y,sh^VZ), \;sh^VZ(W) = Z(V \oplus W).
\]
There is a natural adjunction between the smash product and the function spectrum making $\Sp_G$ a
closed symmetric monoidal model category.  To derive the smash product $X \sma Y$, at least one of
$X$ and $Y$ must be cofibrant.  Similarly, function spectra $F(X,Y)$ are homotopical when $X$ is
cofibrant and $Y$ is fibrant. We observe that the functors $\Sigma^\infty$ take $G$-CW-complexes
(see \cite{adams2006prerequisites} for a fantastic introduction) to cofibrant objects. 

For every finite-dimensional $G$-representation $V$ one has the representation sphere $S^V$ given by
taking the one-point compactification of $V$.  We define the stable $V$-sphere $\SS^V =
\Sigma^\infty_+ S^V$; we will write $\Sigma^V$ to denote the functor $S^V \sma (-)$. There are also
desuspension spectra $F_V$ given by $F_V(W) = J_G(\mathcal{U}')(W,V)$ (computed in any $G$-universe
$\mathcal{U}'$ containing $V$ and $W$). We write $\SS^{-V}$ for $F_V$; there are stable equivalences
$F_V \sma S^V \to \SS$ that are functorial in $V$. The spectra $\SS^V$ and $\SS^{-V}$ are cofibrant
for any $V$. 

Fix a subgroup $H \subseteq G$.  Then there is a forgetful functor 
\[i^*_H \colon \Sp_G \to \Sp_H
\]
with left adjoint $G_+ \sma_H (-)$.  This adjunction is homotopical
and in particular is a Quillen adjunction with respect to the standard
stable model structure.

Finally, we can consider associative and commutative ring spectrum
objects in $\Sp_G$.  These categories have model structures with the
stable equivalences; the model structure on associative rings is
lifted from the stable model structure on $\Sp_G$, whereas the one on
commutative rings is lifted from a model structure called the positive
stable model structure.  We do not require the details of this here,
however.  As a matter of terminology, when discussing modules over a
commutative ring orthogonal spectrum, we will write $M \sma_A N$ for
the symmetric monoidal structure; this is defined explicitly as a
coequalizer of the two natural maps
\[
M \sma A \sma N \to M \sma N \to M \sma_A N.
\]

\subsection{Borel equivariant orthogonal $G$-spectra}

We will sometimes consider a localization of the equivariant stable
category often referred to as the {\em Borel stable category},
obtained by inverting the maps which are non-equivariant stable
equivalences (e.g., see~\cite[\S1]{BlumbergMandell2024Kunneth}).  That
is, weak equivalences in this category are those maps which induce
isomorphisms on $\pi_*^{e}(-)$.  We can model this homotopy theory by
taking a Bousfield localization of the stable model structure; we will
write this category as $\Sp_G^{\mathcal{B}}$.  Alternatively, the
underlying $\infty$-category is equivalently the functor category
$\Sp_{hG} = \Fun(BG, \Sp)$; here $BG$ denotes the $\infty$-category with a single
object and morphism space given by $G$.  We will refer to this $\infty$-category
as the category of $G$-objects in spectra, or the category of spectra with homotopy $G$-action.  
As a consequence, we can
think of the Borel stable category as the full subcategory of
$G$-spectra that are built from free cells.

There are two functors $\Sp_G \to \Sp_G^{\mathcal{B}}$ given by $X
\mapsto X \sma EG_+$ and $\Fun(EG, X)$; the first is derived by
cofibrantly replacing $X$ and the second by fibrantly replacing $X$.
The derived mapping spectra in the Borel stable category can be described as 
follows:
\[
\Map^{\Sp_G^{\mathcal{B}}}(X, Y) \htp \Map^{\Sp_G}(X \sma EG_+, Y \sma EG_+) \htp
\Map^{\Sp_G}(X \sma EG_+, Y) \htp
\Map^{\Sp_G}(X, \Fun(EG, Y)).
\]

\subsection{Fixed points functors}

One of the confusing aspects of equivariant stable homotopy theory is
that there are many different kinds of fixed points, all of which
arise in applications.  Given $X \in \Sp_{G}$ and $H \subset G$ a
closed subgroup, perhaps the most natural notion of fixed point is the
categorical fixed points:
\[
(-)^H \colon \Sp_G \to \Sp_{WH},
\]
where $WH = N_G H / H$.  These are defined by passing to $H$-fixed
points levelwise, and are corepresented by the spectrum of maps out of
the orbit $G/H$ (hence the name).  To derive the categorical fixed
points, we fibrantly replace $X$.  The stable homotopy groups are
precisely the homotopy groups of the derived categorical fixed
points:
\[
\pi_*^H(X) \cong \pi_*(X^H),
\]
and so the family of fixed point spectra detect weak equivalences.

One might naively hope that $(\Sigma^\infty X)^H \htp \Sigma^\infty
X^H$, but this is not true; the categorical fixed points of a
suspension spectrum are described by the tom Dieck splitting.  This
leads to the notion of the geometric fixed points, which we will write
as $\Phi^H$.  As a point-set matter, we take the geometric fixed point
functor to be the lax monoidal variant defined
in~\cite{mandell2002equivariant}; the monoidal structure map is an
\emph{isomorphism} when both objects are cofibrant. The geometric
fixed point functor is homotopical on cofibrant objects, and has the
following properties as a homotopical functor that characterize it up
to equivalence:

\begin{enumerate}
\item $\Phi^G (\Sigma^\infty_+X) \htp \Sigma^\infty_+ X^G$,  
\item $\Phi^G X \sma \Phi^G Y \htp \Phi^G (X \sma Y)$
\item $(\SS^{-V})^{\Phi G} \htp \SS^{-V^G}$, 
\item $\Phi^G$ commutes with homotopy colimits.
\end{enumerate}

It is sometimes useful to note that geometric fixed points can also be
described as follows.  Let $EP$ denote the $G$-space characterized by
$(EP)^G = \emptyset$ and $(EP)^H = \ast$ for all proper subgroups $H
\subseteq G$.  Define $\widetilde{EP}$ to be the cofiber of the map
$EP_+ \to S^0$ given by the collapse map.  Then $\Phi^G X \htp
(\widetilde{EP} \sma X)^G$, where here we assume $X$ is cofibrant and
that $(-)^G$ denotes the derived fixed points.  Another important fact
is that the geometric fixed points detect weak equivalences; a map $f
\colon X \to Y$ is a stable equivalence if and only if $\Phi^H X \to
\Phi^H Y$ is a stable equivalence for all $H \subseteq G$.

We now turn to two notions of fixed point that are derived functors on
the Borel stable category.  The first are the homotopy fixed points:
\[
X^{hG} := \Fun(EG,X)^G \simeq \holim_{BG} X. \]
These clearly depend only on the $G$-Borel homotopy type of $X$.  The second
are the Tate fixed points, which are defined as
\[
X^{tG} = (\widetilde{E}G \sma F(EG_+, X))^G
\]
These are related by the following commutative diagram (e.g., see page
2 of ~\cite{GreenleesMay1995Tate}).
\begin{equation}
\begin{tikzcd}
(X \sma EG_+)^G \ar[r] \ar[d] & X^G \ar[r] \ar[d] & (X \sma \widetilde{E}G)^G
\ar[d] \\
(F(EG_+, X) \sma EG_+)^G \ar[r] & F(EG_+, X)^G \ar[r] & (F(EG_+, X) \sma \widetilde{E}G)^G
\end{tikzcd}
\end{equation}
where here the vertical maps are induced by the collapse map $EG_+ \to
S^0$ and the horizontal maps by the cofiber sequence $EG_+ \to S^0 \to
\widetilde{E}G$.  The lefthand vertical map is always an equivalence
when $G$ is finite, and so the righthand square is a pullback.  In
this case the Adams isomorphism, which states that for Borel $G$-spectra, categorical $G$-fixed
points are computed by taking $G$-orbits levelwise, identifies the terms in the left column as
models for the homotopy
orbit spectrum 
\[X_{hG} = X \sma_G EG_+ \simeq \hocolim_{BG}\; X. \] 

Rewriting to identify the various terms, when $G = C_p$, this can be expressed as:
\begin{equation}
\label{eq:diagram-of-all-fixed-points}
\begin{tikzcd}
X_{hG} \ar[r] \ar[d, equals]& X^G \ar[r] \ar[d] & X^{\Phi G} \ar[d] \ar[d] \\
X_{hG} \ar[r] & X^{hG} \ar[r] & X^{tG}. \\
\end{tikzcd}
\end{equation}
All of the fixed point functors here are lax monoidal and the two
righthand horizontal arrows are monoidal
(e.g.,~\cite{blumberg2024strong, mandell2002equivariant}).  The bottom
row is the \emph{norm exact sequence}, and the map $X_{hG} \to X^{hG}$
is referred to as the \emph{norm}.
(Note that this norm is distinct from the
Hill-Hopkins-Ravenel norm.) Moreover, from an $\infty$-categorical perspective,
the bottom row is a well-defined exact triangle for any object $X \in Fun(BG, \mathcal{C})$ 
where $\CC$ is any complete and cocomplete stable $\infty$-category \cite{nikolaus-scholze}. 

\begin{remark}[A warning for those less familiar with equivariant
    homotopy theory] Given a $G$-space $X$, there is an obvious
  inclusion $X^G \to X$. However there is no way to extend this map to
  a natural transformation $\Phi^G X \to X$ for $X \in Sp_G$. The
  best one can do is the diagram of
  \eqref{eq:diagram-of-all-fixed-points}. There are two ways to
  understand this intuitively. First, if there was a map $\Phi^G X
  \to X$, then there would be a natural map $ \Phi^G \SS^{-V} \cong 
  \SS^{-V^G} \to \SS^{-V}$, i.e., a natural map $\SS^{V - V^G} \to
  \SS^0$ which (after sufficiently many suspensions) must come from a
  map of spaces from a higher-dimensional sphere to a
  lower-dimensional sphere. It is difficult to imagine what such a map
  could be! Second, thinking in terms of homological algebra, if $X$
  is a $G$-CW complex, then the map $H\Z \sma X^G \to H\Z \sma X$ is
  the map corresponding to the inclusion of the generators of CW-homology
  $C_*^{CW}(X)$ of $X$ which are fixed by the $G$-action into the
  whole complex. However, if we build $X$ by inductively taking cones
  of maps from representation spheres, so $X$ is a CW complex but not
  a $G$-CW complex, then $C_*^{CW}(X)^G$ does \emph{not} compute the
  homology of $X^G$, since the $G$-action actually acts nontrivially
  on the individual \emph{cells} of $X$. Similarly, one cannot define
  a map by requiring that  `for every cell $D^V \subset X$, the map
  takes the generator $X^G \supset (D^V)^G = D^{V^G}$ to  the
  generator corresponding to $D^V \subset X$' -- since the cells $D^V$
  are of different dimensions, and thus correspond to generators of
  different degrees!
\end{remark}

\subsection{The Hill-Hopkins-Ravenel norm}

For any finite group $G$, Hill-Hopkins-Ravenel~\cite{hhr} introduced a
multiplicative norm functor
(generalizing the Evens norm from representation theory),
\[
N_e^G \colon \Sp \to \Sp_G.
\]
Roughly speaking, this is defined by taking the smash power $X^{|G|}$
and acting by the natural $G$-action on the indexing set; the details
of the construction are slightly more complicated.  More generally,
for $H \subset G$, there is a norm functor
\[
N_H^G \colon \Sp_H \to \Sp_G
\]
defined analogously.  To derive this functor, we cofibrantly replace
$X$.  We will be focused on the case where $G$ is a finite subgroup of
$S^1$, i.e., a cyclic group.

Given a pointed space $X$, there is a diagonal map $X \to X^{\sma n}$
which is $C_n$-equivariant with respect to the trivial action on the
domain and the action which cyclically permutes coordinates on the
codomain. The \emph{HHR diagonal extends} this construction to
orthogonal spectra: there is a natural transformation
\[
X \xrightarrow{\Delta} \Phi^{C_n} N_e^{C_n}X
\]
which is in fact a point-set \emph{homeomorphism} on cofibrant objects
$X$. If $X = \Sigma^\infty_+ X_0$ then this map is induced from the
space-level diagonal by taking suspension spectra.

Composing with the canonical map $\Phi^{C_n} N_e^{C_n} X \to
(N_e^{C_n}X)^{tC_n}$ gives a map
\[
X \to(N^{C_n}X)^{tC_n}
\]
called the \emph{Tate diagonal}.

\subsection{Simplicial objects}

We now give a very terse review of simplicial objects, as our work
depends on some explicit computations in simplicial homotopy theory.
Recall that we have the simplex category $\Delta$, i.e., the category
of finite nonempty totally ordered sets, which has a skeleton
consisting of the objects $[n] = \{0 < \ldots< n\}, n \geq 0$. The
maps in this category are generated by the maps
\[ \delta^{n}_i\colon [n-1] \to [n], \text{ leave out $i$},\]
\[ \sigma^{n}_i\colon [n+1] \to [n], \sigma^n_i(j) = j \text{ if } j \leq i, \sigma^n_i(j) = j-1
\text{ else}. \]
Here $0 \leq i \leq n$ in each of these definitions.

A simplicial object in a category $\CC$ is a functor $F: \Delta^{op} \to \CC$, and a cosimplicial
object in $\CC$ is a functor $F: \Delta \to \CC$. Given a cosimplicial object we will write
$\delta^n_i, \sigma^n_i$ for $F(\delta^n_i)$ and $F(\sigma^n_i)$ respectively, when there is no
ambiguity. We denote the map dual to $\sigma: [n] \to [k]$ in $\Delta$ by $\sigma^*: [k] \to [n]$.
We also denote the maps in $\Delta^{op}$ dual to $\delta^n_i$ and $\sigma^n_i$ as
\[ d^n_i: [n] \to [n-1], s^n_i: [n] \to [n+1]\]
and for simplicial objects $F$ we will continue to write $d^n_i$ for $F(d^n_i)$, etc,  when there is
no ambiguity. We will often write $F_n$ for $F([n])$ when $F$ is a simplicial object, and $F^n$ for
$F([n])$ when $F$ is a cosimplicial object.

There is a standard cosimplicial space 
\[
\trianglespace: \Delta \to Top, \qquad \trianglespace([n]) = \Delta^n
\]
where $\Delta^n$ is the geometric standard $n$-simplex, i.e., the
convex hull of the basis vectors in $\R^{n}$.

\subsection{Geometric realizations of spectra}

Since orthogonal spectra (or more generally, $A$-modules for $A \in
Alg(\Sp_G)$) are tensored over spaces, we can define the geometric
realization of a simplicial object $X_\bullet$ in $A$-mod via an
explicit formula: 
\begin{equation}
    \label{eq:geometric-realization}
    \begin{gathered}
        |X| = \coeq\left[ \bigsqcup_{\sigma\colon [n] \to [k]} X_k \sma \Delta^n_+
        \xrightrightarrows[id \sma \sigma]{\sigma^* \sma id} \sqcup_{n} X_n \sma \Delta_+^n \right]
        \\
    = \coeq\left[ \bigsqcup_{\substack{\sigma\colon [n] \to [n\pm 1] \\ \sigma = \sigma^n_i \text{
    or } \sigma = \delta^n_i}} X_k \sma \Delta_+^n \xrightrightarrows[id \sma \sigma]{\sigma^* \sma
    id} \sqcup_{n} X_n \sma \Delta_+^n \right]
    \end{gathered}
\end{equation}

Similarly, since orthogonal spectra (or more generally, $A$-modules for $A \in Alg(OSp_G)$) are
cotensored over spaces, we can define the totalization of a cosimplicial object $X^\bullet$ in
mod-$A$:
\begin{equation}
    \label{eq:totalization}
    \Tot X = \eq\left[\prod_{n} F(\Delta^n, X^n) \xrightrightarrows[(s_n) \mapsto (\sigma \circ
    s_n)_\sigma]{(s_n) \mapsto (s_k \circ \sigma)_{\sigma}} \prod_{\sigma: [n] \to [k]} F(\Delta^n,
    X^k)\right] = \underline{Hom}_{\Delta}(\Sigma^\infty_+ \mathbf{\Delta}, X). 
\end{equation}
Here, we interpret $\underline{Hom}_\Delta$ as the internal hom in the
enriched category of functors to an enriched category.  

We see that if $F$ is a levelwise cofibrant orthogonal spectrum or
more generally a levelwise cofibrant module over a ring spectrum, then
$|F|$ is a cofibrant orthogonal ring spectrum/module.  However, in
order to have homotopical control over (co)simplicial orthogonal
$G$-spectra, we need a bit more and so we recall the notions of proper
simplicial and cosimplicial objects.  See for example~\cite[\S
  X.1]{EKMM} for a discussion in the context of spectra,
~\cite{BousfieldKan1972} for the original treatment in the language of
homotopy limits and colimits, and~\cite[\S 14]{Riehl2014} for a nice
modern exposition in the setting of model categories.

A simplicial object in a category $\CC$ enriched in spaces is {\em proper}
if for each level $n$ the inclusion map from the $n$th latching object
is a Hurewicz cofibration (i.e., satisfies the homotopy extension
property); in other words, it is cofibrant with respect to the Reedy
model structure on simplicial objects associated to the Hurewicz model
structure on $\CC$ (when the latter makes sense, as it does for
$Sp_G$).  Dually, a cosimplicial object in $\CC$ is proper if for each
level $n$ the projection map to the $n$-th matching object is a
Hurewicz fibration.  The key property of proper (co)simplicial objects
is that the geometric realization and totalization functors compute the
homotopy colimit and limit over $\Delta^{\op}$ and $\Delta$,
respectively; in particular, $|-|$ and $\Tot(-)$ preserve levelwise
weak equivalences.

\subsection{Topological Hochschild (co)homology}

We quickly review the explicit formulas for topological Hochschild
homology and cohomology we use in our work.

Given $A \in Alg(\Sp_G)$, there is a simplicial $A \sma A^{op}$-module
with $n$-simplices defined as follows:
\begin{equation}
\label{eq:two-sided-bar-complex-spectra}
\begin{gathered}
B(A,A,A)_n = A \sma A^{\sma n} \sma A \\
d^n_i = id^{\sma i} \sma m \sma id^{n-i}, \;s^{n}_i = (id)^{\sma
  i+1}\sma 1_A \sma id^{n-i},
\end{gathered}
\end{equation} 
where here $m\colon A \sma A \to A$ denotes the multiplication of $A$,
$1_A\colon \SS \to A$ is the unit, and the $A \sma A^{op}$ bimodule
structure comes from multiplication on the `outer' copies of $A$.  We
write 
\begin{equation}
B(A,A,A) = |B(A,A,A)_\bullet|
\end{equation}
for the geometric realization of this simplicial spectrum.

The map $B(A,A,A) \to A$ defined by projecting the geometric
realization to the corresponding coequalizer where $n, k \leq 1$ (or
equivalently using levelwise multiplication, regarding $A$ as the
geometric realization of the constant simplicial object on $A$) is
always a homotopy equivalence, by the usual extra degeneracy argument
(e.g., see~\cite{Riehl2014} for a modern discussion).  When $A$ is
cofibrant in a suitable model structure (see the next section for
precise statements), $B(A,A,A)$ provides a cofibrant resolution of $A$
as an $A \sma A^{\op}$-module.  We now define $THH(A)$ as the
geometric realization of the simplicial object $THH^\Delta(A)$ which
has $n$-simplices:
\begin{equation}
\label{eq:simplicial-object-thh}
THH^{\Delta}_n(A) := N^{\cyc}_n A = B(A,A,A)_n \sma_{A
  \sma A^{op}} A = A^{\sma n+1}.
\end{equation}
This is the simplicial $\SS$-module given by the levelwise tensor
product of $B(A,A,A)$ with $A$ over $A \sma A^{op}$, or, equivalently,
the tensor product of $B(A,A,A)$ with the constant simplicial $A \sma
A^{op}$ module on $A$.  Since smash products commute with colimits in
the category of $R$-modules for any orthogonal ring spectrum $R$, this
formula implies that
\[ 
THH(A) = |B(A,A,A) \sma_{A \sma A^{op}} A| \cong |THH^\Delta(A)| \cong
|B(A,A,A)|_{\sma_{A \sma A^{op}}}A.
\]
This computes $A \sma_{A \sma A^{op}}^L A$ whenever $A$ is cofibrant.

When $A$ is a commutative ring orthogonal spectrum and $R$ is an
$A$-algebra, we define $THH(R / A)$ as the geometric realization of
the simplicial object $THH^\Delta(R/A)$ specified as above by
replacing $\sma$ with $\sma_R$:
\[
THH^\Delta(R/A)_n := N^{\cyc, R} A = B_R(A, A, A)_n \sma_{A \sma_R
  A^{\op}} A = A^{\sma_R n+1}.
\]
This formula computes $A \sma_{A \sma^L_R A^{\op}} A$ whenever $A$ is
cofibrant.

We can define $THC(A)$ as the totalization of the cosimplicial
spectrum called the cyclic cobar construction 
\[
THC_\Delta(A) = C_{\cyc}^\bullet(A) \cong [k] \mapsto F(\underbrace{A
  \sma A \sma \ldots A}_k, A)
\]
obtained by applying $F_{A \sma A^{op}}(\cdot, A)$ level-wise to
$B(A,A,A)$, where here recall that $F_R(M, N)$ is the
function spectrum of $R$-module maps from $M$ to $N$.  This
description can be written as
\[
THC(A) = F_{A \sma A^{op}}(|B(A,A,A)|, A),
\]
which computes $\RHom_{A \sma A^{op}}(A, A)$ when $A$ is a
cofibrant-fibrant ring spectrum.  When $R$ is an $A$-algebra for a
commutative ring orthogonal spectrum $A$, then
\[
THC(A/R) = F_{A \sma_R A^{\op}}(|B_R(A,A,A)|, A),
\]
which computes $\RHom_{A \sma_R A^{\op}}(A,A)$ when $A$ is a
cofibrant-fibrant ring spectrum.

\subsection{A convenient model structure}

In order to maintain homotopical control throughout our arguments,
i.e., to ensure that we are computing the correct derived functors
between underlying $\infty$-categories, we find it helpful to use the
new model structures constructed in~\cite{cyc23} for this purpose.

To begin with, we recall that there is a ``standard'' model structure
on (non-equivariant) orthogonal spectra with weak equivalences the
stable equivalences, as well as a ``positive'' variant of this~\cite{mandell2001model}.  The
positive variant is necessary to obtain a model structure on
commutative ring orthogonal spectra.  One inconvenient fact about
these model structures is that the forgetful functor from commutative orthogonal ring spectra to orthogonal
spectra does \emph{not} preserve cofibrant objects.  We will
use the positive convenient $\Sigma$-model structure of~\cite[III.4.1]{cyc23}
instead, which has the same weak equivalences, and which, when lifted to a
model structure on commutative ring orthogonal spectra, does have this
property, namely that the forgetful functor back to orthogonal spectra preserves cofibrant objects.

Note that all of these model structures are symmetric monoidal model
structures with weak equivalences the stable equivalences, and as a
consequence present the symmetric monoidal $\infty$-categories of
spectra as well as the $\infty$-categories of associative and
commutative ring spectra, i.e., $A_\infty$ and $\EE_\infty$ ring
spectra.

In the equivariant setting, there are analogous ``standard'' and
positive stable model structures on orthogonal $G$-spectra, and lifts of these model structures
to associative and commutative ring orthogonal $G$-spectra.  Again,
the forgetful functor does not send cofibrant commutative ring
orthogonal $G$-spectra to cofibrant orthogonal $G$-spectra, and this has the particularly inconvenient
consequence that was previously not known whether $\Phi^G R$ for $R$
such a cofibrant commutative ring spectrum is equivalent to $\Phi^G
\widetilde{R}$ where $\widetilde{R}$ is the replacement as a cofibrant
orthogonal spectrum.  To resolve these issues, for equivariant
orthogonal spectra, we use the symmetric monoidal model structure
of~\cite[Thm.~C]{cyc23} (and see~\cite[\S III.1]{cyc23}) and its
properties.  This is an equivariant analogue of the positive
convenient $\Sigma$-model structure, and in fact the forgetful functor
takes cofibrant objects in this model structure to cofibrant objects
in the positive convenient $\Sigma$-model structure.

We summarize the properties that we require in the following
omnibus proposition, 

\begin{proposition}
\label{prop:convenient-model-structure}
There exists a symmetric monoidal model structure on the category of
orthogonal $S^1$-spectra with weak equivalences the standard stable
equivalences that is enriched over the category of non-equivariant
orthogonal spectra.  This model structure lifts to model structures on
associative ring and commutative ring objects; the model structures
here have fibrations and weak equivalences determined by the forgetful
functor.  For $R$ a commutative ring orthogonal $S^1$-spectrum, there
are induced model structures on the categories of $R$-modules,
$R$-algebras, and commutative $R$-algebras.  These model structures
have the following properties:

\begin{enumerate}

\item When $R$ is cofibrant as a (commutative) ring orthogonal
  $S^1$-spectrum and $M$ is cofibrant as an $R$-module then the
  underlying orthogonal $S^1$-spectrum of $R$ is cofibrant in the
  undercategory of $\bS$ and $M$ is cofibrant as an $\SS$-module.

\item Let $R$ be cofibrant as a commutative $\mathbb{S}$-algebra in
  the positive convenient $\Sigma$-model structure.  Then
  $THH^\Delta_\bullet(R)$ is a proper simplicial spectrum, $THH(R)$ is
  cofibrant as an $\bS$-module, and
  \[
  THH(R) \htp R \sma^L_{R \sma R^{\op}} R.
    \]
    Moreover, if $A$ is cofibrant as a (commutative)
    $R$-algebra in positive convenient $\Sigma$-model structure, then
    $THH^\Delta_\bullet(A)$ is a proper simplicial $R$-module,
    $THH(A)$ is cofibrant as an $R$-module, 
    \[
    THH(A)\htp A \sma^L_{A \sma A^{\op}} A,
    \]
    $THH(A/R)$ is cofibrant as an
    $R$-module, and
    \[
    THH(A/R) \htp A \sma^L_{A \sma_R A^{\op}} A.
    \]

  \item Let $R$ be a cofibrant-fibrant commutative
    $\mathbb{S}$-algebra in the positive convenient $\Sigma$-model
    structure.  Then the two-sided bar complex $B(R,R,R)$ is a
    cofibrant replacement of $R$ as an $R \sma R^{\op}$-module and so
    \[
    THC(R) \htp \RHom_{R \sma R^{\op}}(R, R).
    \]
    Let $A$ additionally be a cofibrant-fibrant (commutative)
    $R$-algebra in the positive convenient model structure; then
    $B_R(A,A,A)$ is a cofibrant replacement of $A$ as an $A \sma_R
    A^{\op}$-module and so
    \[
    THC(A/R) \htp \RHom_{A \sma_R A^{\op}}(A,A).
    \]

\item Then for any finite group $G$, let $R$ be a cofibrant
  commutative ring orthogonal $G$-spectrum, and let $\widetilde{R} \to
  R$ denote a cofibrant replacement of $R$ in the model category of 
  $\bS$-modules.  The induced map
  \[
  \Phi^{G} \widetilde{R} \to \Phi^G R
  \]
  is a weak equivalence.  Moreover, if $R$ is a cofibrant commutative
  ring orthogonal $S^1$-spectrum and $\widetilde{R} \to R$ denotes a
  cofibrant replacement in the model category of orthogonal
  $S^1$-spectra, the induced map
  \[
  \Phi^{C_p} \widetilde{R} \to \Phi^{C_p} R
  \]
  is a weak equivalence of $S^1 / C_p$-spectra.
  
\item When $M$ is a cofibrant orthogonal $G$-spectrum or cofibrant
  (commutative) ring orthogonal $G$-spectrum, the map
  \[
  \Phi^G M \sma \Phi^G N \to \Phi^G (M \sma N)
  \]
  is an isomorphism.

\item Let $G$ be a finite group.  When $R$ is a cofibrant commutative
  ring orthogonal $G$-spectrum and $M$ is a cofibrant $R$-module, the
  $R$-relative norm ${}_RN^G_eM$ is a cofibrant $R$-module.

\item Let $G$ be a finite group.  When $M$ is a cofibrant orthogonal
  spectrum or a cofibrant (commutative) ring orthogonal spectrum, the
  Hill-Hopkins-Ravenel diagonal map $M \to \Phi^G N_e^G M$ is an
  isomorphism.

\end{enumerate}
\end{proposition}

\begin{proof}
Property (1) is a basic property of the model structure
in~\cite[Thm.~C]{cyc23}.  The claims of properties (2) and (3) follow
because these are all symmetric monoidal model structures along with
the fact that the unit map for cofibrant (commutative) ring orthogonal
spectra is a Hurewicz cofibration (since cofibrant objects
in (commutative) rings forget to cofibrant objects in the model category of
orthogonal spectra under $\bS$) and the observation that the
forgetful functor from fibrant (commutative) ring orthogonal spectra
preserves fibrant objects.  Using the fact that smash products (and
relative smash products) preserve Hurewicz cofibrations, an inductive
argument shows that inclusions of the latching objects are Hurewicz
cofibrations.  For $THC$, we use the fact that the derived mapping
spectrum can be computed using cofibrant replacement in the first
variable and fibrant replacement in the second.  Property (4) is
another basic property of the model structure we are working with;
see~\cite[Thm.~C (iii)]{cyc23}.  Property (5)
is~\cite[II.2.7]{cyc23}.  Property (6) reduces to showing that $N^G_e
M$ is cofibrant as an $\bS$-module for a cofibrant $\bS$-module $M$,
which follows from the same line of argument as~\cite[B.104]{hhr}.
Finally, property (7) follows from~\cite[VI.7.9]{cyc23}.
\end{proof}

\subsection{Some coincidences over $\F_p$}

Here we record some facts about the Tate spectrum when we work in the $\infty$-category of $\Z$ and
$\F_p$-module spectra with (homotopy)
$S^1$-action.

\begin{lemma}
\label{lemma:tS^1-tC_p-Z-mod}
For any $M \in S^1-\ZMod$, the map $M^{tS^1} \to M^{tC_p}$ factors through $(M/p)^{tS^1}$ and induces
an equivalence
\[ M^{tC_p} = (M/p)^{tS^1}.
\]
Here we write $M/p$ for the cone of the map $M \xrightarrow{p} M$,
i.e. $\F_p/p \simeq \F_p\oplus \F_p[1].$ 
\end{lemma}

\begin{proof}
This is the third equation of Lemma IV.4.12 of
\cite{nikolaus-scholze}, together with the fact that $\Z^{tC_p} =
\Z^{tS^1}/p = (\Z/p)^{tS^1}$. 
\end{proof}

\begin{lemma}
\label{lemma:tS^1-tC_p-F_p-mod}
For any $M \in S^1-\FpMod$, 
\[
M^{tC_p} = M^{tS^1} \tensor_{\F_p^{tS^1}} \F_p^{tC_p} = M^{tS^1}
\langle \theta \rangle.
\]
\end{lemma}

\begin{proof}
An equivalent statement is proven in~\cite{chen2024operadic}. However, we note that this follows
immediately from Lemma~\ref{lemma:tS^1-tC_p-Z-mod}.
\end{proof}

\subsection{Tate constructions on some infinite spectra}

In this section we recall some useful computational tools.  The
following lemma~\cite[15.2]{Adams74}, which we learned from Tyler
Lawson, is helpful for computing homology groups of inverse limits of
spectra.  (See also~\cite[A.5.13]{Ravenel1992Nilpotence}
and~\cite{LuckReichVarisco2003} for similar results.)

\begin{lemma}
\label{lemma:lawson-lemma}
Let $A$ be a spectrum of finite type, i.e., $A$ is connective and has
a $CW$-model with finitely many cells in each degree.  Let $\ldots \to
B_i \to B_{i-1} \to \ldots$ be a tower of spectra such that there is a natural number $N$ and an integer $M$ such that $\pi_k(B_i)
= 0$ for all $i > N$ and all $k  < M$. Then  
\[
A \sma \varprojlim B_i \to \varprojlim A \sma B_i
\]
is an equivalence.
\end{lemma}

We also recall the following theorem:
\begin{proposition}[Burklund \cite{burklund2024note}]
\label{prop:burklund}
Let $X$ be a $p$-complete spectrum bounded below. Let $I$ be the augmentation ideal of the mod-$p$
Steenrod algebra $\mathcal{A}$, which acts on the $\F_p$-homology of any spectrum. If the homology
of $X$ is $I$-complete, i.e. the map $H_*(X, \F_p) \to \lim_n H_*(X, \F_p)/I_n$ is an isomorphism,
then the canonical map 
\[ X \to X^{tC_p}\]
is an equivalence.
\end{proposition}

\section{The topological cap product}

The purpose of this section is to give the construction of the cap product
\begin{equation}
THC(A) \sma THH(A) \to THH(A)
\end{equation}
that we will use to prove our compatibility formulas.  First, observe
that if we take $A$ to be a cofibrant-fibrant associative ring
spectrum and $\widetilde{A}$ to be a cofibrant-fibrant replacement of
$A$ as an $A \sma A^{\op}$-module, then the formula 

\begin{equation}\label{eq:derived-cap}
F_{A \sma A^{\op}}(\widetilde{A},\widetilde{A}) \sma (\widetilde{A} \sma_{A \sma A^{\op}} A) \to
(F_{A \sma A^{\op}}(\widetilde{A},\widetilde{A}) \sma \widetilde{A}) \sma_{A \sma A^{\op}} A
\to
\widetilde{A} \sma_{A \sma A^{\op}} A, 
\end{equation}

where the second map is the evaluation map, represents the derived cap product map

\begin{equation}
THC(A) \sma^L THH(A) \to THH(A).
\end{equation}

On homotopy groups, this is the map

\begin{equation}
THC^{p}(A) \otimes THH_q(A) \to THH_{q-p}(A), \text{ where }
THC^{p}(A) = \pi_{-p}(THC(A)). 
\end{equation}

However, although the construction given in
equation~\eqref{eq:derived-cap} gives the correct map in the
$\infty$-category of spectra, it is not felicitous for analyzing the
interaction with the cyclotomic structure on $THH(A)$.  For this
purpose, we use a simplicial model.

In order to give a simplicial description of the cap product, we give a description of pairings of
cosimplicial and simplicial objects.  Our treatment uses the prismatic subdivision approach
pioneered by McClure-Smith in their proof of the Deligne conjecture.  We claim no particular novelty
for this approach; see~\cite[\S4.3.1]{malm2010string} for an exposition. 

\begin{definition}
Given a cosimplicial $G$-spectrum $X^\bullet$ and simplicial $G$-spectra $Y_\bullet$ and
$Z_\bullet$, a {\em cap pairing} of simplicial $G$-spectra is specified by maps
\begin{equation}
\cap_{p,q} \colon X^p \sma Y_{q} \to Z_{q-p}
\end{equation}
satisfying the following requirements for compatibility with the simplicial and cosimplicial structure
maps:

\begin{align}
\cap_{p,q} (d^i \sma \id) &= \cap_{p-1,q} (\id \sma d_i) \qquad 0 \leq i < p
\\
\cap_{p,q} (d^p \sma \id) &= d_0 \circ \cap_{p-1, q+1} (\id \sma \id)
\\
\cap_{p-1, q} (\id \sma d_{p+i}) &= d_{i+1} \circ \cap_{p-1, q+1} (\id \sma \id) \\
\label{eq:degeneracy-supernatural}
\cap_{p,q} (s^i \sma \id) &= \cap_{p+1, q} (\id \sma s_i) \\
\cap_{p,q} (\id \sma s_{p+i}) &= s_i \circ \cap_{p+1,q-1} (\id \sma \id).
\end{align}
A map of cap pairings $\cap_{p,q} \to \widetilde{\cap}_{p,q}$ is a family of morphisms $X^\bullet
\to \widetilde{X}^\bullet$, $Y_\bullet \to \widetilde{Y}_\bullet$, and $Z_\bullet \to
\widetilde{Z}_\bullet$ that commute with the pairings.
\end{definition}

The main theorem we require is the following.

\begin{theorem}\label{thm:generalized-cap}
Given a cap pairing $\cap \colon X^\bullet \sma Y_\bullet \to Z_\bullet$, then for each $0 < u < 1$
there is an induced map
\begin{equation}
\cap_u \colon \Tot(X^\bullet) \sma |Y_\bullet| \to |Z_\bullet|
\end{equation}
of $G$-spectra, and the maps vary continuously in $u$.

A morphism of cap pairings $\cap_{p,q} \to \widetilde{\cap}_{p,q}$ gives rise to a commutative
diagram
\begin{equation}
\begin{tikzcd}
\Tot(X^\bullet) \sma |Y_\bullet| \ar[r,"\cap_u"] \ar[d] & \relax|Z_\bullet\relax| \ar[d] \\
\Tot(X^\bullet) \sma |Y_\bullet| \ar[r,"\widetilde{\cap}_u"] & \relax|Z_\bullet\relax|.
\end{tikzcd}
\end{equation}
\end{theorem}
This is proven in Appendix \ref{sec:prismatic-subdivision}.

Therefore, to produce a point-set cap pairing of $THC$ and $THH$, it suffices to produce a suitable
family of maps
\begin{equation}\label{eq:explicit-thh-cap}
C^p_{\cyc}(A) \sma N^{\cyc}_q(A) \to N^{\cyc}_{q-p}(A).
\end{equation}
Writing this out explicitly, these are maps
\begin{equation}
F(\underbrace{A \sma A \sma \ldots \sma A}_p, A) \sma (\underbrace{A \sma A \sma \ldots \sma
A}_{q+1}) \to \underbrace{A \sma A \sma \ldots A}_{(q-p)+1}.
\end{equation}

Motivated by the definition of the cap product above, we will use the pairing given as
\begin{align}
&F(\underbrace{A \sma A \sma \ldots \sma A}_p, A) \sma (\underbrace{A \sma A \sma \ldots \sma
A}_{q+1}) \cong \\ 
( &F(\underbrace{A \sma A \sma \ldots \sma A}_p, A) \sma A \sma (\underbrace{A \sma A \sma \ldots
\sma A}_p) 
\sma \underbrace{A \sma A \sma \ldots \sma A}_{(q-p)} \to \\
&A \sma A \sma (\underbrace{A \sma A \sma \ldots \sma A}_{q-p}) \to
A \sma \underbrace{A \sma A \sma \ldots \sma A}_{q-p} \to 
(\underbrace{A \sma A \sma \ldots \sma A}_{(q-p)+1}),
\end{align}
where the middle arrow is given by evaluation and the bottom arrow by multiplication on the first
coordinate. The following result is elementary:

\begin{proposition}
The cap product maps in equation~\eqref{eq:explicit-thh-cap} specify a cap pairing of simplicial
spectra.
\end{proposition}

Thus, we can immediately conclude that we have a pairing 
\begin{equation}
\Tot(C^\bullet_{\cyc} A) \sma |N^{\cyc}_\bullet A| \to |N^{\cyc}_\bullet A|
\end{equation}
of spectra by Theorem~\ref{thm:generalized-cap}.  Choosing $A$ to be a cofibrant-fibrant ring
spectrum, this gives rise to the desired cap pairing
\begin{equation}\label{eq:derived-cap-2}
THC(A) \sma^{L} THH(A) \to THH(A).
\end{equation}

Working with the resolution $\tilde{A} = B(A,A,A)$ described earlier gives the following
compatibility statement, proven in Appendix \ref{sec:ordinal-sums-etc}.

\begin{lemma}
\label{lemma:cap-product-two-models-comparison}
When $A$ is a cofibrant-fibrant ring spectrum, the map induced by the cap product pairing of
equation~\eqref{eq:derived-cap-2} above is homotopic to the cap product map of
equation~\eqref{eq:derived-cap}, in the sense that under the canonical comparisons between the
domains and codomains in the homotopy category, the two maps agree.
\end{lemma}

Next, we describe how to use the pairing technology above to give rise to a $p$-fold cap product
map.  We first briefly review the edgewise subdivision of a simplicial object.  For each $r \geq 1$
and simplicial object $X_\bullet$ in some category $\mathcal{C}$, we have
\begin{equation}
(\sd_r X)_n = X_{(n+1)r - 1},
\end{equation}
and the faces and degeneracies are specified by the formulas 
\begin{align}
d_i &= d_i \circ d_{i+(n+1)} \circ \ldots \circ d_{i + (r-1)(n+1)} \\
s_i &= s_{i+(r-1)(n+2)} \circ \ldots \circ s_{i+(n+2)} \circ s_i. 
\end{align}
For example,
\begin{equation}
(\sd_p N^{\cyc}(A))_n = N^{\cyc}_{(n+1)p - 1}(A) =
\underbrace{A \sma A \sma \ldots \sma A}_{(n+1)p}.
\end{equation}

Now, when $X_\bullet$ is the simplicial object underlying a \emph{cyclic} object in $\mathcal{C}$,
as is, for example, $N^{cyc}_\bullet(A)$, $\sd_rX$ is now a simplicial object in the category of
functors from $BC_r \to \mathcal{C}$, where $BC_r$ is the category with one object and endomorphisms
given by $\Z/r\Z$. In particular, when $A$ is an associative orthogonal ring spectrum,
we have levelwise equivalences of $C_p$-spectra

\begin{equation}
    \label{eq:levelwise-norm-equivalence}
    N_e^{C_p} A^{\sma n} \cong (\sd_p N^{\cyc} A)_{n-1}. 
\end{equation}

\begin{remark} The edgewise subdivision makes sense for an arbitrary simplicial object, but only for
cyclic object does the edgewise subdivision naturally produce simplicial objects in the category of
objects with a $C_r$-action.	
\end{remark}

Now, for $X^\bullet$ and $Y^\bullet$ cosimplicial spectra, there is a natural lax monoidal structure
map 
\begin{equation}
\alpha: \Tot(X^\bullet) \sma \Tot(Y^\bullet) \to \Tot(X^\bullet \sma Y^\bullet).
\end{equation}
\begin{remark} \label{rk:lax-monoidal-structure-map}
We briefly recall the definition of the lax monoidal structure map: the diagonal functor $\delta:
\Delta \to \Delta \times \Delta$ is covered by the diagonal map of cosimplicial objects
$\bar{\delta}: \mathbf{\Delta} \to \delta^* \mathbf{\Delta} \times \mathbf{\Delta}$ defined via the
levelwise diagonal on spaces. Writing $\bar{\sma}$ for the functor taking a pair of cosimplicial
objects to a bicosimplicial object (a covariant functor from $\Delta \times \Delta$) via levelwise
smash product, we define $\alpha$ as the composition 
\[\underline{Hom}_{\Delta}(\mathbf{\Delta}, X) \sma \underline{Hom}_{\Delta}(\mathbf{\Delta}, Y) \to
\underline{Hom}_{\Delta \times \Delta}(\mathbf{\Delta} \times \mathbf{\Delta}, X \bar{\sma} Y) \to
\underline{Hom}_{\Delta}(\delta^*(\mathbf{\Delta} \times \mathbf{\Delta}), \delta^* X
\bar{\sma} Y) \to \underline{Hom}_{\Delta}(\mathbf{\Delta}, X\sma Y) \]
where the second arrow is pullback along $\delta$, the third is precomposition with $\bar{\delta}$,
and we use that $\delta^* X \bar{\sma} Y = X\sma Y$.
\end{remark}
Keeping track of equivariance, we analogously have a natural map of $C_p$-spectra
\begin{equation}
N_e^{C_p} \Tot(X^\bullet) \to \Tot(N_e^{C_p} X^\bullet),
\end{equation}
where $N_e^{C_p} X^\bullet$ denotes the levelwise norm.

Thus, by Theorem~\ref{thm:generalized-cap}, to specify a map
\begin{equation}
N_e^{C_p} \Tot(X^\bullet) \sma |Y_\bullet| \to |Z_\bullet|,
\end{equation}
it suffices to produce a cap pairing of genuine $C_p$-spectra
\begin{equation}
N_e^{C_p} X^m \sma Y_n \to Z_{n-m}.
\end{equation}

We now define the $p$-fold cap product:
\begin{equation}
\label{eq:cap-pairing-defining-spectral-cap-p}
\cap^p_{m, n}: (\underbrace{C^m_{\cyc}(A) \sma C^m_{\cyc}(A) \sma \ldots \sma C^m_{\cyc}(A)}_p) \sma
(\sd_p N^{\cyc} (A))_n \to (\sd_p N^{\cyc} (A))_{n-m}.
\end{equation}
Writing this out, this is given by the maps
\begin{equation}
\label{eq:p-fold-cap-product-defining-maps}
(F(\underbrace{A \sma A \sma \ldots \sma A}_m, A))^{\sma p} \sma (\underbrace{A \sma A \ldots \sma
A}_{n+1})^{\sma p} \to 
(\underbrace{A \sma A \ldots \sma A}_{(n-m)+1})^{\sma p}
\end{equation}
specified as the composite
\begin{align}
(&F(\underbrace{A \sma A \sma \ldots \sma A}_m, A))^{\sma p} \sma (\underbrace{A \sma A \ldots \sma
A}_{n+1})^{\sma p} \\ &\cong (F(\underbrace{A \sma A \sma \ldots \sma A}_m, A) \sma (A \sma
(\underbrace{A \sma A \ldots \sma A}_{m}) \sma 
(\underbrace{A \sma A \ldots \sma A}_{(n-m)}))^{\sma p} \\
&\to ((A \sma A) \sma (\underbrace{A \sma A \sma \ldots \sma A}_{n-m}))^{\sma p} \to
(\underbrace{A \sma A \sma \ldots \sma A}_{n-m+1})^{\sma p}
\end{align}
where the first map shuffles the terms together, the second evaluates, and the last multiplies in
the first factor.

Thus, this gives rise to a map
\begin{equation}
\label{eq:first-def-of-spectral-cap-p}
\cap^p \colon N_e^{C_p} \Tot(C^\bullet_{\cyc} A) \sma |\sd_p N_\bullet^{\cyc}(A)| \to |\sd_p
N_\bullet^{\cyc}(A)|,
\end{equation}
which for cofibrant-fibrant ring spectra $A$ can be written (ignoring the equivariant structure) as
\begin{equation}
\cap^p \colon N_e^{C_p} THC(A) \sma^L THH(A) \to THH(A),
\end{equation}
using the isomorphism $THH(A) \cong |\sd_p N_\bullet^{\cyc}(A)|$, the proof of which is reviewed in
Appendix \ref{sec:ordinal-sums-etc}.

\section{The topological cap product and the cyclotomic structure}
\label{sec:topological-cap-product-proof}
The purpose of this section is to prove Theorem~\ref{thm:absolute-bk-formula} of the introduction,
namely that the square
\begin{equation}
\label{eq:topological-cap-product-thm-goal}
\begin{tikzcd}
THC(A) \sma THH(A) \ar[rr, "\cap"] \ar[d,"\Delta \sma \phi", swap] && \ar[d,"\phi"] THH(A) \\
\Phi^{C_p} N_e^{C_p} THC(A) \sma \Phi^{C_p} THH(A) \ar[r] & \Phi^{C_p} (N_e^{C_p} THC(A) \sma
THH(A)) \ar[r, "\Phi^{C_p} \cap^p", swap] & \Phi^{C_p} THH(A)
\end{tikzcd}
\end{equation}
commutes.

The proof is straightforward, given the construction of the cap product maps from the previous
section and the definition of the cyclotomic structure maps in terms of the Hill-Hopkins-Ravenel
norm as explained in~\cite{angeltveit2018topological}.  Specifically, the cyclotomic structure map
is the geometric realization of the levelwise diagonal maps
\begin{equation}
A^{\sma n} \to \Phi^{C_p} N_e^{C_p} A^{\sma n},
\end{equation}
using the equivariant equivalences \eqref{eq:levelwise-norm-equivalence}.

Thus, to prove our theorem, it suffices to show that the diagram
\begin{equation}
\begin{tikzcd}
THC(A) \sma THH(A) \ar[rr, "\cap"] \ar[d,"\Delta \sma \phi", swap] && \ar[d,"\phi"] THH(A) \\
\Phi^{C_p} N_e^{C_p} THC(A) \sma \Phi^{C_p} |\sd_p N^{\cyc}_\bullet A| \ar[r] & \Phi^{C_p}
(N_e^{C_p} THC(A) \sma |\sd_p N^{\cyc}_\bullet A|) \ar[r, "\Phi^{C_p} \cap^p", swap] & \Phi^{C_p}
|\sd_p N^{\cyc}_\bullet A|
\end{tikzcd}
\end{equation}
commutes.

Rewriting, this is the diagram
\begin{equation}
\begin{tikzcd}
\Tot(C^\bullet_\cyc A) \sma |N^{\cyc}_\bullet A| \ar[rr, "\cap"] \ar[d,"\Delta \sma \phi", swap] &&
\ar[d,"\phi"] |N^{\cyc}_\bullet A| \\
\Phi^{C_p} N_e^{C_p} \Tot(C^\bullet_\cyc A) \sma \Phi^{C_p} |\sd_p N^{\cyc}_\bullet A| \ar[r] &
\Phi^{C_p} (N_e^{C_p} \Tot(C^\bullet_\cyc A) \sma |\sd_p N^{\cyc}_\bullet A|) \ar[r, "\Phi^{C_p}
\cap^p", swap] & \Phi^{C_p} |\sd_p N^{\cyc}_\bullet A|.
\end{tikzcd}
\end{equation}

\iffalse
Since the diagram
\begin{equation}
\begin{tikzcd}
\Tot(C^\bullet_\cyc A) \sma |N^{\cyc}_\bullet A| \ar[d] \ar[dr] & \\
\Phi^{C_p}_e N_e^{C_p} \Tot(C^\bullet_\cyc A) \sma \Phi^{C_p}|\sd_p N^{\cyc}_\bullet A| \ar[r] &
 \Tot(\Phi^{C_p}_e N_e^{C_p} C^\bullet_\cyc A) \sma |\Phi^{C_p} \sd_p N^{\cyc}_\bullet A|
\end{tikzcd}
\end{equation}
commutes, this is equivalent to showing that the diagram
\begin{equation}
\begin{tikzcd}
Tot(C^\bullet_\cyc A) \sma |N^{\cyc}_\bullet A| \ar[r, "\cap"] \ar[d,"\Delta \sma \phi", swap] &
\ar[d,"\phi"] |N^{\cyc}_\bullet A| \\
\Tot(\Phi^{C_p} N_e^{C_p} C^\bullet_\cyc A) \sma |\Phi^{C_p} \sd_p N^{\cyc}_\bullet A| \ar[r] &
\relax|\Phi^{C_p} \sd_p N^{\cyc}_\bullet A\relax|
\end{tikzcd}
\end{equation}
commutes.
\fi

Now, the following diagram
\begin{equation}
\begin{tikzcd}
F(A^{\sma m}, A) \sma A^{\sma (n+1)} \ar[rr, "\cap"] \ar[d,"\Delta \sma \Delta", swap] &&
\ar[d,"\Delta"] A^{\sma (n-m+1)} \\
\Phi^{C_p} N_e^{C_p} F(A^{\sma m}, A) \sma \Phi^{C_p} A^{\sma p(n+1)} \ar[r] & \Phi^{C_p} (N_e^{C_p}
F(A^{\sma m}, A) \sma A^{p(n+1)}) \ar[r, "\Phi^{C_p} \cap^p", swap] & \Phi^{C_p} A^{\sma p(n-m+1)}
\end{tikzcd}
\end{equation}
is a map of pairings and commutes essentially by construction.

This implies that the diagram 
\begin{equation}
\label{eq:what-we-proved}
\begin{tikzcd}
\Tot(C^\bullet_{\cyc} A) \sma \relax|N^\cyc_\bullet A\relax| \ar[r] \ar[d] & \relax|N^\cyc_\bullet
A\relax| \ar[d] \\
\Tot(\Phi^{C_p} N_e^{C_p} C^\bullet_{\cyc} A) \sma \Phi^{C_p} |\sd_p N^\cyc_\bullet A| \ar[r] &
\Phi^{C_p}|\sd_p N^{\cyc}_\bullet A|
\end{tikzcd}
\end{equation}
commutes.

Now there is a factorization
\begin{equation}
    \label{eq:factorization-of-p-fold-cap-products}
    \begin{tikzcd}[column sep=small]
        \Phi^{C_p} N_e^{C_p} \Tot(C^\bullet_{\cyc} A) \sma \Phi^{C_p} |\sd_p N_\bullet^{\cyc} A| 
        \ar[r] \ar[rd]&\Tot(\Phi^{C_p} N_e^{C_p} C^\bullet_{\cyc}A) \sma \Phi^{C_p} |\sd_p
        N^{\cyc}_\bullet A| \ar[r] &  \Phi^{C_p} |\sd_p N^{\cyc}_\bullet A| \\ &
        \Phi^{C_p} \Tot(N_e^{C_p} C^\bullet_{\cyc}) \sma \Phi^{C_p} |\sd_p N^{\cyc}_\bullet A|
        \ar[ur] \ar[u] &
    \end{tikzcd}
\end{equation}
where we apply functors levelwise to cosimplicial objects inside totalizations, the composition
going down to the right and then up to the right is $\Phi^{C_p}\cap^p$,  the map from top middle to
the right is the map in the bottom of \eqref{eq:what-we-proved}, and the vertical map is assembled
from the canonical maps $\Phi^{C_p}F(\Delta^n, N_e^{C_p} C^n_{\cyc} A)  \to F(\Delta^n,
\Phi^{C_p}N_e^{C_p} C^\bullet_{\cyc} A)$ levelwise in the totalization. This proves that
\eqref{eq:topological-cap-product-thm-goal} always commutes on the point-set level.

Finally, in this discussion so far, we have tacitly assumed that we are computing the derived
functors.  To ensure that this is correct, we assume at the outset that $A$ is a cofibrant-fibrant
associative ring spectrum and we really work with the composite
\begin{equation}
cF(A^{m},A) \sma A^{\sma (n+1)} \to F(A^{m}, A) \sma A^{\sma (n+1)} \to A^{\sma (n-m+1)}
\end{equation}
as our model of $\cap$,
where $c$ denotes the cofibrant replacement functor, and analogously for $\cap^p$.
Then the diagram in question becomes
\begin{equation}
\label{eq:map-of-cap-products}
\begin{tikzcd}
cF(A^{\sma m}, A) \sma A^{\sma (n+1)} \ar[rr, "\cap"] \ar[d,"\Delta \sma \phi", swap] &&
\ar[d,"\phi"] A^{\sma (n-m+1)} \\
\Phi^{C_p} N_e^{C_p} cF(A^{\sma m}, A) \sma \Phi^{C_p} A^{\sma p(n+1)} \ar[r] & \Phi^{C_p}
(N_e^{C_p} cF(A^{\sma m}, A) \sma A^{p(n+1)}) \ar[r, "\Phi^{C_p} \cap^p", swap] & \Phi^{C_p} A^{\sma
p(n-m+1)}
\end{tikzcd}
\end{equation}
We now have the following facts that guarantee homotopical control.
Since $A$ is cofibrant-fibrant in the model structure of Proposition
\ref{prop:convenient-model-structure}:
\begin{enumerate}
\item $A^{k}$ computes the derived smash product, 
\item $F(A^{m}, A)$ computes the derived mapping space, 
\item $\Tot C^\bullet_\cyc A \htp THC(A)$ and $C^\bullet_\cyc A$ is a proper cosimplicial object,
\item $|N^\cyc_\bullet A| \htp THH(A)$ and $N^\cyc_\bullet A$ is a proper simplicial object,
\item $\Phi^{C_p} A^{pk}$ computes the derived functor and is cofibrant,
\item and $A \to \Phi^{C_p} N_e^{C_p} A$ is an isomorphism.
\end{enumerate}
Since $cF(A^m, A)$ is cofibrant: 
\begin{enumerate}
\item $cF(A^m, A) \to \Phi^{C_p} N_e^{C_p} cF(A^m, A)$ is an isomorphism and
\item $N_e^{C_p} cF(A^m, A)$ computes the derived functor.
\end{enumerate}
We also know that the smash products on the far left are the derived smash product.  Finally, we
have the following easy observation, which we apply to $X^\bullet = cC^\bullet_\cyc A$.

\begin{lemma}\label{lem:totdiag}
Let $X^\bullet$ be a cosimplicial spectrum such that each $X^k$ is cofibrant-fibrant.  Then the
zigzag
\begin{equation}
\begin{tikzcd}
\Phi^{C_p} N_e^{C_p} c\Tot(X^\bullet) & c\Tot(X^\bullet) \ar[l,"\cong",swap] \ar[r,"\htp"] &
\Tot(X^\bullet) \ar[r,"\cong"] & \Tot (\Phi^{C_p} N_e^{C_p} X^\bullet) & c\Tot (\Phi^{C_p} N_e^{C_p}
X^\bullet) \ar[l,"\htp"]
\end{tikzcd}
\end{equation}
is a weak equivalence.
\end{lemma}

Since $THH(A)$ is cofibrant, $THC(A) \sma THH(A)$ computes the derived smash product, we can compute
the action of derived geometric fixed points on the morphism $\cap^p$ by finding a cofibrant
replacement for the domain, for example, via the map $c\sma \id$ which cofibrantly replaces the
$THC(A)$ factor.

In summary, we have homotopical control and can conclude Theorem~\ref{thm:absolute-bk-formula} from
our point-set argument. 
\begin{remark}
In fact, we have proven more, namely that maps like $\cap$ and $\Phi^{C_p}\cap^p$ defined via the
construction above commute with weak equivalences of fibrant-cofibrant algebras $A$, and thus are
well defined in the $\infty$-categorical sense.	
\end{remark}

\section{The relative topological cap product}
\label{sec:relative-topological-cap-product}
We now suppose that $A$ is an associative $R$-algebra, for a commutative ring spectrum $R$. 
 In this setting, $THH(A)$ becomes a $THH(R)$-module using the levelwise multiplication:
 \begin{equation}
 |N^{\cyc}_\bullet R| \sma |N^{\cyc}_\bullet A| \to |N^{\cyc}_\bullet (R \sma A)| \to
 |N^{\cyc}_\bullet A|,
 \end{equation}
 where the first map involves the simplicial shuffles (e.g., see~\cite[4.2]{loday2013cyclic} for an
 exposition of formulas for this map).

 Moreover, we have relative constructions $THC(A/R)$ and $THH(A/R)$, which we can compute under
 suitable cofibrancy and fibrancy hypotheses as
 \begin{equation}
 THC(A/R) \htp \Tot(C^\bullet_{\cyc,R} A) \quad\textrm{and}\quad |N^{\cyc,R}_\bullet A|,
 \end{equation}
 where $N^{\cyc, R}_k A = A^{\sma_R (k+1)}$ and $C^k_{\cyc, R} = F(A^{\sma_R k},A)$.  These both
 also inherit natural $R$-module structures.

The same argument as for the absolute cap product gives rise to a
$R$-linear relative cap product,
\begin{equation}
 \cap_R \colon THC(A/R) \sma_R THH(A/R) \to THH(A/R),
 \end{equation}
 since the proof of Theorem \ref{thm:generalized-cap} producing maps
 from pairings adapts immediately to the $R$-linear setting.  When $R$
 is a {\em cyclotomic base}, then $THH(A/R)$ has a cyclotomic
 structure and we can consider an analogue of our result computing the
 interaction of the cyclotomic structure and the cap product map, i.e.
 Theorem \ref{thm:relative-bk-formula-cyclotomic-base}; the proof is
 \emph{identical} given the convenient model structure of
 \cite{cyc23}.
 
However, in many cases of interest, $R$ is not a cyclotomic base.
Nonetheless, there is an interesting version of the cap product that
turns up in this context that is relevant to our applications.

There is a natural map $THC(A/R) \to THC(A)$ induced by the cosimplicial maps
\begin{equation}
C^{\bullet}_{\cyc, R} A \to C^\bullet_{\cyc} A
\end{equation}
given levelwise by the natural maps $F_R(A^{\sma_R m}, A) \to F(A^{\sma m}, A)$.

We will construct a relative cap product map
\begin{equation}
\cap_R \colon THC(A/R) \sma THH(A) \to THH(A)
\end{equation}
that is a map of $THH(R)$-modules (where the action on the domain comes from the action on
$THH(A)$), as follows.

Consider the following commutative diagram:
\begin{equation}
\begin{tikzcd}
F_R(A^{\sma_R m}, A) \sma A^{\sma (n+1)} \sma R^{\sma (n+1)} \ar[r] \ar[d] & F_R(A^{\sma_R m}, A)
\sma A^{\sma (n+1)} \ar[d] \\
F_R(A^{\sma_R m}, A) \sma A \sma A^{\sma m} \sma A^{\sma (n-m)} \sma R^{\sma (n+1)} \ar[r] \ar[d] &
F_R(A^{\sma_R m}, A) \sma A \sma A^{\sma m} \sma A^{\sma (n-m)} \ar[d] \\
F_R(A^{\sma_R m}, A) \sma A \sma A^{\sma_R m} \sma A^{\sma (n-m)} \sma R^{\sma (n+1)} \ar[r] \ar[d]
& F_R(A^{\sma_R m}, A) \sma A \sma A^{\sma_R m} \sma A^{\sma (n-m)} \ar[d] \\
A \sma A \sma A^{\sma (n-m)} \sma R^{\sma (n-m+1)} \ar[r] \ar[d] & A \sma A \sma A^{\sma (n-m)}
\ar[d] \\
A^{\sma (n-m+1)} \sma R^{\sma (n-m+1)} \ar[r] & A^{\sma (n-m+1)}.
\end{tikzcd}
\end{equation}
This is a map of pairs and thus induces a commutative diagram
\begin{equation}
\begin{tikzcd}
\Tot(C^\bullet_{\cyc, R} A) \sma |N^{\cyc}_\bullet (A \sma R)| \ar[r] \ar[d] & \Tot(C^\bullet_{\cyc,
R} A) \sma |N^{\cyc}_\bullet A| \ar[d] \\
\relax|N^{\cyc}_\bullet (A \sma R)\relax| \ar[r] & \relax|N^{\cyc}_\bullet A\relax|
\end{tikzcd}
\end{equation}
or equivalently (using suitable cofibrant-fibrant replacement)
\begin{equation}
\begin{tikzcd}
THC(A/R) \sma THH(A \sma R) \ar[r] \ar[d] & THC(A/R) \sma THH(A) \ar[d] \\
THH(A \sma R) \ar[r] & THH(A),
\end{tikzcd}
\end{equation}
which when composed with the homeomorphism $THH(A) \sma THH(R) \cong THH(A \sma R)$ implies that the
relative cap product exists and is a $THH(R)$-module map.

Note that there is a subtlety here, which is that the action of
$THH(R)$ on $THH(A)$ that arises from the pairing theorem is homotopic
to but not equal to the usual action.  The argument for this is
analogous to the proof of~\cite[2.3.iii]{mcclure1999solution}; we give
a detailed exposition in Section \ref{sec:thh-r-homotopies}.  

Modulo this subtlety, we now explain the proof of
Theorem~\ref{thm:relative-bk-formula}. 

\begin{proof}
We note first that the top right square of
\eqref{eq:relative-bk-formula} was shown to commute in $Sp$ by Theorem
\ref{thm:absolute-bk-formula}. Thus the full square of
\eqref{eq:relative-bk-formula} commutes in $Sp$. Earlier, we explained
the sense in which the composed map in the top row of
\eqref{eq:relative-bk-formula} is $THH(R)$-linear. The composed map on
the left of \eqref{eq:relative-bk-formula} is manifestly linear
relative to $THH(R) \to THH(R)^{\Phi C_p} \to R^{\Phi C_p}$, as is the
composed map on the right. The map on the bottom is manifestly
$R^{\Phi C_p}$-linear. Now using the convenient model structures of
\cite{cyc23}, by choosing $R$ to be a cofibrant commutative ring
spectrum, $A$ to be a cofibrant-fibrant $R$-algebra, we follow the
previous argument, but now we use the additional facts that since
cofibrant $R$-modules are cofibrant $\SS$ modules, we can apply Lemma
\ref{lem:totdiag} to $cC^\bullet_{\cyc}(A/R)$ and $cC^\bullet_\cyc(A)$
with $c$ being the cofibrant replacement functor for $R$-modules in
both cases rather than for $\SS$-modules, and the lemma as well as the
earlier claims continue to hold in this interpretation. This allows us
to conclude that the composition of the top two squares of
\eqref{eq:relative-bk-formula} commutes at the level of the derived
categories, and subsequently that the bottom square does, using the
last few points of Proposition \ref{prop:convenient-model-structure}
and the fact that the smash product of cofibrant objects is
cofibrant. 
\end{proof}

\section{The Cartier formula in terms of relative Tate diagonals}
\label{sec:relative-tate-diagonals}

The purpose of this section is to explain the perspective that Theorem
\ref{thm:relative-bk-formula} is a noncommutative analog of Cartier's
formula~\eqref{eq:classical-bk-formula}.  That is, we will begin the
project of interpreting the formal algebra we have done so far in
terms of arithmetic and symplectic geometry.  In the discussion that
follows, we will shift to stating our results directly in the
$\infty$-categorical language, and as a consequence will implicitly
always be working homotopically.

\subsection{The relative topological Cartier formula}

As mentioned in Section~\ref{sec:background}, for any $\EE_1$-ring 
$A$, there are cyclotomic structure maps in the sense of
Nikolaus-Scholze~\cite{nikolaus-scholze}
\[
THH(A) \to \Phi^{C_p} THH(A) \to THH(A)^{tC_p}
\]
where the first map is an equivalence (that encodes the ``classical''
cyclotomic structure) and the second is the
canonical map from geometric to Tate fixed points. 

The relative situation is more complicated; as we have noted above,
$THH(A/R)$ does not always possess a cyclotomic structure, but there
is always enough structure for relative versions of the cap
product map.  Given an $R$-module spectrum $M$, we can form the
composite map 
\begin{equation}
\label{eq:relative-tate-diagonal}
M \to (N^{C_p}M)^{tC_p} \to (N^{C_p}_R M)^{tC_p}
\end{equation}
where the first map is the Tate diagonal and the second map comes from
applying the collapse map inside the Tate construction.  The
$R$-module structure on $(N^{C_p}_R M)^{tC_p}$ is induced by the
evident $R^{tC_p}$-module structure and the Tate-valued Frobenius map
\[
R \to (N^{C_p}R)^{tC_p} \to R^{tC_p}
\]
of $\EE_\infty$-algebras which makes this map $R$-linear.  Thus, we can
form the adjoint map of $R$-modules 
\begin{equation}
\label{eq:relative-tate-diagonal-adjoint}
M \sma_{R} R^{tC_p} \to (N^{C_p}_R M)^{tC_p}.
\end{equation}
This is referred to as the {\em relative Tate diagonal} or the
$R$-module Tate diagonal (e.g., see~\cite[11.2]{lawson2021unwinding}).

Now, when $R = \F_p$ or more generally $R$ is a commutative
$\F_p$-algebra, the map \eqref{eq:relative-tate-diagonal-adjoint} on
homotopy groups can be constructed directly via chain complexes via a
construction due to Kaledin~\cite[Lemma
  4.1]{kaledin2007cartier}. While this fact seems well-known to some
experts, it is not well known among symplectic topologists, who use
Kaledin's construction in all constructions of $C_p$-equivariant
operations in Floer homology~\cite{seidel2014equivariant,
  shelukhin-zhao}. We give a comparison in
Lemma~\ref{lemma:tate-diagonal-comparison}.

The relative Tate diagonal gives rise to a relative analog of the
cyclotomic structure map as follows.  Namely, the composition  
\begin{equation}
    \label{eq:cyclotomic-structure-collapse-map}
    THH(A) \to THH(A)^{tC_p} \to THH(A/R)^{tC_p}
\end{equation} 
is $THH(R)$-linear relative to the maps of $\EE_\infty$-rings
\begin{equation}
    \label{eq:cyclotomic-structure-collapse-map-base-ring}
    THH(R) \to THH(R)^{tC_p} \to R^{tC_p}
\end{equation}
where the second map is induced from the collapse map $THH(R) \to R$.
The relative $THH$-diagonal map is the adjoint map to
\eqref{eq:cyclotomic-structure-collapse-map} with respect to
\eqref{eq:cyclotomic-structure-collapse-map-base-ring}, namely, the
map 
\begin{equation}
    \label{eq:relative-cyclotomic-structure}
    \phi_R: THH(A) \sma_{THH(R)} R^{tC_p} \to THH(A/R)^{tC_p}.
\end{equation}
It turns out that in good cases, the domain of
\eqref{eq:relative-cyclotomic-structure} can be identified with an
$R$-linear invariant.  One case is when the map $THH(R) \to R^{tC_p}$
given in equation~\eqref{eq:cyclotomic-structure-collapse-map-base-ring}
factors through the map $THH(R) \to R$, i.e., coincides with the
composite $THH(R) \to R \to R^{tC_p}$.  (This is the setting where $R$
is a {\em cyclotomic base}; see below for further discussion.)
We then have an equivalence
\[
THH(A) \sma_{THH(R)} R^{tC_p} \htp
THH(A) \sma_{THH(R)} R \sma_{R} R^{tC_p},
\]
and moreover there is an equivalence
\[
(THH(A) \sma_{THH(R)} R) \sma_{R}
R^{tC_p} \htp
THH(A / R) \sma_{R} R^{tC_p}.
\]

However, although $\F_p$ is not a cyclotomic base, we
can nonetheless perform a similar identification.  Specifically, we now
show that when $R = \F_p$, $\phi_R$ is a variant of the
non-commutative Cartier maps of \cite{kaledin2008non,
  mathew2020kaledin}; it can be compared directly with some of these
constructions, although we will not pursue these comparisons here.

The key observation is that there is a map $\Z_p \to THH(\F_p)$ of
cyclotomic spectra (and in fact $\EE_\infty$ rings in cyclotomic
spectra).  One way to see this is as arising from the computation of $TC(\F_p)$; the
connective cover of $TC(\F_p)$ is $\Z_p$, and so the identity map of
$\Z_p$ induces the map in question.  (See the discussion just
before~\cite[IV.4.14]{nikolaus-scholze}).  This map lifts the unit map
on $THH(\F_p)$ as a
$\Z_p$-algebra~\cite[IV.4.10, IV.4.13]{nikolaus-scholze}.
Alternatively, one can construct it using the topological Dennis trace 
$K(\F_p)^{\wedge}_p \to THH(\F_p)$ and Quillen's calculation
$K(\F_p)^{\wedge}_p \htp \Z_p$.  We need a
few observations about this map.  First, the induced map $\Z_p^{tC_p}
\to THH(\F_p)^{tC_p}$ is an equivalence \cite[Corollary
  IV.4.16]{nikolaus-scholze}.  Second, we have the following lemma:

\begin{lemma}
The composite $\Z_p \to THH(\F_p) \to \F_p$, where the map $THH(\F_p)
\to \F_p$ is the canonical collapse map, coincides with the canonical
map $\Z_p \to \F_p$.
\end{lemma}

\begin{proof}
Because the space of maps of $S^1$-equivariant $\EE_\infty$ spectra $\Z_p
\to \F_p$ (regarding these as having trivial $S^1$-actions) coincide
with the space of maps of $\EE_\infty$ spectra $\Z_p \to \F_p$, and this
coincides with the space of maps of commutative rings $\Z_p \to \F_p$, 
it suffices to observe that $\Z_p \to \pi_0 THH(\F_p) \to \F_p$ must
be the canonical map, since there is a unique map of commutative rings
$\Z_p \to \F_p$.
(Alternatively, one can deduce this from the universal property of the
cyclotomic trace as a multiplicative natural transformation of
localizing invariants~\cite{BlumbergGepnerTabuada2014}.)
\end{proof}

We now have the following key diagram of $\EE_\infty$ ring spectra in
the category of spectra with $S^1$-action:
\begin{equation}
\label{eq:fundamental-diagram}
    \begin{tikzcd}
        \SS \ar[d] \ar[r] & \Z_p \ar[r] \ar[d]& THH(\F_p) \ar[r] \ar[d] & \F_p \ar[d, dotted] \\
        \SS^{tC_p} \ar[r] & \Z_p^{tC_p} \ar[r,"\htp"] &
        THH(\F_p)^{tC_p} \ar[r] & \F_p^{tC_p}
    \end{tikzcd}
\end{equation}
The vertical maps are the structure maps of cyclotomic spectra, where here
the map $\SS \to \Z_p$ is the unit and the other maps have been
discussed above.

All rectangles commute \emph{except} for the right-most rectangle involving
$THH(\F_p)$ and $\F_p$.  The noncommutation of the rightmost
rectangle of \eqref{eq:fundamental-diagram} is precisely the statement
that $\F_p$ is not a standard cyclotomic base (Definition \ref{def:standard-cyclotomic-base}); this is witnessed by
the known values of homotopy groups of all four objects in the
diagram~\cite[7.1]{angeltveit2018topological}.  To remember
this non-commutation we make the arrow $\F_p \to \F_p^{tC_p}$ dotted;
there is such an arrow, but the corresponding arrow does not make the
square with $THH(\F_p)$ commute.  However, it does make the square with
$\Z_p$ commute, since these are just the cyclotomic structure maps of
cyclotomic spectra arising from trivial $S^1$-actions!

The diagram \eqref{eq:fundamental-diagram} then allows us to identify
the domain of $\phi_R$ when $R = \F_p$, at least when $A$ lifts to the
sphere: 
\begin{proposition}
\label{prop:identify-domain-of-relative-map-in-case-1}
Let $\tilde{A}$ be an $\EE_1$-algebra and write $A = \tilde{A} \tensor
\F_p$. Then there is an equivalence
\[
THH(A) \sma_{THH(\F_p)} \F_p^{tC_p} \htp HH(A/\F_p) \tensor_{\F_p} \F_p^{tC_p}.
\]
\end{proposition}

\begin{proof}[Proof of Proposition
\ref{prop:identify-domain-of-relative-map-in-case-1}]
First, we have that
\begin{equation}
THH(\tilde{A} \sma \F_p) \sma_{THH(\F_p)} \F_p^{tC_p} \htp
THH(\tilde{A}) \sma \F_p^{tC_p} \htp THH(\tilde{A}) \sma \Z_p
\sma_{\Z_p} \F_p^{tC_p},
\end{equation}
where the first equation comes from the fact that $THH$ is symmetric
monoidal and the map $\Z_p \to \F_p^{tC_p}$ is given by the composite
$\Z_p \to THH(\F_p) \to THH(\F_p)^{tC_p} \to \F_p^{tC_p}$ in
\eqref{eq:fundamental-diagram}.  Observe that this map coincides with
the composite $\Z_p \to \Z_p^{tC_p} \to \F_p^{tC_p}$, as discussed
above.

We then have the equivalences
\begin{equation}
THH(\tilde{A}) \sma \Z_p \sma_{\Z_p} \F_p^{tC_p}
\htp HH(\tilde{A} \sma \Z_p/\Z_p) \tensor_{\Z_p} \F_p^{tC_p} \htp
HH(\tilde{A} \tensor \F_p/\F_p) \tensor_{\F_p} \F_p^{tC_p},
\end{equation}
where the first comparison uses the fact that $THH$ is symmetric
monoidal along with the basechange formula for relative $THH$ and the
last comparison follows from the commutation of the rectangle
involving the map $\Z_p \to \F_p$.
\end{proof}

We note that as promised this result immediately explains why Theorem
\ref{thm:relative-bk-formula} is a noncommutative analog of Cartier's
formula \eqref{eq:classical-bk-formula}. Indeed, it can be shown by
direct computation that when $A = \F_p[t_1, \ldots, t_k]$, the map
$\phi_R$ agrees with the two-periodic version of the Cartier
isomorphism under the HKR-isomorphisms.  For further discussion, see
Section~\ref{eq:p-fold-covers-on-free-loop-space}. 

\subsection{Relative cyclotomic structure maps}
\label{sec:relative-bk-formula}

We now establish a relative variant of Proposition
\ref{prop:identify-domain-of-relative-map-in-case-1} over $R = \F_p
\sma \tilde{R}$, where $\tilde{R}$ is a particular kind of {\em
  cyclotomic base}.  We begin with a concise review of the notion of a
cyclotomic base; see~\cite[\S3.2]{HahnRaksitWilson2025}
and~\cite{blumberg-mandell-yuan} for more discussion.  An
$\EE_\infty$-ring cyclotomic spectrum $A$ is a cyclotomic base if the
canonical map $THH(\bU A) \to A$ is a map of cyclotomic spectra, where
$\bU A$ denotes the underlying non-equivariant $\EE_\infty$-ring
spectrum of $A$.  The significance of this condition is that when $A$
is a cyclotomic base, $THH(- / A)$ inherits a cyclotomic structure.

In the case of the sphere spectrum, the trivial cyclotomic structure
on the non-equivariant sphere is a cyclotomic base.  More generally,
we will be interested in cyclotomic bases that arise from cyclotomic
structures on the trivial $S^1$ action.

\begin{definition}\label{def:standard-cyclotomic-base} A \emph{standard cyclotomic base} is an
$\EE_\infty$-ring-spectrum $\tilde{R}$ such that the trivial
$S^1$ action on $\tilde{R}$ makes it a cyclotomic base, i.e., that
there is a cyclotomic structure map $\lambda \colon \tilde{R} \to
\tilde{R}^{tC_p}$ such that 
\[
\begin{tikzcd}
  THH(\tilde{R}) \ar[d] \ar[r] & \tilde{R} \ar[d,"\lambda"] \\
  THH(\tilde{R})^{tC_p} \ar[r] & \tilde{R}^{tC_p}
\end{tikzcd}
\]
commutes.
\end{definition}

We now have the following extension of the work above.

\begin{proposition}\label{prop:identify-domain-of-relative-map-over-F_p[t]}
Let $\tilde{R}$ be a standard cyclotomic base and $\tilde{A}$ be an
$\EE_1$-algebra over $\tilde{R}$.  Let $R = \tilde{R} \sma \F_p$ and
$A = \tilde{A} \tensor \F_p$ be the corresponding $\EE_1$-algebra over
$R$. Then there is an equivalence
\[
THH(A) \sma_{THH(R)} R^{tC_p} \htp HH(A/R) \sma_{R} R^{tC_p},
\]
where the map $R = \tilde{R} \sma \F_p \to (\tilde{R} \sma \F_p)^{tC_p}
= R^{tC_p}$ is the map given by the smash of $\lambda$ and the
Tate-valued Frobenius composed with the lax monoidal structure map on
$(-)^{tC_p}$. 
\end{proposition}

\begin{proof}[Proof of Proposition \ref{prop:identify-domain-of-relative-map-over-F_p[t]}]
The facts that $\tilde{R}$ is a standard cyclotomic base and the
properties of~\eqref{eq:fundamental-diagram} together imply that the
following diagram commutes: 
    \begin{equation}
    \label{eq:mix-fp-and-s[t]}
    \begin{tikzcd}
        THH(\F_p) \sma THH(\tilde{R}) \ar[r] & 
        THH(\F_p)^{tC_p} \sma THH(\tilde{R})^{tC_p} \ar[r] &
        \F_p^{tC_p} \sma \tilde{R}^{tC_p} \\
        \Z_p \sma THH(\tilde{R}) \ar[u] \ar[r] & 
        \Z_p \sma \tilde{R} \ar[r] &
        \Z_p^{tC_p} \sma \tilde{R}^{tC_p} \ar[u]. 
    \end{tikzcd}
    \end{equation}
Thus the same diagram commutes when we compose with the lax monoidal
structure maps on the right and the top middle. We also have the
commuting diagram:
    \begin{equation}
        \label{eq:zp-to-fp-relative}
        \begin{tikzcd}
        \Z_p \sma \tilde{R} \ar[r]\ar[d]
        &\Z_p \sma \tilde{R}^{tC_p} \ar[d] \\
        \F_p \sma \tilde{R} \ar[r, "\lambda"] & (\F_p \sma \tilde{R})^{tC_p}
        \end{tikzcd}
    \end{equation}

Putting this all together, we can run the analogous argument to the
proof of the previous proposition:
    \begin{equation}
    \label{eq:identify-domain-given-lift-to-cyclotomic-base}
        \begin{gathered}
            THH(A) \sma_{THH(R)} R^{tC_p} = THH(\tilde{A}) \sma THH(\F_p) \sma_{THH(\F_p) \sma
            THH(\tilde{R})} R^{tC_p} \\
            = THH(\tilde{A}) \sma_{THH(\tilde{R})} R^{tC_p} 
            = THH(\tilde{A}) \sma \Z_p \sma_{\Z_p \sma THH(\tilde{R})} \F_p^{tC_p} \\
            =  THH(\tilde{A}/\tilde{R}) \sma \Z_p \sma_{\Z_p \sma \tilde{R}} R^{tC_p} \\
            = HH(\tilde{A} \tensor \F_p/R) \tensor_{R} R^{tC_p}
        \end{gathered}
    \end{equation}

Here the bottom two equalities use precisely the commutativity of \eqref{eq:mix-fp-and-s[t]} and
\eqref{eq:zp-to-fp-relative}.
\end{proof}

When $\tilde{A}$ is dualizable, we can say something stronger.  Our
argument relies on the techniques developed
in~\cite{antieau2018blumberg} in their proof of the Kunneth theorem
for $THH(-)^{tS^1}$ applied to smooth and proper $k$-linear dg
categories (following~\cite{BlumbergMandell2024Kunneth}).

\begin{proposition}
\label{prop:phi-is-an-equivalence-for-cyclotomic-base}
Let $\tilde{R}$ be a standard cyclotomic base. Then when $\tilde{A}$ is a
dualizable $\tilde{R}$-algebra, the map $\phi_{R}$ of
\eqref{eq:relative-cyclotomic-structure} for $R = \tilde{R} \sma
\F_p$ is an equivalence.
\end{proposition}

\begin{proof}
We use the criterion of~\cite[4.6]{antieau2018blumberg}; any symmetric
monoidal natural transformation between symmetric monoidal functors
where all objects in the domain are dualizable is an equivalence.

Take $\CC_1$ to be the category of dualizable objects of
$\EE_1(\tilde{R})$ (i.e., the category of smooth and proper $\EE_1$
$\tilde{R}$-algebras), and take $\CC_2$ to be the category of
$R^{tC_p}$-modules. The functors 
\[
\tilde{A} \mapsto THH(\tilde{A} \sma \F_p) \sma_{THH(R)}
R^{tC_p} \htp THH(\tilde{A}) \sma_{THH(\tilde{R})} R^{tC_p} \htp
THH(\tilde{A}/\tilde{R}) \sma_{\tilde{R}} R^{tC_p}\text{ and }
\]
and
\[\tilde{A} \mapsto THH(\tilde{A} \sma \F_p/R)^{tC_p} =
(THH(\tilde{A}/\tilde{R})\sma_{\tilde{R}} R)^{tC_p}\] are both
symmetric monoidal functors.  For the second functor, this is a
consequence of the fact that by hypothesis $THH(\tilde{A}/\tilde{R})$
is a perfect $\tilde{R}$-module; the criterion of the discussion
following~\cite[1.1]{antieau2018blumberg} now ensures that the Tate
construction is symmetric monoidal.  The same observation also holds
for the first functor using the third line
of~\eqref{eq:identify-domain-given-lift-to-cyclotomic-base}.

The hypothesis that $\tilde{R}$ is a cyclotomic base gives rise to the
cyclotomic structure map 
\[
\phi'_{\tilde{R}}\colon THH(\tilde{A}/\tilde{R}) \to THH(\tilde{A}/\tilde{R})^{tC_p}
\]
This map is an endofunctor of $\tilde{R}$-mod if we give the target
the module structure induced from the map $\lambda: \tilde{R} \to
\tilde{R}^{tC_p}$.  We have that $\phi_R$ is simply
\[
THH(\tilde{A}/\tilde{R}) \sma_{\tilde{R}} R^{tC_p} \xrightarrow{
  \phi'_{\tilde{R}} \sma 1} THH(\tilde{A}/\tilde{R})^{tC_p}
\sma_{\tilde{R}} R^{tC_p} \to (THH(A/R))^{tC_p}.
\]
An elementary verification using the fact that $THH$ is a symmetric
monoidal functor to cyclotomic spectra shows then that $\phi_R$ is a
symmetric monoidal natural transformation. 
\end{proof}

\subsection{An extension of the relative topological Cartier formula}

We note that the above methods also show the following theorem.

\begin{theorem}
\label{thm:nice-relative-bk-formula-over-Fp-ring}
Let $\tilde{R}$ be a standard cyclotomic base.  If $\tilde{A}$ is a
smooth proper category over $\tilde{R}$, then there is a commutative
square in the $\infty$-category of spectra: 
\begin{equation}
\label{eq:genuine-bk-formula-relative-F_p}
    \begin{tikzcd}
        HH^\bullet(A/R) \tensor_R F^*HH(A/R) \tensor_{\F_p} \F_p^{tC_p}  \ar[rr, "F^*\cap_R
        \tensor_{\F_p} id"] \ar[d, "\Delta \tensor_R \phi_R", swap] & & F^*HH(A/R) \tensor_{\F_p}
        \F_p^{tC_p} \ar[d, "\phi_R"] \\
       (_R N_e^{C_p} HH^\bullet(A/R) \tensor_R HH(A/R))^{tC_p} \ar[rr, "(\cap^p_R)^{tC_p}", swap] &
       & (HH(A/R))^{tC_p}, 
    \end{tikzcd}
\end{equation}
where $A = \tilde{A} \sma \F_p$ and $R = \tilde{R} \sma \F_p$.
\end{theorem}

\begin{proof}
One applies the natural transformation from geometric fixed points to
Tate fixed points to the bottom row of \eqref{eq:relative-bk-formula},
and as in the earlier arguments, composes with the map $HH(\tilde{A}
\tensor \Z_p/\Z_p) \to THH(A)$ on the top row of
\eqref{eq:relative-bk-formula} and then uses the commutative diagram
\eqref{eq:mix-fp-and-s[t]} to obtain the claim.  
\end{proof}

Finally, we observe that not all of the structure of a standard
cyclotomic base was needed for the proofs of the results of the
previous section. 

\begin{definition}
\label{def:weak-Z_p-cyclotomic-base}
Let $\tilde{R}$ be an $\EE_\infty$-ring spectrum. We say that
$\tilde{R}$ is a standard $\Z_p$-cyclotomic base if, writing $R_{\Z_p}
= \tilde{R} \sma_{\SS} \Z_p$, there is a map $F$ such that the
following diagram of $\EE_\infty$-$S^1$-ring spectra in the stable 
$\infty$-category of $\Z_p$-modules commutes:
\begin{equation}
    \begin{tikzcd}
        THH(\tilde{R}) \sma \Z_p = HH(R_{\Z_p}/\Z_p) \ar[r] \ar[d, "\phi \tensor can"] & R_{\Z_p}
        \ar[d, "F"] \\
        HH(R_{\Z_p}/\Z_p) \ar[r] & R_{\Z_p}^{tC_p}
    \end{tikzcd}
    \end{equation}
    where implicitly we apply the lax monoidal structure of the Tate construction to the left
    vertical arrow. 
\end{definition}

\begin{lemma}
Proposition \ref{prop:identify-domain-of-relative-map-over-F_p[t]} and
Proposition \ref{prop:phi-is-an-equivalence-for-cyclotomic-base} hold
if $\tilde{R}$ is only a standard $\Z_p$-cyclotomic
base.
\end{lemma}

\section{Calculus over the affine line over the sphere.}

The purpose of this section is to connect our work on the cap product
pairing on topological Hochschild (co)homology to the Cartier formula
in the context of differential geometry over $\F_p$.  The approach we
take involves working with some basic objects of derived algebraic
geometry, starting with the flat affine line.

\subsection{The flat affine line and related commutative ring spectra}

Since $\Z$ and $\mathbb{N}$ are abelian monoids, their suspension
spectra are $\EE_\infty$-rings.  We write
\[
\SS[x] := \Sigma^\infty_+ \mathbb{N} \qquad\textrm{and}\qquad
\SS[x,x^{-1}] := \Sigma^\infty_+ \Z.
\]
The suspension spectra of the pointed abelian monoids
\[ M_n = \{*\} \cup \{0, \ldots, n\}\]
where the monoidal operation is given by addition, will be written as
\[
\SS[x]/x^{n+1} := \Sigma^\infty M.
\]
The maps of monoids $M_n \to M_{n-1}$ collapsing $n$ to $\{*\}$ induce
a corresponding inverse system of $\EE_\infty$ rings, and their inverse
limit is denoted
\[
\SS[[x]] := \varprojlim \SS[x]/x^n.
\]
We can then invert the action of $x \in \pi_0(\SS[[x]])$ and form
\[
\SS((x)) := \SS[[x]][x^{-1}].
\]

We note that there is an equivalence of $\EE_\infty$-rings
\[
\SS_p[[x]] := \varprojlim_n \SS_p \sma \SS[x]/x^n \simeq \SS[[x]]_p
\]
since limits commute.  The canonical map 
\[
(\SS_p[[x]])[x^{-1}] \to (\SS((x)))_p
\]
associated to the colimit-limit exchange morphism is an equivalence by
computing the corresponding map on homotopy groups. 

We can also form
\[
\SS[1/N][x] := \SS[1/N] \sma \mathbb{N}
\]
and the completion of such a ring with respect to the ideal generated
by $p$ and $x$ in $\pi_0(\SS[1/N][x])$ is equivalent to $\SS_p[[x]]$,
again by computing the canonical colimit-limit exchange map on
homotopy groups, and noting that $\SS_p[[x]][1/N] \simeq \SS_p[[x]].$  

All of these $\EE_\infty$-ring spectra are in fact standard cyclotomic
bases.  For $\SS[x]$, this was observed
in~\cite[11.1]{BhattMorrowScholze2019THH}.

\begin{proposition}
\label{prop:weak-cyclotomic-bases}
Let $\tilde{R}$ be any one of the rings $\SS[x], \SS[x,x^{-1}],
\SS[[x]], \SS((x))$, or any one of these rings with $N$ inverted. Then
$\tilde{R}$ is a standard cyclotomic base.  
\end{proposition}

We make no particular claim to originality for the preceding result,
although our argument for 
$\SS((x))$ uses some nontrivial recent results; we will need this
latter result for our application to Fukaya categories. We give a
proof below in Sections \ref{sec:differential-forms-on-A1-S} for $R = \SS[x]$ and
$\SS[x,x^{-1}]$ along the lines of the argument
in~\cite[11.1]{BhattMorrowScholze2019THH}, and in
\ref{sec:cyclotomic-bases-exist} for the other rings.

We also use the following convenient algebraic lemma.

\begin{lemma}
If $\tilde{R}$ is connective and a standard cyclotomic base then so is 
$\tilde{R}[1/N]$.
\end{lemma}

\begin{proof}
This follows by taking the pushout of the defining diagram along the
map $\SS\to \SS[1/N]$, which is a map of $\EE_\infty-S^1$-rings, and
noting that the Tate construction commutes with localization due to
\cite[Lemma I.2.9]{nikolaus-scholze}.  
\end{proof}

Finally, note that as an application of the criterion of
Proposition~\ref{prop:burklund} (using Lemma~\ref{lemma:lawson-lemma}), we
have that there are equivalences 
\begin{equation}
    \label{eq:tate-construction-on-infinite-objects}
    \tilde{R}_p \to \tilde{R}^{tC_p}
\end{equation} 
for each of these $\EE_\infty$ rings.

\subsection{Differential forms on the spectral affine line.}
\label{sec:differential-forms-on-A1-S}

Let us now focus on $R = \SS[x]$, which can be thought of as functions
on the (flat) affine line over the sphere spectrum. Then there is a
map of cofiber sequences of spectra carrying a significant amount of
additional structure. 

\begin{equation}
\label{eq:differential-forms-on-the-sphere}
\begin{tikzcd}
    \Sigma \SS[x]' \ar[r] \ar[d] & THH(\SS[x])\ar[r] \ar[d]& \SS[x] \ar[d] \ar[r]   & \Sigma^2
    \SS[x]' \ar[d]\\
    \Sigma (\SS[x]')^{tC_p} \ar[r] & THH(\SS[x])^{tC_p} \ar[r]& (\SS[x])^{tC_p} \ar[r] & (\Sigma^2
    \SS[x]')^{tC_p}.
\end{tikzcd}
\end{equation}
\begin{itemize}
\item The map $\alpha\colon THH(\SS[x]) \to \SS[x]$ is the canonical collapse map, and is thus a map
of $S^1-\EE_\infty$-rings. 
\item Since $\SS[x]$ is an $\EE_\infty$-ring, the map $\alpha$ has a section $\beta\colon \SS[x] \to
THH(\SS[x])$, as a map of $\EE_\infty$ rings. However, we will see that $\beta$ does not lift to a
map of $S^1-\EE_\infty$-rings. 
\item The fiber of $\alpha$ can be taken in the category of $S^1$-modules over $THH(\SS[x])$; this is
$\Sigma \SS[x]'$. The reason for the notation is that if we view the fiber as an $\SS[x]$-module via
$\beta$, it turns out to agree with $\Sigma \SS[x]$. However, the spectrum $\Sigma \SS[x]$ has a
nontrivial $S^1$-action by construction; that is the reason for the notation $\Sigma \SS[x]'$. We note
that $\Sigma \SS[x]'$ is \emph{not} an $S^1$-module over $\SS[x]$. 
\item The bottom row is the Tate construction applied to the top
  row. The vertical structure maps are defined by the cyclotomic
  structure map on $THH$ and the map $t \mapsto t^p$ on $\SS[x]$,
  followed by the canonical map to the Tate construction. In
  particular, this latter map witnesses the statement that $\SS[x]$ is
  a standard cyclotomic base (see~\cite[\S 11.1]{nikolaus-scholze},~\cite[\S
    7.1]{blumberg-mandell-yuan}).
\item The analogous diagram replacing $\SS[x]$ by $\SS[x, x^{-1}]$ in
  all places also commutes.  
\end{itemize}
In fact, an analogous diagram with $\SS[x]$ replaced by any $\EE_\infty$
ring $R$ manifestly exists once one establishes that $R$ is a standard
cyclotomic base.
 
The existence of this diagram in the cases where $R = \SS[x]$ and
$\SS[x,x^{-1}]$ follows from work of
Hesselholt~\cite{hesselholt-p-typical-curves} which we now explain.
Indeed, working with orthogonal spectra, we have that for any pointed
monoid $M$, we have
\[THH(\SS[M]) \cong THH(\SS) \sma \Pi^{cy}_{\wedge} M \cong \SS \sma \Pi^{cy}_{\wedge} M \cong
\Sigma^\infty \Pi^{cy}_{\wedge} M\]
where we use the fact that orthogonal spectra are tensored over
pointed spaces via levelwise smash product, and $\Pi^{cy}_{\wedge} M$
denotes the cyclic bar construction performed in pointed spaces. When
$M$ is a commutative monoid, $\Pi^{cy}_{\wedge} M$ is manifestly a
commutative monoid (as the geometric realization of a simplicial
object in topological abelian groups), and \cite[2.2.3]{hesselholt-p-typical-curves} then shows that
\[ \Pi^{cy}_{\wedge} \mathbb{N} \htp * \sqcup \bigsqcup_{k \geq 1}
S^1_k \quad \Pi^{cy}_{\wedge} \Z \htp \bigsqcup_{k \in \Z} S^1_k, \]
and so we conclude that
\[ THH(\Sigma^\infty_+ \mathbb{N}) \htp \SS \vee \bigvee_{k=1}^\infty
\Sigma^\infty_+(S^1_k),\; THH(\Sigma^\infty_+ \Z) \htp \bigvee_{k=-\infty}^\infty
\Sigma^\infty_+(S^1_k), \]
where $\Sigma^\infty_+ S^1_k$ is the suspension spectrum of the circle $S^1_k$ equipped with the
$S^1$-action given by 
\[ S^1 \times S^1 \to S^1 , (z, w) \mapsto z^kw. \]  
The commutative monoid structure on $\Pi^{cy}_{\wedge} M$ is in this case given by
\[ S^1_{k_1} \times S^1_{k_2} \to S^1_{k_1+k_2}, (\theta_1, \theta_2) \mapsto (\theta_1 + \theta_2)
\]
where we identify $S^1_k$ with $\R/\Z$ for every $k$. 
One can verify that the above formulae do indeed give the ring
structure by noting that it is the unique $S^1$-equivariant ring
structure on $\Pi^{cy}_\wedge M$ extending the map of rings $M \to
\Pi^{cy}_\wedge M$.

The spectra $\SS^1_k$ can also be written as 
\begin{equation}
\label{eq:decomposition-of-k-circle}
\SS^1_k = \SS \vee \SS^{V_k-\R},
\end{equation} 
where $V_k$ is the representation of $S^1$ on $\C$ with character $z \mapsto z^k$.  The top cofiber
sequence of \eqref{eq:differential-forms-on-the-sphere} is a wedge of the maps 
\[ \SS^{V_k-\R} \to \Sigma^\infty_+ S^1_k\to \SS \to \SS^{V_k},\]
thus explaining the $S^1$-action on the spectrum $\Sigma^2 \SS[x]'$ in the diagram
\eqref{eq:differential-forms-on-the-sphere}.  From this description, one also sees that $\SS \to
\SS^{V_k}$ is induced from the inclusion $S^0$ as the fixed point locus of $S^{V_k}$, and thus is
trivial on homology but nontrivial on equivariant cohomology. 

The diagram \eqref{eq:differential-forms-on-the-sphere} comes from
constructing a corresponding diagram for the Tate fixed points.  Then
following Hesselholt's computation for $R = \SS[x]$ or $\SS[x,
  x^{-1}]$ \cite[Lemma 3.1.6]{hesselholt-p-typical-curves} (see also
\cite[Proposition 11.3]{BhattMorrowScholze2019THH} for a discussion
for the case of $\SS[x]$), the genuine cyclotomic structure maps
$THH(R) \to THH(R)^{t C_p}$ are induced on components by the map
sending $\Sigma^\infty_+ S^1_k$ to $\Sigma^\infty_+ S^1_{pk}$ by
suspending the map $w \mapsto w^p$.  Thus, we conclude that there is a
commutative diagram of $S^1$-orthogonal ring spectra
\[
\begin{tikzcd}
THH(R) \ar[r] \ar[d] & R \ar[d] \\ THH(R)^{t C_p} \ar[r] & R^{t C_p}
\end{tikzcd}
\]
with the map $R \to R^{t C_p}$ induced simply from the map of
commutative monoids $x \mapsto x^p$.  This completes the proof of
Proposition \ref{prop:weak-cyclotomic-bases}.
 
A geometrically-minded reader can interpret the diagram above for $R =
\SS[x,x^{-1}]$ in terms of free loop spaces. Indeed, one has an
equivalence of cyclotomic spectra (\cite{hesselholt1997k,
  nikolaus-scholze})
\[ THH(\SS[x, x^{-1}]) = THH(\Sigma^\infty_+ \Omega S^1) = \Sigma^\infty_+ LS^1 = \bigvee_{k \in \Z}
\Sigma^\infty_+ S^1_k\]
where the latter wedge sum is just the decomposition of $LS^1$ into
connected components labeled by the winding number $k$. Under the
correspondence with free loop spaces, the cyclotomic structure map
arises by suspending the $p$-fold-cover map on the free loop space,
i.e., it is induced by the isomorphism
\[ \Sigma^\infty LS^1 \xrightarrow{\Sigma^\infty_+ \bar{\phi}_p }
\Sigma^\infty_+ (LS^1)^{C_p} \cong \Phi^{C_p} (\Sigma^\infty_+ LS^1)\]
where $\bar{\phi}_p$ is the $S^1 \simeq S^1/C_p$-equivariant
homeomorphism $LS^1 \to (LS^1)^{C_p}$ given by sending loops to their
$p$-fold iterates.

Informally, motivated by the Hochschild-Kostant-Rosenberg theorem,
one should think of $THH(\SS[x])$ as the ``differential forms on
$\mathbb{A}^1_{\SS}$'', with the de Rham differential corresponding to
the $S^1$-action. From that perspective, the top row of
\eqref{eq:differential-forms-on-the-sphere} can be regarded as the
exact sequence   
\[
\Omega^1_{\mathbb{A}^1_{\SS}}[1] \to
\Omega^\bullet_{\mathbb{A}^1_{\SS}} \to
\mathcal{O}_{\mathbb{A}^1_{\SS}} \to \Omega^1_{\mathbb{A}^1_{\SS}}[2].
\]
Here one should regard the component $\SS^{V_k - \R}$ as corresponding
to the differential form $x^k dx/x$; the fact that ``differential
forms carry nontrivial $S^1$-actions'' is closely connected to
properties of calculus in characteristic $p$ and in $p$-adic
settings. We now explain these phenomena in this concrete setting.

\subsection{From spectral to classical calculus.}
\label{sec:spectral-to-classical}

Let $\kk$ be a classical commutative algebra.  We are now going to
translate the work of the previous section back into algebra, and to
do this we will rely on the comparisons between the spectrum $H\kk$
and the stable category of $H\kk$-modules and $\kk$ regarded as a
differential graded algebra and the dg category of $\kk$-modules.
Furthermore, recall that there is an evident equivalence $H\kk \sma
\bS[x] \htp H\kk[x]$, and that
\begin{equation}
THH(\SS[x]) \sma H\kk \htp THH(H\kk[x]/ H\kk) \htp HH(\kk[x] / \kk).
\end{equation}

We will now begin the process of reinterpreting the
diagram~\eqref{eq:differential-forms-on-the-sphere} in terms of $\kk$
and $\kk[x]$.  Taking the smash product of the vertical map
of~\eqref{eq:differential-forms-on-the-sphere} with the canonical map
$H\kk \to H\kk^{tC_p}$ and composing with the lax monoidality of the
Tate construction in the bottom row now gives the diagram:

\begin{equation}
\label{eq:differential-forms-on-Hk}
\begin{tikzcd}
    \Sigma H\kk[x]' \ar[r] \ar[d] & THH(H\kk[x]/ H\kk)\ar[r] \ar[d]& H\kk[x] \ar[d] \ar[r]   &
    \Sigma^2 H\kk[x]' \ar[d]\\
    \Sigma (H\kk[x]')^{tC_p} \ar[r] & THH(H\kk[x]/H\kk)^{tC_p} \ar[r]& (H\kk[x])^{tC_p} \ar[r] &
    (\Sigma^2 H\kk[x]')^{tC_p}.
\end{tikzcd}
\end{equation}

We are going to explain in the remainder of this section why the
preceding diagram can be rewritten as follows:
\begin{equation}
\label{eq:differential-forms-over-k}
\begin{tikzcd}
    \Omega^1_{\kk[x]/\kk}[1] \ar[r] \ar[d] & HH(\kk[x]/\kk) \ar[r]\ar[d, "\phi \tensor can"]& \kk[x]
    \ar[d, "F"] \ar[r, "ud"]
    & \Omega^1_{\kk[x]/\kk}[2] \ar[d, "F'"]\\
    (\Omega^1_{\kk[x]/\kk}[1])^{tC_p} \ar[r] & (HH(\kk[x]/\kk))^{tC_p} \ar[r]& (\kk[x])^{tC_p}
    \ar[r] & (\Omega^1_{\kk[x]/\kk}[2])^{tC_p}.
\end{tikzcd}
\end{equation}

To begin to make sense of this, recall that the normalized bar complex
of a classical commutative algebra is the geometric realization of a
simplicial commutative algebra, and so is a commutative differential graded
algebra.  Moreover, recall that for any commutative algebra $A/\kk$,
there are maps of cdgas 
\[
\widetilde{\Omega}^\bullet_{dR}(A/\kk) = \bigoplus_i \Omega^i_{A/\kk}[i]
\xrightarrow{\epsilon = \oplus_i \epsilon_i} HH(A/\kk)
\xrightarrow{\pi = \oplus_i \pi_i} \bigoplus_i
\Omega^i_{A/\kk}[i]
\]
where the differentials on the domain and codomain are zero, each
given by antisymmetrization, such that  
\begin{enumerate}[(a)]
\item The composition  $\pi_i \epsilon_i$ is multiplication by $i!$, and 
\item (remarkably) The map $\epsilon$ is always an isomorphism on
  homotopy groups if $A$ is smooth over $\kk$ by the HKR theorem (e.g., see~\cite{HKR}
  and~\cite[Theorem 3.4.4]{loday2013cyclic}).  
\end{enumerate}
These maps allow us to understand $HH(A/\kk)$ in terms of differential
forms as an $\EE_1$-algebra over $\kk$, whenever $A$ is a (classical)
smooth commutative algebra of finite type over $\kk$.  (More
generally, the theorem holds for algebras that are smooth over $\kk$
in the sense of \cite[Appendix E]{loday2013cyclic}.) However,
understanding $HH(A/\kk)$ in these terms as an $S^1$-$\EE_\infty$
algebra is more subtle. When $\Q \not\subset \kk$, the above diagram
does not even immediately reflect the structure of $HH(\kk[x]/\kk)$ as
an $\EE_\infty$-algebra, since in that case there is no homotopy
theory on cdgas for which maps of cdgas correspond to maps of
$\EE_\infty$-algebras in $Mod_\kk$.

\subsubsection{Mixed complexes and $S^1$-actions.}

To handle $S^1$-actions algebraically, we use the theory of mixed
complexes~\cite{Kassel87, loday2013cyclic}.  A \emph{mixed complex} (in
$\kk$-modules) is a graded $\kk$-module equipped with a pair of
operators $d, B$ of degree $-1$ and $+1$ respectively such that $d^2 =
B^2 = db+Bd = 0$. In other words, it is a module over the exterior
algebra $H_*(S^1, \kk) = \kk[B]$ in chain complexes, with the algebra
structure on $H_*(S^1)$ given by the Pontrjagin product on $S^1$.

The category of mixed complexes has a symmetric monoidal structure
induced from the coalgebra structure on $H_*(S^1)$ coming from the
diagonal map $S^1 \to S^1 \times S^1$; 
explicitly, one sets
\[
B_{M_1 \tensor M_2} = B_{M_1} \tensor 1 \pm 1
\tensor B_{M_2}
\]
with the sign given by the Koszul sign rule.  The homotopy theory of
mixed complexes comes from the quasi-isomorphisms; i.e.,
$\infty$-category of mixed complexes $\kk[B]-mod$ arises by inverting
the quasi-isomorphisms.

Classically, mixed complexes arise from cyclic objects in the category
of $\kk$-modules; for example,~\cite{goodwillie1985cyclic} constructs
a functor from cyclic $\kk$-modules to mixed complexes, see
also~\cite[Section 2.5]{loday2013cyclic}.  Put slightly differently, 
the $\infty$-category of $\kk$-modules with $S^1$-action is the functor
category $F(BS^1, Mod_\kk)$, or equivalently, a module over the simplicial
commutative algebra $\kk[S^1]$ (see \cite[Lemma
  3.9]{beardsley2023koszul} for a detailed proof, or
\cite{dror1980equivariant} for an old simplicial argument).  Here the
diagonal map on $S^1$ makes $\kk[S^1]$ into a commutative coalgebra in
simplicial commutative algebras, and this coalgebra structure makes
$Mod_{R[S^1]} \simeq F(BS^1, Mod_R)$ into a symmetric monoidal
$\infty$-category. Realization then provides a functor from cyclic
objects in a stable $\infty$-category to $S^1$-equivariant objects
(e.g.,~\cite[B.5]{nikolaus-scholze}). 

Now, there is an equivalence of $\EE_1$-algebras $R[S^1] \simeq
H_*(S^1, R)$; this follows from the homotopy transfer theorem for
$A_\infty$-structures, as there is no possibility for higher order
terms in the $A_\infty$ equations on $H_*(S^1, R)$. Thus, the
$\infty$-categories  $R[B]$-mod and $R[S^1]$-mod are equivalent.  
It is a fundamental fact, however, that the above equivalence of
$\infty$-categories does \emph{not} extend to a symmetric monoidal
equivalence, although it does when $\Q \subset \kk$.  In this case,
there is an equivalence of $\EE_\infty$-bialgebras (i.e. $\EE_\infty$-algebras in
$\EE_\infty$-coalgebras) $\kk[S^1] \simeq H_*(S^1, \kk)$~\cite{toen}. 
Moreover, when additionally $A/\kk$ is smooth, the HKR map extends to
an equivalence of $S^1$-equivariant cdgas (i.e. of
$\EE_\infty-S^1$-algebras).  Away from characteristic zero, when
considering properties related to $\EE_n-\kk$-algebras with
$S^1$-action for $n\geq 2$, computations become more delicate. 

\begin{remark}
One can see that the symmetric monoidal structures on $Mod^{S^1}_\kk$
and $\kk[\epsilon]-mod$ do not agree in general by a Koszul-duality
argument. The $\EE_n$-coalgebra structure on $\kk[S^1]$ is Koszul dual
to the $\EE_n$-algebra structure on $C^*(\C P^\infty, \kk)$, while the
Koszul dual to $\kk[B]$ is $H^*(\C P^\infty, \kk)$; the putative
equivalence of the two symmetric monoidal $\infty$-categories is then
obstructed by the Steenrod operations on $C^*(\C P^\infty, \kk)$. 
\end{remark}

\subsubsection{Understanding $HH(A/\kk)$ as an $\EE_\infty$-$S^1$-algebra. }

The complex $\widetilde{\Omega}^\bullet_{dR}(A/\kk)$ is naturally a
mixed complex: one takes the operator $B$ to be the de Rham differential
$d_{dR}$.  Note that usually, the de Rham differential \emph{increases}
cohomological degree, but here, the mixed complex operator increases
\emph{homological} degree.  Thus, the complex
$(\widetilde{\Omega}^\bullet_{dR}(A/\kk), B)$ is not isomorphic in
$\mathcal{D}(\kk)$ to the complex 
$(\Omega^\bullet_{dR}(A/\kk), d_{dR})$ in the algebraic geometry
literature; this is why we use the ``tilde'' notation.

On the other hand, as discussed above there is an explicit mixed
complex structure on $HH(A/\kk)$ coming from the cyclic structure;
using the reduced bar complex to compute $HH(A/\kk)$,   
\[
HH(A/\kk) = \bigoplus_{n \geq 0} A \oplus (A/(\kk \cdot 1))^{\oplus n}
\]
with differential induced from the standard bar complex by taking the
quotient of the submodule of terms which have a $1$ in a non-initial
tensor entry, we have that \cite[Prop. 2.3.3, 2.3.4]{loday2013cyclic} 
\[
B(a_0 \tensor \ldots \tensor a_{n}) = \sum_i (-1)^{ni} (1, a_i,
\ldots, a_n, a_0, \ldots, a_{i-1}).
\]
It is then standard that $\epsilon_n d_{dR} = B \epsilon_{n-1}$,
$\pi_n B = n d_{dR} \pi_{n-1}$.  (Please note here that we
are using \emph{homological} gradings on Hochschild homology, in
contrast to the standard convention in symplectic topology.)  

From the above discussion, we can observe several facts:

\begin{enumerate}[a)]

\item The only $\kk$-linear $S^1$-action on an element of $Mod_\kk$
  concentrated in a single degree is trivial. 

\item As a $\kk[x]$-module, the HKR theorem tells us that
  \[
  \hofib(HH(\kk[x]) \to \kk[x]) \htp \Omega^1_{\kk[x]/\kk}[1] \simeq
  \kk[x][1].
  \]
In particular, the $S^1$-action on this $\kk$-module is trivial, and
the $\kk[x]$-module structure has a unique extension (up to
equivalence) to a $\kk[x]-S^1$-module structure.

\item Degree considerations imply that the top row of the
diagram~\eqref{eq:differential-forms-over-k} exhibits $HH(\kk[x]/\kk)$ as a
  square-zero extension of $\kk[x]$ by $\kk[x][1]$ in
  $\kk-S^1$-algebras. Indeed, the underlying extension of
  $\EE_\infty$-rings is $n$-small in the sense of \cite[7.4.1.18]{HA}
  (see also~\cite[Definition 3.12]{DAGIV}), and thus the
  underlying extension is a square-zero-extension
  by~\cite[7.4.1.26]{HA} (see also~\cite[3.17, 3.18]{DAGIV}). Now 
  these latter theorems immediately imply our claim, since we have
  precisely an element of  
  \[
  Fun(BS^1, Fun_{n-sm}(\Delta^1, \EE_\infty(Mod_\kk))) \simeq Fun(BS^1,
  Der_{n-sm}(\EE_\infty(Mod_\kk))) \subset Der(\EE_\infty(Mod^{S^1}_\kk)). \]

\item The map $\kk[x] \to \Omega^1_{\kk[x]/\kk}[2]$, when forgetting the
  $S^1$-equivariant structure, becomes a derivation (with the domain
  thought of as an $\EE_\infty$-algebra, and the codomain thought of
  as a module over this algebra). Since the underlying $\kk$-modules
  are concentrated in different degrees, this is the zero map. Thus,
  $HH(\kk[x]/\kk)$ is the trivial square zero extension $\kk[x] \oplus
  \Omega^1_{\kk[x]/\kk}[1] = \widetilde{\Omega}^*_{dR}(A/\kk)$ of
  $\EE_\infty$-algebras when the $S^1$-action is forgotten.  
\end{enumerate}

We now make two core claims clear via direct computation:

\begin{itemize}
    \item First, we claim that the map $\kk[x] \to \Omega^1_{\kk[x]/\kk}[2]$ is nonzero \emph{as a
    map of $HH(\kk[x]/\kk)$-modules}, even though it is zero as a map of $\kk[x]$-modules. To see
    this, we note that we can verify this claim just by thinking of $HH(\kk[x]/\kk)$ as an
    $\EE_1$-algebra. The discussion earlier lets us identify this with the $\EE_1$-algebra $\kk[x]
    \oplus \kk[x][1]$. We can then use the $2$-periodic resolution of $\kk[x]$ as a module over this
    algebra given by 
    \[ \bigoplus_{i =0}^\infty (\kk[t]1_i \oplus
    \kk[t]\epsilon_i[1])[2i],
    \]
    where $d1_i = \epsilon_{i-1}$ for $i \geq 1$.
    The morphism $HH(\kk[t]/\kk) \to \kk[t]$ lifts to the inclusion $i_0$ into $(\kk[t]1_i \oplus
    \kk[t]\epsilon_i[1])$ for $i=0$, and the map $\kk[t] \to \kk[t][2]$ is the cone on this map. We
    see from the resolution above that 
    \begin{equation}
    \label{eq:computed-endomorphisms-over-hochschild-homology}	
    Hom_{\kk[x] \oplus \kk[x][1]}(\kk[x], \kk[x][2]) = \kk[x][[u]][2] 
    \end{equation}
	where $u$ has homological degree $-2$, just as in the
        computation of $Hom_{\kk[B]}(\kk, \kk) = \kk[[u]]$; and just
        as in that computation, this morphism is given by the element
        $u \neq 0$ in
        \eqref{eq:computed-endomorphisms-over-hochschild-homology}.
        
    \item To compute the map $ud$ and thus justify our notation for it, we first note that 
    \[ Hom_{\kk}^{S^1}(\kk[x], \kk[x][2]) = Hom_{\kk}(\kk[x], \kk[x])[[u]][2] \]
    again by the two periodic resolution of $\kk[\epsilon]\to \kk$;
    tensoring the top row of \eqref{eq:differential-forms-over-k} with
    $\Q$ gives us the same diagram with $\kk$ replaced by $\kk \tensor
    \Q$; so long as this is nonzero, we can then use the cyclic HKR
    theorem~\cite{toen} to compute $ud \tensor \Q$, which tells us
    that $ud$ is given precisely by that quantity in $Hom_{\kk \tensor
      \Q}(\kk \tensor \Q[x], \kk \tensor \Q[x])[[u]][2]$.  Finally, we observe that
    \[ Hom_{\kk}^{S^1}(\kk[x], \kk[x][2]) \to Hom_{\kk \tensor \Q}(\kk \tensor \Q[x], \kk \tensor
    \Q[x])[[u]][2]\]
    is injective by the previous equation and that $ud$ lies in its image.

  \item In particular, this computation together with Lemma \ref{lemma:tS^1-tC_p-Z-mod} shows that
  the map $R^{tC_p} \to (\Omega^1_{R/\kk})^{tC_p}$ is given by $ud$ on homotopy groups.
    \end{itemize}

Putting this all together, we have shown that
diagram~\eqref{eq:differential-forms-over-k} is equivalent to
diagram~\eqref{eq:differential-forms-on-Hk}.  All the same arguments
hold with $\kk[x]$ replaced throughout by $\kk[x,x^{-1}]$.

\subsection{$p$-fold covers on the loop space and the Cartier map.}
\label{eq:p-fold-covers-on-free-loop-space}

We now specialize to $\kk = \Z$.  We will compute the map $F'$ in
\eqref{eq:differential-forms-over-k} on homotopy groups for $R =
\kk[x,x^{-1}]$. This reduces to computing $\pi_*(\phi)$.  We use
the HKR identification, which exactly corresponds to  
\[
H_*(S^1_k, \Z) = \Z x^k \oplus \Z x^k \frac{dx}{x}[1].
\]
 On the Tate construction, we identify
  \[ \pi_*((S^1_k \sma \Z)^{tC_p}) = \pi_*((S^1_k \sma \Z_p) \sma_\Z \Z^{tC_p}) = (\F_p x^k \oplus
  \F_p x^k \frac{dx}{x}[1])((u)) \text{ for }p \mid k \]
 where we are using that the action is trivial and the map $\Z \to \Z^{tC_p}$ in the tensor product
 is the canonical map, and for $p \nmid k$,
\[ (S^1_k \sma \Z)^{tC_p} = (S^1_k \sma \F_p)^{tS^1} = (S^1_k/C_p
\sma \F_p)^{tS^1/C_p} = 0 \] 
 where we use that the $S^1$-action on $S^1_k$ has stabilizer $C_k$ with $|C_k|$ invertible in
 $\F_p$, so equivariant cohomology agrees with the residual cohomology of the quotient, and
 subsequently that the residual action is free and the equivariant cohomology is $u$-torsion and
 thus goes to zero upon taking the Tate construction. Under these identifications we see that
 \begin{equation}
 \label{eq:cartier-operation}
 \phi(x^k) = x^{pk}, \phi\left(x^k \frac{dx}{x}\right) = x^{pk} \frac{dx}{x}. 
 \end{equation}
 When $\kk = \F_p$, the computation is similar (by Lemma \ref{lemma:tS^1-tC_p-F_p-mod}), and in that
 case $\phi$ and thus $F'$ are \emph{equivalences}. 
 In particular we conclude that $F'(x^k dx/x) = x^{pk} dx/x$. We note that when $\kk = \F_p$, these
 are precisely the differential forms representing cohomology classes of $H^{dR}(\kk[x]/\kk)$, and
 under this isomorphism, $\pi_*(\phi)$ agrees with the \emph{Cartier isomorphism}
 \[
 C^{-1}: \Omega^*_{\kk[x,x^{-1}]/\kk} \simeq H^*_{dR}(\kk[x,
   x^{-1}]/\kk).
 \]
 The computations with $R=\kk[x]$ are nearly identical. In particular,
 putting together this computation, the \emph{Cartier formula} of
 equation~\eqref{eq:classical-bk-formula}, and the comparison of
 Proposition \ref{prop:identify-domain-of-relative-map-in-case-1},
 shows that the map $(\cap^p_{\F_p})^{tC_p}$ appearing at the bottom
 of \eqref{eq:relative-bk-formula} agrees on homotopy groups
 \emph{precisely} with $\iota^{[p]}$, in the case when $A = \F_p[x]$ (or
 $R = \F_p[x, x^{-1}]$) and $R = \F_p$.  
 
\begin{remark}
In particular, when $R = \F_p[x,x^{-1}] \simeq WFuk(\C, \F_p)$ is the
Fukaya category with $\F_p$ coefficients of $\C^* = T^*S^1$, by
identifying $(\cap^p_{\F_p})^{tC_p}$ with the equivariant pants
product on symplectic cohomology and subsequently with an equivariant
string topology operation \cite{abbondandolo2006floer} on the homology
of the free loop space of $S^1$, we would get a string topology
interpretation of Cartier's formula
\eqref{eq:classical-bk-formula}. It would be interesting 
to calculate these operations on the free loop spaces of interesting
manifolds such as compact Lie groups (whose free loop space homologies
are closely related to the homologies of the corresponding affine
Grassmannians).
\end{remark}

\subsection{Kodaira-Spencer classes and connections from the circle action}
\label{sec:kodaira-spencer-classes-and-connections}

Let $A$ be a smooth and proper dg category over $R = \kk[x], \kk[x,x^{-1}]$. Then $HH(A/\kk)$ is an
$HH(R/\kk)$-module, and we have corresponding morphisms of $R$-modules with $S^1$-action 
\begin{equation}
    \label{eq:our-connection}
    HH(A/\kk) \tensor_{HH(R/\kk)} \kk \simeq HH(A/R) \xrightarrow{ 1 \tensor ud} HH(A/R) \tensor_{R}
    \Omega^1_{R/\kk}[2]. 
\end{equation} 
Even when we forget the $S^1$-action, this is often a nontrivial morphism of $R$-modules because
$ud$ was nontrivial as a map of $HH(R/\kk)$-modules. Indeed, we will see that this map turns out to
be the cap product with the \emph{Kodaira-Spencer class} 
\[\kappa \in HH^2(A/R), \]
also known as the ``Kaledin class'' (or the ``Borman-Sheridan class'') in the literature on
symplectic topology, associated to the family of categories $A/R$.  

Taking Tate fixed points, we get a map 
\[ HP(A/R) \to HP(A/R)[2].\] 
On homotopy groups, this turns out to be a $u$-connection, i.e.,
an operator satisfying the $u$-Leibniz rule, and we will show that
when $\kk$ has characteristic zero that this map agrees on homotopy
groups with the Getzler-Gauss-Manin (GGM) connection
\cite{getzler1993cartan}, which is defined in terms of explicit
formulae on the cyclic bar complex.  

This comparison is important because in the setting we are interested
in, when  $A = Fuk(M)$ is a Fukaya category, the GGM connection of $A$
can be compared with the quantum connection of $M$ on quantum
cohomology. However, in the symplectic setting, we often are forced to
choose $R$ to be a \emph{Novikov ring}, a typical example of which is
$R = \kk((x))$. There is no problem with the existence of the GGM
connection, since it is defined in terms of explicit formulae
(e.g.,~\cite{getzler1993cartan, sheridan2020formulae}) that carry over
to this setting.  However, we cannot copy the definition of the map
$ud$ here, because 
\[
HH_1(\kk((x))/\kk) = \Omega^1_{\kk((x))/\kk}
\]
is infinitely generated over $\kk((x))$
\cite[5.5]{kunz1986kahler}. Thus, the earlier discussion about the HKR
theorem implies that in characteristic zero, $HH_*(\kk((x))/\kk)$ is
nonzero in infinitely many degrees, and we lack the analog of the
cofiber sequence \eqref{eq:differential-forms-over-k}.

Instead of considering the module of K\"ahler differentials, it is
natural here to consider instead its $x$-adically continuous analogs 
\[
\hat{\Omega}^1_{\kk((x))/\kk} = \kk((x)) \, dx,
\; \quad \hat{\Omega}^1_{\kk[[x]]} = \kk[[x]] \, dx.
\]
When $R$ is finite type over $\kk$ we will continue to let
$\hat{\Omega}^1_{R/\kk}$ denote its usual module of K\"ahler
differentials.

\begin{theorem}
\label{thm:general-calculus-diagram}
	Let $R = \kk[x], \kk[x,x^{-1}], \kk[[x]], $ or $\kk((x))$. There is a diagram 
	\begin{equation}
	\label{eq:general-calculus-diagram}
	\begin{tikzcd}
	HH(R/\kk) \ar[r] \ar[d, "\phi \tensor \kk"] & R \ar[r] \ar[d, "F"] & \hat{\Omega}^1_{R/\kk}[2]
	\ar[d, "F'"]\\
	HH(R/\kk)^{tC_p} \ar[r] & R^{tC_p} \ar[r] & (\hat{\Omega}^1_{R/\kk})^{tC_p}[2]
	\end{tikzcd}
	\end{equation}
	where the left square is a diagram of
        $\EE_\infty$-$S^1$-algebras induced from the cyclotomic
        structure on relative $THH$, the top row is a diagram of
        $HH(R/\kk)-S^1$-module (but not a fiber sequence in general),
        the bottom row is $tC_p$ of the top row, and the vertical maps
        give a map of $S^1$-modules linear over the map $\phi \tensor
        id_\kk$. Moreover, the map $F'$, which is $R$-linear via $R
        \to HH(R/\kk)$, agrees with the composition given
        by~\eqref{eq:cartier-operation} on $\hat{\Omega}^1_{R/\kk}[2]$
        composed with the canonical map to its Tate fixed points; in
        particular, it is $F$-linear.
\end{theorem}
	
\begin{proof} (Proof of Theorem \ref{thm:general-calculus-diagram} for $R = \kk[x]$,
$\kk[x,x^{-1}]$)
Given the discussion of the previous sections, all that remains to show is the claimed factorization
of $F$. But $\hat{\Omega}^1_{\kk[[x]]}$ is a free $R$-module, so maps of $R$-modules from
$\hat{\Omega}^1_{\kk[[x]]}$ are determined by their action on a generator $dx \in
\pi_0(\hat{\Omega}^1_{\kk[[x]]})$; the fact that there are no $u, \theta$ terms in $\pi_*(F)(dx)$
proves the claim.
\end{proof}
The rest of the proof of Theorem \ref{thm:general-calculus-diagram} in  Section
\ref{sec:cyclotomic-bases-exist}. The following theorem is proven in Section
\ref{sec:comparing-connections}:
\begin{theorem}
\label{thm:ggm-comparison}
	Let $R = \kk[x], \kk[x,x^{-1}], \kk[[x]], $ or $\kk((x))$, let
        $A$ be a smooth and proper dg category over $R$, and suppose that $\pi_*(A/R)$ is
        torsion-free.	Let
        \[
        u\Grad: HH(A/\kk) \tensor_{HH(R/\kk)} \kk \simeq HH(A/R)
        \xrightarrow{ 1 \tensor ud} HH(A/R) \tensor_{R}
        (\hat{\Omega}^1_{R/\kk})[2]
        \]
	be the $R-S^1$-linear map induced from the top row of
        \eqref{eq:general-calculus-diagram}. Then 
	\[
        (u\Grad)^{tS^1}: HP(A/R) \to HP(R) \tensor_R \hat{\Omega}[2]
        \]
	 agrees on homotopy groups with the Getzler-Gauss-Manin connection. 
\end{theorem}

\begin{remark}
We will continue to write $\kappa'$ for the nonequivariant map underlying $u\Grad$.\end{remark}

We can now give the promised interpretation in terms of the
Kodaira-Spencer class $\kappa$.

\begin{lemma}\label{lemma:identify-kodaira-spencer}
In the situation of Theorem~\ref{thm:ggm-comparison}, suppose that $A$
lifts to a smooth and proper dg category over $\tilde{R}[1/N]$ where
$\tilde{R}$ is the corresponding spherical lift $\tilde{R}$ of $R$
described in Proposition~\ref{prop:weak-cyclotomic-bases}.  Then the map~\eqref{eq:our-connection}
$\kappa'$, as a non-equivariant map, is given on homotopy groups by
    \[ \pi_*(\kappa') = - \pi_*(e_\kappa),\] 
where $e_\kappa$ is the cap product with the Kodaira-Spencer class $\kappa$ as defined
in~\cite{petrov2018gauss, seidel2018connections, sheridan2020formulae}.  
\end{lemma}

\begin{proof}
The noncommutative Hodge-de-Rham spectral sequence degenerates.  That
is, writing $Gr_u$ for the associated graded of the filtration on Tate
cohomology, we have canonical isomorphisms 
\[
Gr_uHC^-_*(A/R) \simeq HH_*(A/R)[[u]]\langle \theta \rangle
\]
    inducing
    \begin{equation}
    \label{eq:hdr-degen}
        Gr_uHP_*(A/R) \simeq HH_*(A/R)((u))\langle \theta\rangle
    \end{equation}  
Now we note that the Getzler $u$-connection has, on the chain level, $u=0$ term equal to $-e_\kappa$
\cite{petrov2018gauss}. Thus the induced map on the associated graded of the filtration is exactly
$-e_\kappa$. The earlier comparison means that this agrees with the induced map of the associated
graded for our connection, which is precisely the map $\kappa'$. 
\end{proof}

\begin{remark}
The additional difficulty when working with the rings $R = \kk[[x]],
\kk((x))$ is that whenever one takes a tensor product of $R$-modules
over $\kk$, one should use the \emph{completed} tensor product (with
respect to the $x$-adic topology) rather than the usual tensor
product.  Homotopically, one should consider the derived completed
tensor product, which is slightly delicate. In particular, when
symplectic topologists work with Fukaya categories over such $R$, they
may (implicitly) complete all tensor products when writing down all
$A_\infty$ formulae. Such computations do \emph{not} result in
complexes that compute ordinary Hochschild (co)homology, but instead
compute a completed variant of Hochschild homology; luckily, due to
the focus on explicit $A_\infty$ formulae in symplectic topology, this
is unlikely to introduce any errors. We clarify the homotopical
meaning of such completions in Appendix~\ref{sec:solid-spectra}.
\end{remark}

\begin{remark}
Geometrically, while many Fukaya categories are smooth and proper over
Novikov fields like $\kk((x))$, they almost never have extensions to
smooth and proper dg categories over $\kk[[x]]$, nor are they in any straightforward manner the
completions of categories over $\kk[x,x^{-1}]$ or other bases of finite type. The reason for the
appearance of rings like $\kk((x))$ is two-fold: first, infinitely many powers of $x$ are required
to keep track of infinitely many homology classes of holomorphic curves, and second, Floer homology
of Lagrangian submanifolds is Hamiltonian isotopy invariant when defined over a Novikov \emph{field}
like $\kk((x))$, but not over a Novikov ring, where the $x$-torsion in Floer homology captures
interesting quantitative and dynamical information. In important mirror symmetry examples, the
reason the Fukaya categories do not extend to smooth proper categories over $\kk[[x]]$ can be seen
on the mirror side, where the mirror family of varieties over $\kk((x))$ extends to $\kk[[x]]$ with
a singular fiber over $x=0$. This singular behavior is also intrinsically visible on the symplectic
side, as in Remark \ref{rk:quantum-connection-pole}.
\end{remark}

\section{Reprise of Petrov-Vaintrob-Vologodsky}
\label{sec:pvv}

In this section, we prove a variant of the main result of
Petrov-Vaintrob-Vologodsky \cite{petrov2018gauss} using the map
$\phi_R$ for $R = \F_p[x], \F_p[x,x^{-1}], \F_p[[x]], \F_p((x))$
instead of Kaledin's noncommutative Cartier isomorphism. This is
necessary primarily because we do not know how to identify the maps
$\phi_R$ with Kaledin's maps, which are defined purely
algebraically. Some comments on potential methods for comparison can
be found in~\cite{2004.04279}, but no comparison is available in the
literature.  In any case, we hope that a spectral argument for a
variant of the results of~\cite{petrov2018gauss} is instructive. 

\subsection{The inverse Cartier transform.}
The following is standard and is adapted from \cite[1.1]{petrov2018gauss}; in this section we avoid
the language of Higgs modules and simply make the contents of the cited section completely explicit
in our case of interest. We hope that the naively concrete nature of this section is helpful to
readers less familiar with arithmetic geometry.

Let $R$ be one of $\F_p[x]$, $\F_p[x, x^{-1}]$, $\F_p[[x]]$, or $\F_p((x))$. Let $F$ denote the
Frobenius map on $R$; we will write $F^*$ for the map corresponding to pullback along the Frobenius
thought of as a map of schemes; thus, given a module $M$ over $R$, we have that $F^*M = M \tensor_R
R$ where the $R$-module structure on the right is given by composition with $F$. A map of
$R$-modules $\theta: M \to M \tensor_R \hat{\Omega}^1_R$ gives $M$ the structure of a \emph{Higgs
module} over $R$, that is, the antisymmetrization of the square of the map $\theta^2: R \to R
\tensor \hat{\Omega}^2_{R/\F_p}$ is the zero map (since for $1$ parameter this latter condition is
tautological). Given a Higgs module $(M, \theta)$, we can define the structure of a module with
integrable connection on $F^*(M)$, by setting 
\[ \Grad^{F^*M} =  \Grad^{can} + C^{-1}_{\tilde{R}, \tilde{F}}(\theta). \]
Here 
\[ \Grad^{can}: F^*M \to F^*M \tensor_R \Omega^1_R\]
is the canonical connection on $F^*M$, i.e. the unique connection which annihilates all sections of
$F^*M$ pulled back via $F$. If $M$ is a free module over $R$, then so is $F^*(M)$; choosing a
trivialization $M = R^m$ we have that $F^*M$ is canonically isomorphic to $R^m$ as well, with
pullback of sections $M \to F^*M$ given by $(f_1, \ldots, f_k) \mapsto (f_1^p, \ldots, f_k^p)$. Thus
it is clear that in this trivialization $\Grad^{can}_{\partial/\partial x} = \partial_x$, since this
is exactly the operator that annihilates $p$-th powers of functions. 

To define $C^{-1}_{\tilde{R}, \tilde{F}}(\theta)$, recall that we have
standard liftings to $\tilde{R}$ of $R$ given by replacing $\F_p$ with
$W_2(\F_p)$ in the definitions, and we lift the Frobenius to
$\tilde{R}$ by taking $\tilde{F}(x) = x^p$, with $\tilde{F}$ acting
trivially on $W_2(\F_p)$.  Writing $\hat{\Omega}^1_{\tilde{R}}$ by
replacing $\F_p$ with $W_2(\F_p)$ in the definitions, we have:
\begin{align*}
C^{-1}_{\tilde{R}, \tilde{F}}\colon End(M) \tensor_R \hat{\Omega}^1_R
\to End_R(F^*(E)) \tensor_R \hat{\Omega}^1_R
\\
C^{-1}_{\tilde{R},
  \tilde{F}}(\tilde{\theta} \tensor \alpha) = F^*(\tilde{\theta})
\tensor \frac{1}{p}\tilde{F}^*\tilde{\alpha}
\end{align*}
with $\tilde{\alpha}$ a lift to $\hat{\Omega}^1_{\tilde{R}}$ of
$\alpha$ in $\hat{\Omega}^1_R$.

In our setting, in a given trivialization $M = R^m$, we have that $\theta$ is given by the matrix of
$1$-forms $A \;dx$ for some  $A \in Mat_{m \times m}(R)$, and
\[ C^{-1}_{\tilde{R}, \tilde{F}}(\theta) = F^*(A) x^{p-1} dx, \text{ i.e. }
\Grad^{F^*M}_{\partial_x} = \partial_x + F^*(A) x^{p-1}. \]
In general, $C^{-1}_{\tilde{R}, \tilde{F}}$ manifestly depends on the choice of lifts $\tilde{R}$
and $\tilde{F}$. 

From the previous formula, we note the following helpful and trivial observation.

\begin{lemma}
\label{lemma:characterization-of-cartier-pullback-connections}
Assume that $M$ is free. A connection on $F^*M$ is of the above form
if for any pulled back section $F^*g$ we have  
    \[ \Grad^{F^*M}_{\partial x}F^*g = F^*(A)x^{p-1}F^*g\]
for an $R$-linear map $A\colon M \to M$. Note that $F^*(A)$ is just the matrix of $p$-th powers
of the matrix entries of $A$. 
\end{lemma}

Now, Theorem 2.8 of \cite{ogus-vologodsky} computes the $p$-curvature of the inverse Cartier
transform of a Higgs module $(M, \theta)$ via the following formula. 
\begin{lemma}
    \label{lemma:ogus-vologodsky}
    Suppose that $(\theta(v))^p = 0$ for every vector field $v$ on $R$. The $p$-curvature of the
    connection on $F^*M$, when viewed as a map
    \[ \psi: F^*M \to F^*M \tensor_R F^*\hat{\Omega}^1_R \simeq F^*(M \tensor \hat{\Omega}^1_R);\]
is given by 
\[ \psi =  -F^*(\theta). \] 
\end{lemma}
 Concretely, in our case, this implies that for $v \in F^*M \simeq R^k$, writing $\Grad =
 \Grad^{F^*M}$, we have
\begin{equation}
    \label{eq:explicit-p-curvature-formula}
    \psi(v \tensor (1 \tensor \partial_x)) = \Grad^p_{\partial_x}v - \Grad_{\partial_x^p}v =
    \Grad^p_{\partial_x}v = -F^*(A)v. 
\end{equation}
We give an elementary proof of the formula above, which makes it clear
how this works in the formal power series setting:

\begin{proof}[Proof of Lemma \ref{lemma:ogus-vologodsky}]
One has that $\partial_x^p =0$ as a derivation by an elementary
computation. We thus need to explicitly compute the operator 
\[
v \mapsto (\partial_x + F^*(A) x^{p-1})^pv.
\]
We have that $\partial_x^p = (F^*(A)x^{p-1})^p = 0$ because the assumption that $(\theta(v))^p=0$
means that $A^p=0$.  Thus, the operator is given in terms of Lie
polynomials of the operators $A, B$, with $A = \partial_x$ and $B =
F^*(A) x^{p-1}$. We know that
\[
  [A,A] = [B,B] = [B, [B, A]] = 0.
  \]
This implies that the only nonzero possible Lie polynomial is $[A,
  \ldots, [A, B]]$ with $p-1$ copies of $A$; but we have also that
$[\partial_x, F^*(A)] = 0$ since $F^*(A)$ is a matrix of $p$-th
powers; thus, this Lie polynomial simply has the value $F^*A$.  This
polynomial only arises when computing $(\partial_x)^{p-1}$ through
$B$, thus explaining the sign $(-1)=(-1)^p$ multiplying the Lie
polynomial.
\end{proof}

\subsection{The cyclotomic structure and the Cartier transform}
\label{sec:proof-of-pvv}

In this section we relate the Cartier transform to the cyclotomic
structure on the spectral lifts of the rings we are studying.  We
prove the following theorem.

\begin{theorem}\label{thm:pvv}
Suppose $p > N$ is a prime. 
Let $\tilde{R}$ be one of the standard cyclotomic bases of Proposition
\ref{prop:weak-cyclotomic-bases}. Let $\tilde{A}$ be a smooth and proper $\EE_1$-algebra over
$\tilde{R}$, 
 and let $R = \tilde{R} \tensor_\SS \F_p$, $A = \tilde{A} \tensor_\SS \F_p$. 
 Suppose moreover that $\pi_* (HP^-(\tilde{A} \tensor \Z_p/\tilde{R} \tensor \Z_p))$ is
 $p$-torsion-free.

There is a commutative diagram 
\begin{equation}
    \label{eq:frobenius-intertwiner-polynomial-case}
\begin{tikzcd}
    F^*HH_*(A/R)((u)) \langle \theta \rangle \ar[r] \ar[d] & F^*HH_*(A/R) \tensor_R \hat{\Omega}^1_R
    ((u)) \langle \theta \rangle[2]  \ar[d] \\
    \pi_*(HH_\bullet(A/R))^{tC_p} \ar[r] & \pi_*(HH_\bullet(A/R))^{tC_p}\tensor_R \hat{\Omega}^1_R
\end{tikzcd}
\end{equation}
where the top map is the $u, \theta$-linear extension of the  inverse
Cartier transform of the Higgs module with  $\theta = \pi_*(\kappa')$,
and the vertical maps are the equivalences of
$\pi_*(\F_p^{tC_p})$-modules on $\pi_*(R^{tC_p})$-modules given by
$\phi_R$ and $\phi_R \tensor_R id_{\hat{\Omega}^1_{R/\F_p}}$ of
Propositions~\ref{prop:phi-is-an-equivalence-for-cyclotomic-base} and
\ref{prop:identify-domain-of-relative-map-over-F_p[t]}. 
\end{theorem}
 
This result then immediately proves the main algebraic comparison
result of the paper: 

\begin{proof}[Proof of Theorem \ref{thm:algebraic-p-curvature}]	
By Theorem \ref{thm:ggm-comparison} and Lemma
\ref{lemma:tS^1-tC_p-F_p-mod}, the bottom horizontal map of
\eqref{eq:frobenius-intertwiner-polynomial-case} above is the
$\theta$-linear extension of the GGM connection. So the above diagram
is an isomorphism of connections. Thus the vertical maps of
\eqref{eq:frobenius-intertwiner-polynomial-case} must send the
$p$-curvature of the inverse Cartier transform connection to the
$p$-curvature of the Getzler-Gauss-Manin connection. But the
$p$-curvature of the inverse Cartier transform connection is
$-F^*(\theta) = F^*(e_\kappa)$, by Lemmas \ref{lemma:ogus-vologodsky}
and \ref{lemma:identify-kodaira-spencer}. We conclude by Theorem
\ref{thm:nice-relative-bk-formula-over-Fp-ring}.
\end{proof}

We now return to prove Theorem~\ref{thm:pvv}.

\begin{proof}[Proof of Theorem \ref{thm:pvv}]
	Let's write $R_{\Z_p} = \tilde{R} \sma \Z_p$ and $A_{\Z_p} = \tilde{A} \sma \Z_p$; we note that
	$R_{\Z_p}$ is $R = R_{\F_p}$ with $\F_p$ coefficients replaced by $\Z_p$-coefficients due to Lemma
	\ref{lemma:lawson-lemma}. There is a commutative diagram \eqref{eq:comparing-calculus-diagrams}
	\begin{equation}
    \label{eq:comparing-calculus-diagrams}
    \begin{tikzcd} 
        HH(R/\F_p) \ar[r] \ar[bend left=30, color=blue]{dddd}{} & R \ar[r, "ud"] \ar[bend left=30,
        color=blue]{dddd}{} & \Omega^1_{R/\F_p}[2] \ar[bend left=30, color=blue]{dddd}{} \\
        THH(\tilde{R}) \tensor \Z_p = HH(R_{\Z_p}/\Z_p) \ar[d]\ar[r] \ar[u] & 
        R_{\Z_p} \ar[r, "ud"] \ar[d] \ar[u]& 
        \hat{\Omega}^1_{R_{\Z_p}/\Z_p}[2]  \ar[dd] \ar[u]\\
        THH(\tilde{R}) \tensor THH(\F_p) = THH(R) \ar[r]\ar[d] & 
        \tilde{R} \tensor THH(\F_p)  \ar[d]& 
        \ar[d] \\
        THH(R)^{tC_p} = HH(R_{\Z_p}/\Z_p)^{tC_p} \ar[r] \ar[d]
        & R_{\Z_p}^{tC_p} \ar[r] \ar[d]& 
        (\hat{\Omega}^1_{R_{\Z_p}/\Z_p}[2])^{tC_p} \ar[d]\\
        HH(R/\F_p)^{tC_p} \ar[r] 
        & R^{tC_p} \ar[r] & 
        (\hat{\Omega}^1_{R/\F_p}[2])^{tC_p}
    \end{tikzcd}
\end{equation}
the existence of which follows from Theorem \ref{thm:general-calculus-diagram} and the diagrams
\eqref{eq:mix-fp-and-s[t]} and \eqref{eq:zp-to-fp-relative} in the following manner. The cited
statements manifestly construct the middle three rows; then the top row is the pushout along $\Z_p
\to \F_p$ of the second row, the bottom row is the pushout along $\Z_p^{tC_p} \to \F_p^{tC_p}$ of the
fourth row, and the commutativity of the rectangle in \eqref{eq:fundamental-diagram} consisting of
the second and fourth columns of that equation induces the blue arrows in
\eqref{eq:comparing-calculus-diagrams}.

\begin{remark}
	In general, the rows in \eqref{eq:comparing-calculus-diagrams} are not exact triangles in the
	settings where $R$ has been $x$-adically completed, but this does not affect the argument given.
	It's also worth noting that $\hat{\Omega}^1_{R/\kk}$ is free of rank $1$ over $R_\kk$, thus
	explaining why the derived tensor products in the rest of the argument do not need to be completed
	in any sense.
\end{remark}

Now, the domain of $\phi_R$ \eqref{eq:relative-cyclotomic-structure} is $THH(A) \tensor_{THH(R)}
R^{tC_p}$, where the map $THH(R) \to R^{tC_p}$ is the map obtained by composition of arrows in
\eqref{eq:relative-cyclotomic-structure} from the middle left to the middle bottom object; we note
that all maps in the left rectangles of \eqref{eq:comparing-calculus-diagrams} are maps of
$\EE_\infty$-algebras with homotopy $S^1$-action. Using this, we construct a commutative diagram
\begin{equation}
    \label{eq:fundamental-commutation}
    \begin{tikzcd}
    HH(A_{\Z_p}/R_{\Z_p}) \tensor_{R_{\Z_p}} R_{\F_p}^{tC_p} \ar[r]\ar[d, "\sim"] &
    HH(A_{\Z_p}/R_{\Z_p}) \tensor_{R_{\Z_p}} (\hat{\Omega}^1_{R/\F_p}[2])^{tC_p}, \ar[d, "\sim"]\\
    THH(A) \tensor_{THH(R)} R^{tC_p} \ar[r] \ar[d, "\phi_R"] &
    THH(A) \tensor_{THH(R)} (\hat{\Omega}^1_{R/\F_p}[2])^{tC_p} \ar[d] \\
    HH(A/R)^{tC_p} \ar[r]& \left(HH(A/R)\tensor_{R} \hat{\Omega}^1_{R/\F_p}[2]\right)^{tC_p}. 
    \end{tikzcd}
\end{equation}
The third horizontal map in \eqref{eq:fundamental-commutation} is obtained  by applying the map
$R^{tC_p} \to (\Omega^1_{R/\F_p}[2])^{tC_p}$ of \eqref{eq:comparing-calculus-diagrams}, which is a
map of $THH(R)$-modules, to $THH(A) \tensor_{THH(R)} R^{tC_p}$. The outer square (consisting of rows
2 and 4 mapping to rows 1 and 5) of the diagram \eqref{eq:comparing-calculus-diagrams}, together
with the factorization claim about $F'$ in Theorem \ref{thm:general-calculus-diagram}, lets us
identify the top row of  \eqref{eq:fundamental-commutation} with the map 
\[ (1 \tensor F' \tensor 1) \circ F^*(\kappa') \tensor 1: F^*HH(A/R) \tensor_{\F_p} \F_p^{tC_p} \to
F^*HH(A/R) \tensor_{R} \hat{\Omega}^1_{R/\F_p} \tensor_{\F_p} \F_p^{tC_p}. \]
The same argument lets us identify vertical maps of \eqref{eq:fundamental-commutation} with the
claimed form of the vertical maps in Theorem \ref{thm:pvv}. Again, we know in this situation that
the bottom map is a connection on homotopy groups and these vertical maps are isomorphisms on
homotopy groups; so the top map is a connection as well.   
Lemma \ref{lemma:characterization-of-cartier-pullback-connections} and the fact that $F'$ acts by
\eqref{eq:cartier-operation} then establishes the claim that the map on the top of
\eqref{eq:frobenius-intertwiner-polynomial-case} is the Cartier transform of the Higgs module with
$\theta = \pi_*(\kappa')$.
\end{proof}

\section{Lifting to the Sphere}
\label{sec:lifting-to-the-sphere}
The author thanks Vadim Vologodsky and Sasha Petrov for the idea of the following argument.
\begin{theorem}
\label{thm:lifting-to-the-sphere}
    Let $A$ be an $\EE_1$-algebra over $\Z[1/N]$. Suppose that $HH^j_{\Z[1/N]}(A)=0$ for $j \geq
    q+1$ and $\pi_{-j}(A)=0$ for $j\leq q+1$, and $N\geq q/2+1$. Then there is an $\EE_1$-algebra
    $\tilde{A}$ over $\SS[1/N]$ such that $\tilde{A} \tensor_{\SS[1/N]} \Z[1/N] \simeq A$. If this
    holds for $N \geq (q+3)/2$ then $\tilde{A}$ is unique up to equivalence. 
\end{theorem}
\begin{proof}
To be concise, in this proof we write $\varSS = \SS[1/N]$. 

The result will follow from an elementary obstruction theory
computation, together with Serre's fundamental result that
$\pi_i(\SS)\tensor \F_p=0$ for $0 < i < 2p-3$. We will perform a
computation with the $\EE_1$-cotangent complex; we will use
\cite{francis2013tangent} as a useful reference, although the actual
results needed (the case of $\EE_1$-algebras rather than
$\EE_n$-algebras) are originally due to
Basterra-Mandell~\cite{Basterra1999AQ, Mandell2003TAQ,
  BasterraMandell2005} and Lazarev~\cite{lazarev2001homotopy}.

The basic property of the $\EE_1$-cotangent complex $L_A$ is that,
given an $\EE_1$-algebra $A$ in a stable $\infty$-category $\CC$,
square-zero extensions of $A$ by a given $A \tensor A^{op}$-module 
\[ M \to \tilde{A} \to A \]
are in bijection with the set of maps $\Hom_{A \tensor A^{op}}(L_A,
M[1])$~\cite[\S 2]{francis2008derived}.  In particular, there is the
trivial square-zero extension $A \oplus M$, for which the map $A
\oplus M \to A$ is the projection, and thus it admits a section
$s_M$. The cotangent complex is the object which corepresents
derivations, that is, there is an equivalence of functors in $M$ 
\[
\Hom_{A \tensor A^{op}}(L_A, M) \simeq Der(A, M) :=
Map_{\EE_1(\CC)/A}(A, A\oplus M).
\]
Given a map $L_A \to M[1]$ from the cotangent-complex corresponding to
$d\colon A \to M[1]$, the corresponding square zero extension is given
by the fiber product square in $\EE_1(\CC)$ on the right of the
diagram below: 
\begin{equation}
\begin{tikzcd}
M \ar[r] \ar[d, equals]& \tilde{A} \ar[r] \ar[d] & A \ar[d, "d"]\\
M \ar[r] & A \ar[r, "s_{M[1]}"] & A \oplus M[1].
\end{tikzcd}
\end{equation}
Here the horizontal arrows are fiber sequences in $\CC$. 

One sees from this discussion that this correspondence is
\emph{functorial}; in other words, given maps in  
\begin{equation}
\begin{tikzcd}
    M_B \ar[r]\ar[d, "\bar{f}"] & \tilde{B} \ar[r] \ar[d, "\tilde{f}"] & B \ar[d, "f"] \\
    M_A \ar[r] & \tilde{A} \ar[r] & A 
\end{tikzcd}
\end{equation}
where the horizontal rows are square-zero extensions, and the square
on the right commutes in $\EE_1(\CC)$, and the whole diagram commutes
in $\tilde{B} \tensor \tilde{B}^{op}-mod$, we get a diagram in $B
\tensor B^{op}-mod$  
\begin{equation}
    \begin{tikzcd}
        L_{B} \ar[d] \ar[r] & M_B [1] \ar[d, "\Sigma \bar{f}"] \\
        L_A \ar[r] & M_A[1].
    \end{tikzcd}
\end{equation}

We note the following basic fact: if $A \in
\EE_1(\varSS-mod)$ then the unit map $\varSS \to A$ arises from a map
in $\EE_1(\varSS-mod)$, and we can ask if it factors through the
canonical map $\varSS \to \tau_{\leq n} \varSS$ in the same
category. If it does, then, since the left and the right
$\varSS$-module structures on $A$ agree, $A$ arises from a unique
$\EE_1$ algebra in $\tau_{\leq n} \varSS$-mod, which we also denote by
$A$.  

We now give a concise review of basic facts about the cotangent
complex that we use.
\begin{itemize}
    \item The cotangent complex sits in the fiber sequence 
    \begin{equation}
        \label{eq:cotangent complex in terms of Hochschild cohomology}
        L_A \to A \tensor A^{op} \to A \text{ in } A \tensor A^{op}-mod 
    \end{equation}
    and thus maps out of the cotangent complex can be computed in
    terms of Hochschild cohomology groups of $A$~\cite[2.26]{francis2013tangent}.
    
    \item Given a pair of maps $C \to B \to A$ in $\EE_1(\CC)$, one has a fiber sequence of
    cotangent complexes~\cite[2.11]{francis2013tangent}
    \begin{equation}
        \label{eq:relative-cotangent-complex-exact-sequence}
        L_{B/C} \tensor_{B \tensor B^{op}} A \tensor A^{op} \to L_{A/C} \to L_{A/B}.   
        \end{equation} 
 
    \item Combining the two formulae above, we have a fiber sequence
    \begin{equation}
    \label{eq:relative cotangent complex in terms of Hochschild cohomology}
        L_{A/B} \to A \tensor_B A \to A  \text{ in } A \tensor A^{op}-mod.
    \end{equation}
    In particular, if $B$ is $\EE_\infty$ and  the map $B\to A$ is induced by the unit of a lift of
    $A$ to $\EE_1(B-mod)$, then we see that $L_{B/A}$ comes from the corresponding quantity computed
    in terms of $\CC=B-mod$ by pullback under the map $A \tensor A \to A \tensor_B A$.  
\end{itemize}

Using these basic facts about the cotangent complex, we can establish
the proposition via a computation.  There are square zero extensions of
$\EE_\infty$ (and thus, $\EE_1$-algebras) 

\begin{equation}
    \label{eq:spherical-tower-1}
    \pi_{n+2}(\varSS)[n+1] \to \tau_{\leq n+1} \varSS \to \tau_{\leq n} \varSS
\end{equation}
with $\tau_{\leq 0} = \Z[1/N]$. We will write $J_n$ for the left-most quantity. Thus, given an
$\EE_1$-algebra $A$ over $\Z[1/N]$, we wish to inductively construct square-zero extensions 
\begin{equation}
    \label{eq:extension-tower-1}
    A_n \tensor_{\tau_{\leq n} \varSS} J_n \to A_{n+1} \to A_n
\end{equation}
which lie over \eqref{eq:spherical-tower-1} in the sense that there is
a map of square-zero extensions in $Sp$ from
\eqref{eq:spherical-tower-1} to \eqref{eq:extension-tower-1}.  Fixing
$A_n$, the existence of $A_{n+1}$ corresponds to the existence of the
vertical map on the right such that the right-most triangle in the
diagram below (in $A \tensor_{\varSS} A^{op}$-mod) commutes: 
\begin{equation}
\begin{tikzcd}
L_{A/\tau_{\leq n} \varSS} [-1] \ar[r]\ar[rrd]&  
L_{\tau_{\leq n} \varSS} \tensor_{\tau_{\leq n} \varSS \tensor \tau_{\leq n} \varSS^{op}} A_n
\tensor A_{n}^{op} \ar[r] \ar[rd] & 
L_{A_n} \ar[d] \\
&&
J[1] \tensor_{\tau_{\leq n}\varSS} A_n.
\end{tikzcd}
\end{equation}
Here the top right map is the map induced from the map of square-zero
extensions, the middle diagonal map is induced by base extension from
the map $L_{\tau_{\leq n} \varSS} \to J_n[1]$ classifying the
square-zero extension \eqref{eq:spherical-tower-1}, and the top row is
a fiber sequence by
\eqref{eq:relative-cotangent-complex-exact-sequence}. The obstruction
to the right vertical map existing is thus the map $L_{A/\tau_{\leq n}
  \varSS} [-1]  \to J[1] \tensor_{\tau_{\leq n} \varSS} A_n$.
Via \eqref{eq:relative cotangent complex in terms of Hochschild
  cohomology} and using the tensor-hom adjunction, the obstruction is
an element of $\pi_0$ of the fiber of 
\begin{equation}
\label{eq:computing-the-fiber}
    HH^\bullet_{\tau_{\leq n} \varSS}(A_n, J_n \tensor_{\tau_{\leq n} \varSS} A[3]) \to J_n
    \tensor_{\tau_{\leq n} \varSS} A_n[3].
\end{equation}
By adjunction we have that 
\[ HH^\bullet_{\tau_{\leq n} \varSS}(A_n, J_n \tensor_{\tau_{\leq n} \varSS} A_n[3]) = HH^\bullet_{\tau_{\leq
n} \varSS}(A_n)[3] \tensor_{\tau_{\leq n} \varSS} J_n = HH^\bullet_{\Z[1/N]}(A_0) \tensor_\Z
\pi_{n+1}(\varSS)[n+4] \]
and similarly 
\[
J_n \tensor_{\tau_{\leq n} \SS} A_n[3] = \pi_{n+1}(\varSS)
  \tensor_\Z A_0.
\]
Above, we used that if $A_n$ is an $\EE_1$-algebra in $\ModC{R}$ for an
$\EE_\infty$-algebra $R$ admitting a map of $\EE_\infty$-algebras $R \to
R_0$, and $A_0 = A \tensor_{R} R_0$, then writing $A$ for the
diagonal bimodule of $A$ we have that  
\[
A \tensor_{A \tensor_{R} A^{op}} (A_0 \tensor_{R_0} A^{op}_0) \htp A_0
\tensor_{R_0} A^{op}_0;
\]
this can be verified by taking the two-sided bar complex resolving $A$
and using the equivalences
\[
A^{\tensor_R n} \tensor_R R_0 \simeq A^{\tensor_R n} \tensor_R
R_0^{\tensor_R n} \simeq (A \tensor_R R_0)^{\tensor_{R_0} n}.
\]
Using that $\pi_0(X[n])= \pi_{-n}(X)$, the fact that under typical
indexing conventions, $\pi_iHH^\bullet(A) = HH^{-i}(A)$, and the Tor
exact sequence we see that for $\pi_0$ of the fiber of
\eqref{eq:computing-the-fiber} to be zero we need $\pi_{-j}(A_0)$ to
have $p$ act invertibly for $j=-(n+4), -(n+5), -(n+3)$, and similarly,
$HH^j_{\Z[1/N]}(A_0)$ should have $p$ act invertibly for $j = n+4,
n+3, n+2$, where $n+1 > 2p-3$. Thus, if we want these conditions to
hold for all $n$, then, if $HH^j_{\Z[1/N]}(A_0)=0$ for $j \geq q+1$
and $\pi_{-j}(A_0)=0$ for $j\leq q+1$, we need to ensure that $p$ acts
invertibly on $HH^j_{\Z[1/N]}(A_0)$ and $\pi_{-j}(A_0)$ for all $j
\geq 0$ and all $p <q/2+1$.   
\end{proof}

Recall that an $\EE_1$-algebra $A$ over a commutative ring spectrum
$R$ is defined to be smooth if $A$ is compact as an $R$-module and
proper if $A$ is compact as an $(A,A)$-bimodule~\cite[Chapter
  11]{lurie2018spectral}.

\begin{proposition}
\label{prop:lifts-are-smooth-and-proper}
In the setting of Theorem \ref{thm:lifting-to-the-sphere}, if $A$ was
smooth and proper over $\Z[1/N]$  then the lift $\tilde{A}$ is smooth
and proper over $\SS[1/N]$.
\end{proposition}

This is an immediate consequence of the following result.

\begin{lemma}
\label{lemma:perfectness-lifts}
Let $\kk$ be a connective $\EE_\infty$-algebra such that all
$\pi_i(\kk)$ are finite modules over $\kk_0 = \pi_0(\kk)$. Let $R$
be an $\EE_1$-algebra over $\kk$ which is compact as a
$\kk$-module, and let $M \in Mod_{R}$ be almost connective, i.e.,
its homotopy groups vanish below some fixed degree $N$. Write
\[
M_0 = M \tensor_{\kk} \kk_0, \quad R_0 = R \tensor_{\kk} \kk_0.
\]
Suppose that $R_0 \in Mod_{R_0}$ is compact. Then $M \in Mod_R$ is
compact.
\end{lemma}

\begin{proof}
We say that a module over a ring spectrum is {\em finitely built} if
it is equivalent to an iterated cone of free modules.  (This is
referred to as a {\em finite cell module} in older literature.)
We will first show that if $M_0$ is finitely built, then $M$ is finitely built. This follows by an
induction, using the fact that a map from a free $R_0$-module is given by a choice of finitely many
elements of $\pi_0(M_0)$, a map from a free $R$-module is given by a choice of finitely many
elements of $\pi_0(M)$, and that if $\pi_i(M)$ is the lowest-degree homotopy group of $M$ then the
map $\pi_i(M) \to \pi_i(M_0)$ is surjective. In fact, this map is an equivalence; this follows by
inspecting the Tor spectral sequence~\cite[IV.4.1]{EKMM}
    \[ Tor_{\pi_*(\kk)}(\pi_*(M), \pi_*(\kk_0)) \Rightarrow \pi_*(M \tensor_{\kk} \kk_0). \]

    Using this, we can find a finitely built $R$-module $M'$ and a map $\alpha: M' \to M$ such that
    $\alpha \tensor_{\kk} \kk_0$ is an equivalence.  Then by 
    Cor 2.6.1.4 of SAG, we have that the $\alpha$ is an equivalence. 

Now, a compact module is a wedge summand of a finitely built
module. Let $p_0 \in \pi_0(End_{R_0}(M_0))$ be a central idempotent
whose image is the given compact module. It suffices to show that this
idempotent lifts uniquely to an idempotent $p \in \pi_0(End_R(M))$, by
the previously cited corollary. But this is an elementary result
(Hensel's lemma) once one shows that the map $\pi_0(End_R(M)) \to
\pi_0(End_{R_0}(M_0))$ is a map of finite algebras over $\pi_0(\kk) =
\kk_0$ (which implies that the power series solution lifting the idempotent
terminates.) This in turns follows from the fact that $M$ is finitely
built over $R$ and $R$ is compact over $\kk$ and the hypothesis on the
homotopy groups of $\kk$.  
\end{proof}

\begin{proof}[Proof of Proposition \ref{prop:lifts-are-smooth-and-proper}]
We apply Lemma \ref{lemma:perfectness-lifts}, first thinking of $A$ as
$M$ and $\SS[1/N]$ as both $\kk$ and $R$.  This shows that $A$ is
compact over $\SS[1/N]$. One then thinks of $A \tensor_{\SS[1/N]}A$ as
$R$ and the cotangent complex as $M$; this shows that this
cotangent complex is compact as a bimodule, which implies that the
diagonal is compact as a bimodule, i.e. $A$ is smooth over $\SS[1/N]$.
\end{proof}

An identical argument shows
\begin{theorem}
\label{thm:lifting-to-sphere-formal}
The statement of Theorem \ref{thm:lifting-to-the-sphere} and
Proposition \ref{prop:lifts-are-smooth-and-proper} holds with
$\Z[1/N]$ replaced by $\Z[1/N]((x))$ and $\SS[1/N]$ replaced by
$\SS[1/N]((x))$.
\end{theorem}

\section{Comparing connections}
\label{sec:comparing-connections}

The purpose of this section is to establish Theorem
\ref{thm:ggm-comparison}, which shows that the standard definition
of the Getzler-Gauss-Manin connection, which the symplectic geometry
literature relies on, and the $u$-connections $(u\Grad)^{tS^1}$
defined in Section~\ref{sec:kodaira-spencer-classes-and-connections},
coincide.

\subsection{Connections and crystals.}

We first recall the `crystalline' description of a connection. Given a
module $M$ over a commutative ring $R$, we write $X = \spec
R$ and $\Delta\colon X \to X \times_k X$ for the diagonal map. We can
consider the universal square-zero extension
\begin{equation}
    \label{eq:universal-sq-zero-extension}
    \Omega^1_R \to R \tensor_k R/I^2 \to R, I = \ker (R \tensor_k R \to R) = (\{x\tensor 1 - 1
    \tensor x\}_{x \in R}).
\end{equation}
Writing $R \tensor_k R/I^2 = R^{(2)}$, we have that $\spec R^{(2)} = X^{(2)} \subset X \tensor_k X$
is the infinitesimal neighborhood of the diagonal, which has projections $p_1, p_2\colon X^{(2)} \to
X$. It is a standard fact that a $k$-connection on a module $M$, i.e., a $k$-linear map 
\[ \Grad: M \to M \tensor_R \Omega^1_R \]
satisfying the Leibniz rule, is the same as the data of an isomorphism 
\begin{equation}
    \label{eq:connection}
    \phi_{12}: p_1^* M \simeq p_2^* M \text{ such that } \Delta^* \phi_{12} = id. 
\end{equation}  
\cite[\href{https://stacks.math.columbia.edu/tag/07J5}{Section 07J5}]{stacks-project}
The connection corresponding to $\phi_{12}$ is given by the formula 
\[ \Grad(s) = p_1^*s - \phi_{12}^{-1}p_2^*s,\]
which lies in $M \tensor_R \Omega^1_R$ by \eqref{eq:connection}.

\begin{remark}
To see the earlier description of connections is true, note first that
the maps $p_1, p_2$ correspond to the maps of rings $R \to R^{(2)}$,
$x \mapsto 1 \tensor x$ or $x \mapsto x \tensor 1$.  Each of these
maps split the exact sequence described earlier, and define
isomorphisms $R^{(2)} \simeq R \oplus \Omega^1_R$. Using the
pullback-pushforward adjunction, we see that \eqref{eq:connection} is
simply the data of the $R$-linear map 
\[
1 \oplus \Grad\colon M \to (p_1)_*p_2^*M \simeq (p_1)_* ((R \oplus
\Omega^1_{R}) \tensor_{p_2, R} M).
\]
The map $\Grad$ satisfies the Leibniz rule 
\[ \Grad(rm) = dr \tensor m + r \Grad M\]
for the following reason: under the isomorphism
\begin{equation}
\label{eq:p1-linear-iso}
(p_1)_* p_2^* R \simeq (p_1)_* R \oplus \Omega^1_R, 
\end{equation} 
the $R$-module structure is given by 
\begin{equation}
\label{eq:r-module-structure-on-p1p2R}
r (x, y dz) = (rx, x \;dr + r y \;dz);
\end{equation} 
given this we see that the $R$-linearity of $1 \oplus \Grad$ means that 
\[
(1 \oplus \Grad)(rm) = r (m, \Grad m) = (rm, dr \tensor m + r \Grad
m).
\]
To see \eqref{eq:r-module-structure-on-p1p2R}, we note that the isomorphism 
\[
p_2^* R = R^{(2)} \simeq R \oplus \Omega^1_R
\]
coming from the splitting associated to $p_2$ of
\eqref{eq:universal-sq-zero-extension} is given by  
\[
x \tensor y = 1 \tensor xy + (x \tensor 1 - 1 \tensor x)(1 \tensor y)
\mapsto (xy, y \;dx),
\]
where we recall that the isomorphism $\Omega^1_R \to I/I^2$ is given
by $dr \mapsto r \tensor 1 - 1 \tensor r$.
We then compute that 
\[
(r \tensor 1)(1 \tensor x) = (1 \tensor rx)+ (r \tensor 1 - 1 \tensor
r)(1 \tensor x),
\]
so under \eqref{eq:p1-linear-iso},
\[
r (x, 0) = (rx, x \;dr),
\]
while 
\[
(r \tensor 1) (z \tensor y - 1 \tensor yz) = (1 \tensor r)(z \tensor y
- 1 \tensor yz)
\]
because we mod out by $I^2$, so under \eqref{eq:p1-linear-iso}, 
\[
r (0, y \;dz) = (0, ry \;dz).
\]
These computations establish \eqref{eq:r-module-structure-on-p1p2R},
and thus show that $\Grad$ satisfies the Leibniz rule, i.e., is a
connection. In particular, this shows how to construct $\phi_{12}$
from a connection $\Grad$ explicitly.
\end{remark}

\begin{remark}
	In the rest of the paper, given $x \in R$ and an $R$-module $M$, an `$x$-connection' always refers
	to an operator $\Grad\colon M \to M \tensor_R \Omega^1_R$ such that $\Grad(fm) = x m \tensor df + f
	\Grad(m)$ for $m \in M, f \in R$. Thus a usual connection is when $x=1$, and when $x$ is invertible
	in $R$, there is a bijection between $x$-connections and connections.
\end{remark}

\subsection{The Getzler-Gauss-Manin connection as a crystal}

Now, let $k$ be a field, and suppose either that $R$ is an Artin
$k$-algebra of finite type, or a formal power series algebra or formal
Laurent series algebra over $k$, or a smooth algebra over $k$. Let $A$
be a $dg$ algebra over such an $R$. In any of these settings, Getzler
\cite{getzler1993cartan} constructed a connection  
\begin{equation}
    \Grad^{GGM}\colon HP_*(A/R) \to HP_*(A/R) \tensor_R \hat{\Omega}^1_R. 
\end{equation}
Now, if $\Q \subset R$, and there is an augmented finite type $R$
algebra $R'$ such that the kernel of the augmentation is nilpotent,
then Goodwillie proved~\cite{goodwillie1985cyclic} 
 that for a dg algebra $A/R'$, the map
\begin{equation}
HP_*(A/R) \to HP_*(A\tensor_{R'} R/R) 
\end{equation}
is an isomorphism. One should think of $\spec R'$ as an infinitesimal
thickening of $R$; the map above is thus a `specialization map', and
the image of the periodic cyclic homology of the fiber category
$HP_*(A\tensor_{R'} R/R)$ under the collapse map  
\[ HP_*(A/R) \to HP_*(A/R') \]
defines a way of canonically spreading out classes in periodic cyclic
homology along infinitesimal extensions of the base.  

There is a characterization of the GGM connection, based on ideas of
Kaledin~\cite{kaledin2010cyclic}, in terms of Goodwillie's map.

\begin{proposition}
\label{prop:ggm-exact-sequence}
Let  $R$ be a finite type $\kk$-algebra, and $A$ a
cofibrant $dg$ algebra over $R$.

There is an exact sequence of $k$-vector spaces
    \[ HP_*(A/k) \to HP_*(A/R) \xrightarrow{\Grad^{GGM}} HP_*(A/R) \tensor_R \Omega^1_R\]
 where the first map is the collapse map, and the second map above is
 induced from the collapse map from the cyclic bar complex over $k$ to
 the cyclic bar complex over $R$.
\end{proposition}

\begin{proof}
The proof of ~\cite[Proposition 3.2]{petrov2018gauss} that
explains that if $\Q \subset \kk$ then the Getzler-Gauss-Manin connection is defined
by the map 
\begin{equation}\label{eq:getzler-crystal}
CP(A/R) \xleftarrow{\sim} CP((p_1)_*p_2^*A/R) \to
CP((p_1)_*p_2^*A/R^{(2)}) \simeq (p_1)_* p_2^* CP(A/R) 
\end{equation}  
where the first equivalence is the Goodwillie isomorphism (taking $R'
= R^{(2)}$), the second map is the collapse map, and the last
isomorphism is elementary.  Indeed the proof of ~\cite[Proposition 3.2]{petrov2018gauss} applies in
all characteristics, and notes that collapse map going left in
\eqref{eq:getzler-crystal} is not an equivalence away from characteristic zero; but
\cite[\S 3.2]{petrov2018gauss} notes that the collapse map factors through
$CP((p_1)_*p_2^*A/R)/I_2$, where $I_j$ is the decreasing filtration on
$CP((p_1)_*p_2^*A/R)$ corresponding to the filtration on the
underlying cyclic object induced by the filtration $\Omega^1_R \subset
R^{(2)}$, and moreover that the resulting morphism is an equivalence, and that the Getzler
connection comes exactly from
composing an inverse to this quasi-isomorphism with the corresponding
map to $(p_1)_*p_2^* CP(A/R)$, which necessarily factors through $I_2$ as well.

By direct inspection, we see that if $\Q \subset \kk$ then the  diagram \eqref{eq:getzler-crystal}
receives a map of
$k$-modules, which is linear with respect to the map $k \to R$, from
the constant diagram at $CP(A/k)$ via the collapse map on the left and
the collapse map followed by the inclusion of $M$ into \[
(p_1)_*p_2^*M \simeq (R \oplus \Omega^1_R) \tensor_R M\] via $m \mapsto 1 \tensor m$.  This immediately
implies the desired claim. For general $\kk$ we replace the second term of
\eqref{eq:getzler-crystal} with its corresponding quotient by $I_2$ as in \cite[\S
3.2]{petrov2018gauss} and the argument goes through verbatim.

\end{proof}

\begin{remark}
In the exact sequence in the statement of
Proposition~\ref{prop:ggm-exact-sequence}, it is not clear from the
proof that the image of $HP_*(A/k)$ is the whole kernel of
$\Grad^{GGM}$.  We will see however that when $A$ is smooth and proper
over $R$ then this holds.  
\end{remark}

\subsection{Proof of Theorem \ref{thm:ggm-comparison} for polynomial
  and Laurent series rings.}

In this section, we assume $R = \kk[x]$ or $\kk[x,x^{-1}]$, and $\Q
\subset \kk$. 

\begin{proposition}
\label{prop:u-connection-in-char-0}
Let $R$ be as above and let $A$ be smooth and proper over $R$.
The map 
\[
([u\Grad])^{hS^1}\colon HC^-(A'/R) \to \Sigma^2 HC^-(A'/R)
\]
is a $u$-connection on homotopy groups. Thus,
$u^{-1}([u\Grad])^{tS^1}$ is a connection on homotopy groups.

\end{proposition}

\begin{proof}
     If $A=R$ then the map is simply given by $ud$ on $R^{hS^1}$ by
     the HKR theorem; thus it is a $u$-connection in this case. Since
     $A$ is smooth and proper, $HC^-(A/R)^{hS^1}$ is a compact
     $R^{hS^1}$-module, and as such, it is built as the image of an
     idempotent of an iterated cone of free $R^{hS^1}$-modules. It
     thus suffices to check that the $u$-Leibniz rule is inherited by
     summands, and if given a fiber sequence of $R^{hS^1}$-modules
     equipped with endomorphisms, if two out of the three
     endomorphisms satisfy the $u$-Leibniz rule then the third one
     does. Both of these statements are elementary diagram chases.
\end{proof}

\begin{proof}[Proof of Theorem \ref{thm:ggm-comparison} for $R$ as above]
We know the Getzler connection vanishes on the image of $HP(A/k)$ in
$HP(A/R)$ by Proposition~\ref{prop:ggm-exact-sequence}. By the fact
that the top row of \eqref{eq:general-calculus-diagram} is a cofiber
sequence in our setting, we have that there is a cofiber sequence  
\[
HP(A/\kk) \rightarrow HP(A/R) \xrightarrow{u^{-1}([u\Grad])^{tS^1}}
HP(A/R) \tensor_R \Omega^1_R.
\]
By the smoothness and properness of our category, $HP_*(A/R)$ is a
projective (and thus free) module over $R$ in every degree; by
Proposition \ref{prop:u-connection-in-char-0} and the exact sequence
above, we must have that   
\[
\dim Im HP_*(A/\kk) = rank_R HP_*(A/R)
\]
in every degree. But then the sections that are flat with respect to
the GGM connection must also be exactly $Im HP_*(A/\kk)$ in every
degree, because the dimension of the space of flat sections is at most
the rank of the ambient module. Finally, two connections on a free
module over $R$ with the same flat sections are the same in characteristic
zero.
\end{proof}

\begin{remark}
In fact Proposition \ref{prop:ggm-exact-sequence} holds without the
characteristic zero assumption, using an idea of Kaledin and
\cite{petrov2018gauss}. We will not spell out the details here,
because our final result is still constrained to the characteristic
zero setting, as one cannot conclude that two connections on an
$R$-module $M$ agree when they have the same flat sections if $M$ has
nontrivial $p$-torsion.
\end{remark}

\subsection{Completed Hochschild Homology}

In this section, $R = \kk[[x]]$ or $\kk((x))$.

As mentioned earlier in Section \ref{sec:kodaira-spencer-classes-and-connections}, when working with
such rings, we are forced to use a completed version of Hochschild homology. For the moment, let us
treat $R$ as a Banach ring over the trivially normed algebra $\kk$, with the norm of a Laurent
series being $|x^rf(x)| = 2^{-r}$ where $f(0) \neq 0$. Given two Banach modules $M$ and $N$ over
$R$, the completed tensor product is the completion of the algebraic tensor product with respect to
the projective cross-norm
\[ M \hat{\tensor}_\kk N := \widehat{M \tensor_\kk N},   \|x\| = \inf_{x = \sum_i m_i \tensor n_i}
\{ \max_i |m_i||n_i|\} \text{ for } x\in M \tensor_\kk N. \]
We have that 
\begin{equation}
\label{eq:elementary-tensor-products}	
 \kk[[x]]\hat{\tensor}_\kk \kk[[y]] = \kk[[x,y]], \kk((x))\hat{\tensor}_\kk \kk((y)) = \kk[[x,
 y]][x^{-1}, y^{-1}]. 
\end{equation}

The completed Hochschild homology $HH^\completed(R/\kk)$ is the geometric realization of the cyclic
commutative algebra in rings over $\kk$ 
\[ \underline{HH}^\completed_n(R/\kk) = R^{\hat{\tensor}_\kk n} \]
with all cyclic structure maps defined by their unique continuous extensions from usual cyclic bar
complex of $R$. 

\begin{remark}
Here, when we discuss cyclic $\kk$-algebras we are \emph{forgetting} the topologies on the rings
$R^{\hat{\tensor}_\kk n}$. Thus their geometric realizations have the structure of
$\EE_\infty$-$S^1$-algebras in $\kk$-modules. This forgetting of the topology is by design, and this
section is written carefully to avoid explicit discussion of how to derive the completed tensor
product, in order to avoid excessive mathematical baggage.
\end{remark}

\begin{remark}
Nonetheless, one would like to interpret the completed Hochschild homology, given here via an
explicit formula, as a derived tensor product
\begin{equation}
    \label{eq:make-sense-of-completed-hochschild-homology}
    HH^\completed(R/\kk) = R \hat{\tensor}^L_{R \hat{\tensor}_\kk R} R
\end{equation}
so that it has a homotopical meaning. 
Making sense of this, however, poses a technical challenge, because one must work in some kind of
homotopical category where the tensor product of Banach modules is given by the projective tensor
product. This is a well-known difficulty: for example, the derived category of the category of
Banach modules $Ban_\kk$ over $\kk$ is \emph{quasi-abelian} rather than abelian, which requires
significantly more delicacy in all arguments. We recommend \cite[Appendix A]{meyer-hrm} and
\cite{kelly2022analytic} for a discussion of the relevant phenomena, with \cite{schneiders1999quasi}
for a foundational treatment of quasi-abelian categories. 

In Appendix \ref{sec:solid-spectra}, for the interested reader, we summarize the approach to derived
completed tensor products via condensed mathematics, which has the advantage of allowing us to work
with an \emph{abelian} category of solid $\kk$-modules. In this formalism, our definition of
$HH^\completed(R/\kk)$ via an explicit formula agrees with the meaning, in that formalism, of the
right hand side of \eqref{eq:make-sense-of-completed-hochschild-homology}.
\end{remark}

\begin{remark}
    When considering Hochschild cohomology of Fukaya categories over Novikov rings in the setting
    where wrapping occurs (so that morphism spaces are infinite dimensional over the Novikov ring),
    it is almost certainly necessary to complete tensor products over the Novikov ring in order to
    make correct sense of results like the Abouzaid generation criterion. We hope that symplectic geometers
    working in such settings pay attention to such subtleties and the resulting complications of the
    homotopical treatment of  completed tensor products. 
\end{remark}

The formulae defining the HKR map extend to define a map of commutative differential graded
$\kk$-algebras
\[  R \oplus \hat{\Omega}^1_{R/\kk}[1] \to HH^\completed(R/\kk). \]
\begin{lemma}
\label{lemma:first-exact-sequence-completed}
There is a diagram 
\begin{equation}
\label{eq:first-exact-sequence-completed}	
HH(R/\kk) \to HH^\completed(R/\kk) \to R \to \hat{\Omega}^1_{R/\kk}[2] 
\end{equation}
where the first two maps are a factorization of the collapse map $HH(R/\kk) \to R$ in
$\EE_\infty-S^1$-algebras, and the second two maps are a fiber sequence in the category of
$HH^{\completed}(R/\kk)-S^1$-modules. 	
\end{lemma}
\begin{proof}
The factorization arises from the obvious factorization of the map of cyclic commutative
$\kk$-algebras.  The second statement follows from the statement that the $HKR$ map is an
equivalence on homotopy groups. 

For $R = \kk[[x]]$, this can be checked by noting that there is an \emph{isomorphism} of chain
complexes
\[ HH^{\completed}(\kk[[x]]/\kk) \to \varprojlim_n HH((\kk[x]/x^n)/\kk) \]	
where above we mean the actual inverse limit of bar complexes, not the homotopy limit. Since the
maps on the right hand side satisfy the Mittag-Leffler conditions, the standard
computations of $HH_*((\kk[x]/x^n)/\kk)$ (see \cite{loday2013cyclic}, but also the space-level
computation of Lemma \ref{lemma:weak-cyclotomic-base-for-truncated-polynomial}) show that the
$\varprojlim^1$-term vanishes and thus the completed HKR map is an isomorphism on homotopy groups
and the inverse limit above is a homotopy inverse limit.

For $R = \kk((x))$, one first notes that $HH(\kk[x]/\kk) \to HH(\kk(x)/\kk)$ is a map of
$\EE_\infty-S^1$-rings which just inverts the class of $x$ on homotopy groups. Then, one notes that
there are maps in $\EE_\infty-S^1$-rings 
\begin{equation}
\label{eq:localization-equivalences}	
\begin{gathered}
 HH(\kk(x)/\kk) \tensor_{HH(\kk[x]/\kk)}  HH(\kk[[x]]/\kk) \to HH(\kk((x))/\kk), \\
 HH(\kk(x)/\kk) \tensor_{HH(\kk[x]/\kk)}  HH^\completed(\kk[[x]]/\kk) \to
 HH^{\completed}(\kk((x))/\kk) 
 \end{gathered}
\end{equation}

which are equivalences on homotopy groups. These maps are explicitly defined levelwise: the first
and second lines of \eqref{eq:localization-equivalences} are levelwise the isomorphisms
\[ \kk(x)^{\tensor n} \tensor_{\kk[x]^{\tensor n}} \kk[[x]]^{\tensor n} \to \kk((x))^{\tensor n}\]
\[\kk(x)^{\tensor n} \tensor_{\kk[x]^{\tensor n}} \kk[[x]]^{\hat{\tensor} n} \to
\kk((x))^{\hat{\tensor} n},\]
respectively (noting the completed tensor products in the second line).

Finally, note that the classical tensor product on the left is already computing the derived tensor
product since $\kk[x] \to \kk(x)$ is a localization of commutative $\kk$-algebras. This shows that
the claim for $R = \kk[[x]]$ implies the claim for $\kk((x))$. 
\end{proof}

\begin{proposition}
\label{prop:completed-u-connection}
Let $A$ be a $dg$ algebra over $R$ as above, and suppose $\Q \subset \kk$. The map 
\[ (u \Grad)^{hS^1}: HC^-(A/R) \to HC^-(A/R)\tensor_R \hat{\Omega}^1_{R/\kk}[2]\]
where $u\Grad: HH(A/R) \to HH(A/R)\tensor_R \hat{\Omega}^1_{R/\kk}[2]$ is the map of $R-S^1$-modules
induced by the right map of \eqref{eq:first-exact-sequence-completed} is a $u$-connection on
homotopy groups. 
\end{proposition}
\begin{proof}
As in the proof of Proposition \ref{prop:u-connection-in-char-0}, we can reduce to verifying this
when $A = R$. The statement for $R = \kk((x))$ follows from the statement for $R = \kk[[x]]$ by the
localization arguments earlier. The statement for $R = \kk[[x]]$ follows from the computation of its
Hochschild homology, as a $\kk-S^1$-module, as an inverse limit, together with the fact that homotopy
fixed points commute with inverse limits. This is most clear if we use the formula implied by Lemma
\ref{lemma:alternative-filtration} and Lemma \ref{lemma:lawson-lemma}
	\[ HH^\completed(\kk[[x]]/\kk)^{hS^1} = \varprojlim (THH(\SS[x]/x^n)' \sma \kk)^{hS^1} =
	\prod_{k=1}^\infty C_*(S^1_k,\kk)^{hS^1} \]
	and simply using the geometry of the circle action on $S^1_k$. 
\end{proof}

We conclude with the proof of Theorem \ref{thm:ggm-comparison}. 
\begin{proof}(Proof of Theorem \ref{thm:ggm-comparison} for $R = \kk[[x]], \kk((x))$.) 
Writing 
\[ HH^\completed(A/\kk) = HH(A/\kk) \tensor_{HH(R/\kk)} HH^\completed(R/\kk),\]
Lemma \ref{lemma:first-exact-sequence-completed} shows that there is a fiber sequence of
$\kk-S^1$-modules 
\[ HH^\completed(A/\kk) \to HH(A/R) \xrightarrow{u\Grad} HH(A/R) \tensor_R \hat{\Omega}^1_R[2].\]

Looking back to the proof of Proposition \ref{prop:ggm-exact-sequence}, we write $\tilde{p}_1$ and
$\tilde{p}_2$ for the maps $R \to R \hat{\tensor}_k R$ defined by the same formulae as $p_1$ and
$p_2$ in \eqref{eq:getzler-crystal} respectively. Then the sequence \eqref{eq:getzler-crystal} maps
to the same sequence with $\tilde{p}_1$ replaced with $\tilde{p}_2$ everywhere. Thus the induced map
$u \hat{\Grad}^{GGM}: HP(A/R) \to HP(A/R) \tensor_R \hat{\Omega}^1_{R/\kk}$ is the composition of
the Getzler map $u \Grad^{GGM}: HP(A/R) \to HP(A/R) \tensor_R \Omega^1_{R/\kk}$ with the map 
$\Omega^1_{R/\kk} \to \hat{\Omega}^1_{R/\kk}$. The fact $HP'(A/\kk)$ lies in the kernel of $u
\hat{\Grad}^{GGM}$ then follows from the same argument as in Proposition
\ref{prop:ggm-exact-sequence}. Combining this with Proposition \ref{prop:completed-u-connection},
the argument for Theorem \ref{thm:ggm-comparison} in the previously treated cases now adapts to this
case. 
\end{proof}

\section{Comparing conventions}
\label{sec:comparing-conventions}

In order to connect the algebraic results from prior sections to our
applications in symplectic topology, we must compare all the relevant
notational conventions.

\subsection{Numerical conventions.}

In the theory of Fukaya categories, \emph{cohomological conventions}
are typically used~\cite{seidel2008fukaya, FOOO1,
  ganatra2013symplectic}, i.e. an $A_\infty$-algebra over a ring $R$
is a free graded $R$-module $A^i$ equipped with operations  
\[
\mu_j\colon (A^{\tensor_R j})^i \to A^{i+2-j}, j \geq 1
\] 
satisfying the $A_\infty$ relations. Thus, in particular, $A_\infty$-algebras are \emph{cochain}
complexes, while spectra are more naturally compared with \emph{chain} complexes. We write $Mod_R$
for the symmetric monoidal category of chain complexes over $R$; then, in order for $A$ to
correspond to an algebra $A_\bullet$ over the $A_\infty$-operad in $Mod_R$ defined by the cellular
chain complexes with $R$-coefficients on the associahedra \cite{markl2002operads}, we must take the
underlying object of $Mod_R$ to be the chain complex $A_\bullet$ corresponding to $A$, i.e. the
chain complex $A_i$ with $A^i = A_{-i}$. Similarly, the ``Hochschild chain complex'' of $A$ is
defined in the symplectic literature to be \cite{ganatra2013symplectic}
\[ CC_*(A^\bullet) = \oplus_{n \geq 0} A^{\tensor n}, \deg(y \tensor x_1 \tensor \ldots \tensor x_n)
= |y| + \sum_i \|x_i\|, \|x\| = |x|-1, |x| = j \text{ if } x \in A^j \]
One sees immediately that the (standard) terminology is somewhat misleading, as this is actually a
\emph{cochain} complex. We will write 
\[ CC_*(A_\bullet) = CC_{-*}(A^\bullet) =\oplus_{n \geq 0} A^{\tensor n}, \deg(y \tensor x_1 \tensor
\ldots \tensor x_n) = |y| + \sum_i \|x_i\|', \|x\| = |x|+1, |x| = j \text{ if } x \in A_j\]
for the chain complex corresponding to the ``chain complex'' used in the symplectic geometry
literature. With this grading shift, we see that if $A^\bullet$ was a (cohomological) differential
graded algebra then $CC_*(A_\bullet)$ is the usual Hochschild complex of the corresponding
(homological) $dga$ $A_\bullet$, i.e. the geometric realization of the simplicial chain complex
described in Appendix \ref{sec:hh-appendix}.

Similarly, the ``Hochschild \emph{cochain} complex'' discussed in symplectic topology
\cite{ganatra2013symplectic} is graded as the typical cochain complex of morphisms of $(A,
A)$-bimodules between the (cohomologically graded) two sided bar resolution of the diagonal bimodule
of $A$, and $A$ itself. Recall that given a pair of cochain complexes $(M^\bullet, N^\bullet)$, the
\emph{cochain} complex of morphisms from $M^\bullet$ to $N^\bullet$ satisfies  
\[ Hom^k(M^\bullet, N^\bullet) = \{ f: M^\bullet \to N^\bullet, f(M^i) \subset N^{i+k} \text{ for
all } i\} \]
and similarly, given a pair of chain complexes $(M_\bullet, N_\bullet)$, the \emph{chain} complex of
morphisms from $M_\bullet$ to $N_\bullet$ satisfies
\[ Hom_k(M_\bullet, N_\bullet) = \{ f: M_\bullet \to N_\bullet, f(M_i) \subset N_{i+k} \text{ for
all } i\}.\]
Thus, replacing cochain complexes with their corresponding chain complexes (i.e. setting $M_k =
M^{-k}$) commutes with taking Hom, i.e. the $Hom_k(M_\bullet, N_\bullet) = Hom^{-k}(M^\bullet,
N^\bullet)$. 

Now, given an associative algebra $A$, this means that $RHom^{A \tensor A^{op}}_\bullet(A, A)$ is
concentrated entirely in negative degrees. Thus, many sources define $HH^i(A, A) = RHom^{A \tensor
A^{op}}_{-i}(A,A)$, for example Loday \cite[1.5.1]{loday2013cyclic}, and we follow this convention in
this paper. With this latter convention, one has that the cap product (see Appendix
\ref{sec:hh-appendix}) is a map
\begin{equation}
    \label{eq:algebraic-cap-product-convention}
    HH^i(A) \tensor HH_j(A) \to HH_{j-i}(A).
\end{equation}
This of course disagrees with the convention in the symplectic geometry literature, where the
convention is that the cap product takes the form (see e.g. \cite[Corollary
8.3]{ganatra2013symplectic})
\begin{equation}
    \label{eq:symplectic-cap-product-convention}
    HH^i(A) \tensor HH_j(A) \to HH_{i+j}(A). 
\end{equation} 
The discussion above explains the source of this disagreement. Indeed, if $A$ is a (cohomological)
dg algebra, then, writing $HH_{alg}^i, HH^{alg}_i$ for the ``algebraic'' grading conventions described
above, we have
\[ HH^i_{alg}(A_\bullet) = HH^i_{symp}(A^\bullet), HH_i^{alg}(A_\bullet) =
HH_{-i}^{symp}(A^\bullet)\]
and making these changes takes \eqref{eq:algebraic-cap-product-convention} to
\eqref{eq:symplectic-cap-product-convention}.

\subsection{Comparing homotopy fixed points, etc. }
For any finite group $G$ or for $G=S^1$, and for any ring $k$, there is a chain of equivalences of
$k$-linear stable $\infty$-categories  contravariantly functorial in $G$,
\begin{equation}
    Fun(BG, Mod_k) \simeq Mod(k \sma \Sigma^\infty_+ G) \simeq Mod(k[G]') .
\end{equation} 
Here $k[G]'$ denotes the classical algebra $k[G]$ if $G$ is finite, and denotes the algebra
$k[\epsilon]$ with $\epsilon$ in (homological) degree $1$ if $G = S^1$. 
Given $X$ in the category on the left, we have that $X_{hG}$ is the homotopy colimit of this
functor, while $X^{hG}$ is the homotopy limit. Under this equivalence, given a $dg$-module
over $k[G]'$, one can compute $X_{hG}$ as $X \tensor_{k[G]'} k$, and one can compute $X^{hG}$ as
$Hom_{k[G]'}(k, X)$. The standard resolutions of the augmentation of $k$ as a $k[G]'$ module then
let us write explicit complexes computing $X^{hG}$ and $X_{hG}$, which agree with the complexes
used in symplectic geometry \cite{ganatra2019cyclic, shelukhin-zhao} under the changes of convention
described in the previous paragraph. 

\subsection{Tate diagonals.}

\begin{lemma}
\label{lemma:tate-diagonal-comparison}
    Let $R$ be a commutative $\F_p$-algebra,  let $V$ be perfect complex of $R$ modules with
    $C_p$-action and let 
    $\hat{C}^\bullet(C_p,V)$ denote the complex computing Tate cohomology of $V$ defined in
    \cite[Section 2]{shelukhin-zhao}. Write $HV$ for the object  (\cite{shipley2007hz},
    \cite[7.1.1.16]{HA}) of the $\infty$-category of module spectra over the $\EE_1$-algebra
    corresponding to $k$, which we will also denote by $k$. There are functorial isomorphisms of
    graded abelian groups 
    \begin{equation}
    \label{eq:isomorphism-of-cohomology-groups}
        H^*(F^*V) {\tensor^L_R} R((u))\langle \theta \rangle \simeq \pi_{-*}(HV \tensor_R R^{tC_p})
    \end{equation} 
    where on the right the map $R \to R^{tC_p}$ is the map of $\EE_\infty$-algebras $R \to R^{tC_p}$
    given by the relative Tate diagonal and on the left the map $R \to R[[u]] \langle \theta
    \rangle$ is simply inclusion into the degree zero part of the complex; and also functorial
    isomorphisms. 
    \begin{equation}
        \label{eq:equivariant-cohomology-is-computed-in-the-obvious-way}
        \hat{H}^*_{C_p}(V^{\tensor_R p}) \simeq \pi_{-*} (HV^{\tensor_R p})^{tC_p}.
    \end{equation} 
    Here the $C_p$ action on $V^{\tensor_R p}$ is by permuting tensor factors with appropriate
    Koszul signs; write $|x|$ for the degree of $x$, and $F$ for the Frobenius map on $R$.  
    Under these isomorphisms, the following two maps of graded abelian groups agree. The first map
    is the map given on cohomology \cite[Lemma 2.5]{seidel-wilkins} by the chain-level formula
    \begin{equation}
    \label{eq:kaledin-map}
        F^*V \tensor^L_R \hat{C}^*(C_p, R) \to F^*\hat{C}^*(C_p, V) \to \hat{C}^*(C_p, V^{\tensor_R
        p}), x \tensor 1 \mapsto x^{\tensor p} \tensor u^{-(p-1)|x|/2}.
    \end{equation}
    where $x$ is a \emph{cycle} in $F^*V$. 
    
    The second map is the map on homotopy groups induced by the relative Tate diagonal 
    \begin{equation}
        \label{eq:relative-tate-diagonal-for-HV}
         HV \tensor^L_R R^{tC_p} \to ((HV)^{\tensor_R p})^{tC_p}.
    \end{equation}
    The two maps are also the same if $R$ is a field over $\F_p$. If $V$ is moreover finite
    dimensional over $R$ then both maps are isomorphisms (and otherwise they are given by
    $u$-completion).
\end{lemma}

\begin{remark}
    The main subtle point here is that the map $k \to k^{tC_p}$ is not a map of $k$-modules if we
    give the right hand side its usual $k$-module structure. Indeed, for $p=2$, this map $k \to
    k^{tC_2}$ is the product of the Steenrod square \cite[Theorem IV.1.15]{nikolaus-scholze}. Thus,
    if this was a map of $k$-modules, it would be determined by the corresponding map on $\pi_0$,
    which is simply the Frobenius map \cite[VI.1.2]{nikolaus-scholze}. On the other hand, Kaledin's
    map of \cite[Section 2]{shelukhin-zhao} does \emph{not} come from a map of chain complexes;
    indeed, it is not even $k$-linear until one multiplies by $u$! Nonetheless, the maps agree on
    homology. 
\end{remark}

\begin{remark}
    A basic fact is that the map \eqref{eq:kaledin-map}, if the $u$-factors are modified so that it
    is defined by the formula $x \tensor u^r \mapsto x^{\tensor p} \tensor u^{r_p}$, then this new
    map lifts to a map of abelian groups $H^*(F^*V) \tensor_R H^*(C_p, R) \to H^*(C_p, V^{\tensor_R
    p})$ \cite[Lemma 2.5]{seidel-wilkins}, which is not an isomorphism and \emph{multiplies degree
    by p}. It would be very interesting to understand the meaning of this purely algebraic
    construction in terms of spectra, and to understand to what extent there is an analog for other
    cohomology theories beyond $H\F_p$. 
\end{remark}

\begin{proof}
    The isomorphism \eqref{eq:equivariant-cohomology-is-computed-in-the-obvious-way} follows from
    the discussion of the previous paragraph. The isomorphism
    \eqref{eq:isomorphism-of-cohomology-groups} is more subtle: one first notes that both sides
    commute with homotopy colimits of complexes, and second that there is a functorial such
    isomorphism for $V$ which are free over $R$. (In particular, the isomorphism
    \eqref{eq:isomorphism-of-cohomology-groups} exists for any element $V \in Mod-R$.)

    One directly sees that for projective $R$-modules, the two maps are the same under this
    comparison; the rest of the claim follows from the five lemma and exactness of the Tate
    construction and the derived tensor product. The claim when $R$ is a field follows from the
    isomorphism $V \simeq H^*(V)$ in $Mod_R$ when $R$ is a field, together with the fact that the
    Tate construction is a homotopy limit and that the map \eqref{eq:relative-tate-diagonal-for-HV}
    factors through the colimit-limit exchange map

    \[ \colim_i ((V_i)^{\tensor_R p})^{tC_p} \to ((\colim_i V_i)^{\tensor_R p})^{tC_p}.\]
    The final claim follows from the above together with the fact that the maps are isomorphisms for
    $V = R$. 
\end{proof}

\subsection{Reformulating our main result in symplectic conventions.}

We define Novikov rings $\Lambda = \Lambda_{\F_p}$ in the subsequent section; here we summarize the
comparison between symplectic and homotopical conventions following from the discussion above. 
Take $A$ an associative orthogonal ring spectrum over $H \Lambda$ and denote the corresponding
differential graded algebra over $\Lambda$ by $A$ as well. Recall that $\pi_* H\Lambda^{tC_p} =
\Lambda  \tensor_{\F_p} \F_p[[u]]\langle \theta \rangle$. We write $\Delta_\Lambda'$ for the Kaledin
diagonal of \cite{shelukhin-zhao} and other symplectic geometry papers on the quantum Steenrod
operations, and $\Delta_{\Lambda}$ for the map described in Lemma
\ref{lemma:tate-diagonal-comparison}, the cohomology-level shadow of the relative Tate diagonal,
which can be seen as coming from the HHR diagonal. Proposition \ref{prop:zihong-ours-comparison}
concludes the proof that the following diagram commutes: 
\begin{equation}
    \label{eq:p-fold-cap-product-comparison}
    \begin{tikzcd}
        F_{\Lambda}^* \pi_{-k} THC(A/H\Lambda) \tensor_{\Lambda} \pi_{-j}THH(A/H\Lambda)^{tC_p}
        \ar[d, "u^{\frac{(p-1)k}{2}}\Delta_{\Lambda}"] \ar[r, equals]&
        F_{\Lambda}^*HH^k_{symp}(A) \tensor_{\Lambda} HH^{tC_p, symp}_j(A) 
        \ar[d, "\Delta'_{\Lambda}"]  \\
        \pi_{-pk} N_{H\Lambda}^{C_p} THC(A/H\Lambda) \tensor_{\pi_* H\Lambda^{tC_p}} \pi_{-j}
        THH(A/H\Lambda)^{tC_p}
        \ar[d] \ar[r, equals]&
        \hat{H}_{C_p}^{pk}(CC^*_{symp}(A)) \tensor_{\pi_* H\Lambda^{tC_p}} HH^{tC_p, symp}_j(A)
        \ar[d] \\
        \pi_{-(pk+j)} THH(A/H\Lambda)^{tC_p}  \ar[r, equals]
        &
        HH^{tC_p, symp}_{pk+j}(A) 
    \end{tikzcd}
\end{equation}

\begin{remark}
    The right hand side of the diagram above can actually be somewhat refined, replacing all of the
    Tate cohomologies with corresponding equivariant cohomologies. This follows from observation
    together with the proof of \cite[Lemma 2.5]{seidel-wilkins}, which shows in fact that the map
    taking a cocycle $x \in X$, where $X$ is a cochain complex, to its $p$-th power
    $x^{\tensor_p} \in C^*(C_p, (X^{\tensor p}))$, gives a well-defined map $H^*(X) \to
    H^*_{C_p}(X^{\tensor p})$ which is Frobenius-linear over $\Lambda$ and additive  after also
    multiplying by $u$. It is not clear to the author how to make sense of this refinement at the
    level of spectra; however, since in the cases of interest we have preferred isomorphisms
    $HH^{hC_p, symp}_*(A) \simeq HH^{symp}_*(A) \tensor_{\F_p} \F_p[[u]]\langle \theta \rangle$,
    this does not pose any particular challenge for applications.
\end{remark}

\section{Review of Symplectic Geometry}
\label{sec:putting-it-all-together}

We now recall the structure of quantum cohomology and the quantum Steenrod operations, as well as
the comparisons between the quantum cohomology and the Hochschild (co)-Homology of the Fukaya
category. We then use these comparisons to prove Theorem
\ref{thm:symplectic-p-curvature-for-calabi-yau}. 

\begin{remark}
The treatment here is telegraphic, but the axioms and constructions described here are standard in
the symplectic geometry community, with unifying philosophy being that `algebraic structures on
Hochschild (co)homology correspond to closed-string symplectic operations via open-closed maps', a
philosophy also underlying Costello's work \cite{costello2007topological} and exposited in
Blumberg-Cohen-Teleman \cite{blumberg2009open}. We suggest that readers with a more algebraic
background treat this section as a black box, or as an efficient review comparing algebra with
symplectic geometry.
\end{remark}

\subsection{Novikov rings.} 
Let $R$ be a ring, and let $M$ be a Calabi-Yau symplectic 
manifold (i.e. the restriction of $c_1(M)$ to spherical classes in $M$ is zero) and let $G \subset
\R$ be a submonoid.  Suppose further that $[\omega] \in H^2(M, \Z)$. One can then define an integral
Novikov ring 

\[ \Lambda_{R, G}^{univ}(M) = \left\{ \sum_{A \in H^{sphere}_2(M), A=0 \text{ or } \omega(A) > 0}
c_q q^A | c_q \in R, \#\{c_q \neq 0 | \omega(q) < C\} < \infty \text{ for all } C\right\}\]

where $H^{sphere}_2(M) = im(\pi_2(M) \to H_2(M, \Z))/Tors$. Write $r = \text{rank }H^{sphere}_2(M)$.
In fact, $\Lambda_R^{univ}$ is isomorphic to a subring of a formal power series ring $R[q_2^{\pm 1},
\ldots, q_r^{\pm 1}][[q_1]]$. To see this, we pick a basis of primitive vectors $A_2, \ldots, A_r$
of the rank $r-1$ submodule $\{x | \omega(x) = 0\} \subset H^{sphere}_2(M)$, and let $A_1 \in
H^{sphere}_2(M)$ be a primitive element such that $\omega(A_1) =1$, in which case $A_1,\ldots, A_r$
are  a basis of $H^{sphere}_2(M)$. The map then sends $q^{\sum_i a_i A_i}$ to $\prod_i q_i^{a_i}$;
in fact, its image is the ring $R[q_i^{\pm 1}][[q_1]] \setminus \{ R[q_i^{\pm 1}] \setminus R\}$. 

One can also define $\Lambda_R(M)$ as the image of $\Lambda^{univ}(M)$ in $R[[q]]$ under the map
sending $\sum_A c_q q^A$ to $\sum_A c_q q^{\omega(A)}$. This is of course isomorphic to $R[[q]]$ by
the previous discussion, corresponding to evaluation $q_i = 1$ for $i=2, \ldots r$. Naturally,
taking $R=\Z$ gives the universal case. 
\begin{remark}
When the symplectic form lies in $H^2(M,\hat{G})$ for $G \subset \R$ a submonoid, the definition of
$\Lambda^{univ}_R(M)$ is unchanged, but one must define $\Lambda_{R, G}(M)$ to given by the same
formula as $\Lambda^{univ}_R$, but with $c_qq^A$ replaced by $c_qq^z$ for $z \in G$ and $\omega(q)$
replaced by $z$ as well. When $G$ is a \emph{subgroup} then $\Lambda_{R, G}$ is a field whenever $R$
is. Thus in general one must take $G$ to contain the discrete monoid $\hat{G} \subset \R_{\geq 0}$
generated by the areas of all relevant holomorphic curves. 
\end{remark}

\subsection{Quantum Cohomology and the Quantum Connection.}
For triples of classes $c_1, c_2, c_3$ in $H^*(M, \Z)$ and elements $A \in H^{sphere}_2(M)$, there
are three pointed genus zero Gromov-Witten invariants 
\[ \langle c_1, c_2, c_3 \rangle_A \in \Z.\]
These are zero if $\omega(A) \leq 0$ and $A \neq 0$. Setting 
\[ \int_M (c_1 *_A c_2) c_3 = \langle c_1, c_2, c_3\rangle_A\]
defines the quantum product 
\[  *: H^*(M, \Lambda^{univ}_R)^{\tensor_R 2} \to  H^*(M, \Lambda^{univ}_R)^{\tensor_R 2}, \gamma_1
* \gamma_2 = \sum_A (\gamma_1 *_A \gamma_2)q^A  \]
and similarly a product on $H^*(M, R[[q]])$ by evaluating $q^A \mapsto q^{\omega(A)}$. These
products turn out to be commutative and associative. 

For $a \in H^2(M, \Z)$, write $\partial_a q^A = a(A) q^A$. This is in fact a derivation of
$\Lambda^{univ}_R$, since $\partial_a (q^{A} q^{B}) = a(A)q^{A+B} + a(B) q^{A+B}$.  One can define
the quantum connection on $H^*(M, \Lambda^{univ}_R)[[u]]$, where $u$ is a formal variable of degree
$2$, by  
\[ \Grad_a \gamma = u \partial_a \gamma + a * \gamma. \]
One sees immediately that this is actually a $u$-connection relative to $\partial_a$ (i.e. it
satisfies the $u$-Leibniz rule 
\[ \Grad_a(f\gamma) = u (\partial_a f) \gamma + f \Grad_a \gamma \]
rather than a connection).

We will be primarily interested in the choice of $a = [\omega]$. In that case, we see that
$\Grad_{[\omega]}$ descends to an operator on $H^*(M, \Lambda_R)[[u]] = H^*(M, R[[q]])[[u]]$. Here,
the operator $\partial_a$ becomes $q \partial_q$; thus, we see that the quantum connection on
$H^*(M, \Lambda_R)[[u]]$ is a \emph{$qu$-connection}. 

\subsection{Quantum Steenrod Operations.}
For any $A \in H_2^{sphere}(M)$ and any $b \in H^*(M, \F_p)$, there are $C_p$-equivariant
Gromov-Witten invariants giving the coefficients of an operation 
\[ Q\Sigma_{b, A}: H^*(M, \F_p) \to H^{*+p|b|}(M, \F_p)[[u]]\langle \theta \rangle\]
where $\F_p[[u]]\langle \theta \rangle = H^*(BC_p, \F_p)$. When $A=0$ this operation is 
\[ Q\Sigma_{b, 0}c = St(b)c\]
the terms of $St(b)$ when expanded in $t$ and $\theta$ agree with the Steenrod powers of $b$ up to
standard signs \cite{seidel2023formal}. More generally, one defines
\[Q\Sigma_b: H^*(M, \F_p) \to H^{* + p|b|}(M, \Lambda^{univ}_{\F_p}) [[u]] \langle \theta \rangle\]
by $Q\Sigma_b(c) = \sum_A (Q\Sigma_{b, A}c)q^A$. Then one has the relations 
\[ Q\Sigma_1 = id, Q\Sigma_b(1) = QSt(b)\]
where $QSt$ is the quantum Steenrod power of $b$ \cite{wilkins2021quantum}. One can take the pushout
along the map $\Lambda^{univ}_{\F_p} \to \Lambda_{\F_p}$ to get Quantum Steenrod operations valued
in $H^*(M, \F_p[[q]][[u]] \langle \theta\rangle)$. 

The map $Q\Sigma_b$ is extended to an endomorphism $Q\Sigma_{\beta}$ of $H^*(M,
\Lambda^{univ}_{\F_p})[[u]]\langle \theta \rangle$ or  $H^*(M, \Lambda_{\F_p})[[u]]\langle \theta
\rangle$ of degree $p|\beta|$ by setting
\[ Q\Sigma_\beta (\sum_{r} \gamma_r u^r + \gamma'_r u^r \theta) = \sum_A q^{pA} \sum_r \left(u^r
Q\Sigma_{b_A} \gamma_r + u^r \theta Q\Sigma_{b_A} \gamma'_r\right)\text{ for } \beta = \sum_A b_A
q^A. \]
With this definition one has that 
\[ Q\Sigma_{b_1} Q\Sigma_{b_2} = (-1)^{|b_1||b_2|\frac{p(p-1)}{2}}Q\Sigma_{b_1 * b_2}.\]

\subsection{Domains of definition of Fukaya category.}
There are several choices that go into the definition of a Fukaya category:
\begin{itemize}
    \item The Lagrangian submanifolds $\{L_i\}_{i \in I}$ to be considered as objects;
    \item The ring $R$ over which fundamental classes of moduli spaces of holomorphic curves are
    defined, and 
    \item A large amount of non-canonical data $\mathfrak{P}$, the most important of which is an
    almost complex structure $J$;
    \item A submonoid $G \subset R$, which must contain the values $\omega(A)$ for $A$ any
    $J$-holomorphic polygon with boundary on $\cup_{i \in I} L_i$. It is a nontrivial consequence of
    Gromov compactness that one can choose $G$ to always be a \emph{discrete} submonoid whenever
    $|I|$ is finite. In many cases, e.g. when $\omega \in H^2(M, \cup_{i \in I} L_i, \Z)$, one can
    choose $G$ to be the set of natural numbers.
\end{itemize}
Given this data, the Fukaya category $Fuk(M, \{L_i\}; \omega, J; \Lambda_{R, G})$ is a curved
filtered $A_\infty$-category linear over $\Lambda_{R, G}$, which is well defined up to curved
filtered $A_\infty$-equivalence  (henceforth referred to as \emph{equivalence}) independently of the
set of $\mathfrak{P}$ with fixed $J$.  For any inclusion $G \subset G'$, we have the base-change
isomorphism 
\[ Fuk(M, \{L_i\}; \omega, J; \Lambda_{R, G}) \tensor_{\Lambda_{R, G}} \Lambda_{R, G'} = Fuk(M,
\{L_i\}; \omega, J; \Lambda_{R, G'}) \]
Given different choices of $J$, there is a larger discrete submonoid $G' \supset G$, such that the
categories are equivalent after base change to $\Lambda_{R, G'}$. In particular, so long as
$\{L_i\}_{i \in I}$ is finite, the category $Fuk(M, \{L_i\}; \omega; J; \Lambda_{R, G})$ is
independent of $J$ when $R = \R$ and $G=\R_+$, and is moreover strictly unital in that setting
\cite{FOOO1, fukaya2025unobstructed}. 

\begin{remark}
In a large class of examples related to mirror symmetry, one can choose $\{L_i\}_{i \in I}$, $R = \Z$
and $G$ to be the monoid of natural numbers \cite{perutz2023constructing}, although the comparison
to the more general foundations of \cite{FOOO1, fukaya2025unobstructed} remains conjectural
\cite[Conjecture 1.7]{perutz2023constructing}. The choice of $R$ is primarily a consequence of the
contributions of multiply covered sphere bubbles, leading to orbifold points in moduli spaces of
curves. 
\end{remark}

\textbf{Assumption A}: There is a set of choices $\{L_i\}_{i \in I}, J$ such that we can take $R =
\Z[1/N]$ and $G= c\mathbb{N}$ for some $c \in \R$. In particular, $\Lambda_{R, G} = \Z[1/N][[q]]$. 

Given a curved filtered $A_\infty$-category $A$, there is an associated filtered $A_\infty$-category
$A^{bc}$ with objects $(L, b)$ of $A^{bc}$ consisting of an object $L$ of $A$ and a \emph{bounding
cochain} $b$ for $L$ \cite{sheridan-versal, FOOO1, fukaya2025unobstructed}. We will henceforth
assume that Assumption A is satisfied. Write $\Lambda_0$ for $\Lambda_{R, G} = \Z[1/N][[q]]$ and
write $\Lambda = \Lambda_0[q^{-1}]$. Write $Fuk(M, \Lambda_0)$ for $Fuk(M, \{L_i\}_{i \in I};
\omega, J; \Lambda_{R, G})$, $Fuk(M, \Lambda) = Fuk(M, \Lambda_0) \tensor_{\Lambda_0} \Lambda$. We
will change notation to be consistent with the rest of the paper, writing $\kk$ for $R$. We assume
that $p > N$. For any field $F$ of characteristic zero or of characteristic greater than $N$, we
write $\Lambda^{F} = F((q))$, $\Lambda^{F}_0 = F[[q]]$, and as before we write $Fuk(M,
\tilde{\Lambda})$ for $Fuk(M, \Lambda_0) \tensor_{\Lambda_0} \tilde{\Lambda}$ where
$\tilde{\Lambda}$ is either of these rings. 

Assumption $A$ has been verified in the case of Calabi-Yau hypersurfaces in projective
space,\cite{sheridan-versal, perutz2023constructing} and more generally for Greene-Plesser mirrors
\cite{2312.01949}. 

\subsection{Open-closed maps and their equivariant analogs.}

There is a class of comparison maps between the Hochschild (co)homology, with its $S^1$-action, and
the quantum cohomology of a symplectic manifold, which have been defined in a variety of settings
\cite{abouzaid2010geometric, ganatra2019cyclic, ganatra2025cyclic}. Write $\tilde{\Lambda}$ for `an
appropriate Novikov ring' as per the earlier discussion, e.g. $\tilde{\Lambda} = \Lambda_0$ or
$\tilde{\Lambda} = \Lambda$ under Assumption $A$. Below, we use the \emph{symplectic} conventions
for the grading on Hochschild (co)homology.

The basic examples of the comparison maps are the \emph{open-closed map}
\begin{equation} \label{eq:oc}OC: HH_*(Fuk(M), \tilde{\Lambda}) \to QH^{*+n}(M, \tilde{\Lambda})
\end{equation}
and the \emph{closed-open map}
\begin{equation}
\label{eq:co}
    CO: QH^{*}(M, \tilde{\Lambda}) \to  HH^*_{symp}(Fuk(M; \tilde{\Lambda})) 
\end{equation}
The model for this construction is \cite{abouzaid2010geometric}; the challenge to define it in more
general settings is essentially an issue of transversality or of the construction of appropriate
virtual fundamental cycles. The paper \cite{ganatra2013symplectic} verifies in the Liouville setting
that $CO$ is a map of algebras, and the method is quite general \cite{sheridan2025constructing}.

One also expects equivariant refinements, for $G = C_p$ or $G=S^1$, of the map $OC$ to maps of
$A_\infty$-modules over $H_*(G, \kk)$, which induce maps
\begin{equation}
\label{eq:oc-equivariant}
    OC^{hG}_{\tilde{\Lambda}}: HH_\bullet(Fuk(M; \tilde{\Lambda}))^{hG} \to QH^{*+n}(M,
    \tilde{\Lambda})^{hG}
\end{equation} 
where on the right hand side the $G$-action is trivial, and similarly with $hG$ replaced by $tG$. 

For $G=S^1$, the map $OC^{hG}_{\tilde{\Lambda}}$ has been constructed in a variety of cases by
\cite{ganatra2025cyclic, sheridan2025constructing}; for $G=C_p$, and $M$ Fano, and for
$\tilde{\Lambda} = \F_p[q, q^{-1}]$ with $q$ of degree $2$ (as appropriate in the Fano setting),
an analogous map has been constructed by \cite{chen2024operadic}. Moreover, in that case,
\cite{chen2024operadic} verifies that $OC^{hC_p}_{\tilde{\Lambda}}$ is intertwined with
$OC^{hS^1}_{\tilde{\Lambda}}$ under the comparison of Lemma \ref{lemma:tS^1-tC_p-F_p-mod}.

\textbf{Assumption B}: The maps $OC, CO$ are defined for $\tilde{\Lambda} = \Lambda$, as well as the
maps $OC^{hS^1}_{\Lambda}$; moreover, so is $OC^{hC_p}_{\Lambda^{\F_p}}$, which is intertwined with
$OC^{hS^1}_{\Lambda}$ under the comparison of Lemma \ref{lemma:tS^1-tC_p-F_p-mod}. 

\begin{definition}
    The symplectic manifold $M$ is \emph{nondegenerate} over $\tilde{\Lambda}$ when
    $OC_{\tilde{\Lambda}}$ is an isomorphism. The maps $OC_\Lambda^{hC_p}$ and $OC_\Lambda^{hS^1}$
    are necessarily isomorphisms. A general argument \cite{abouzaid2010geometric} then implies that
    $CO$ is also an isomorphism, and that $M$ is smooth over $\tilde{\Lambda}$. 
\end{definition}

\textbf{Assumption C}: The symplectic manifold $M$ is nondegenerate over $\Lambda$ (and thus is so
over $\Lambda^{\F_p}$ for every $p > N$, as well as over $\Lambda^F$). 

\begin{remark}
    This is the most constraining assumption, as it requires the construction of a sufficient number
    of convenient Lagrangian submanifolds. However, this has been verified for Calabi-Yau
    hypersurfaces in projective space by \cite{sheridan2015homological}, and for Greene-Plesser
    mirrors by \cite{2312.01949}.
\end{remark}

The above maps are expected to obey several compatibility conditions. 

\subsection{Comparison with Getzler-Gauss-Manin connection:}
The map $OC^{hS^1}$ is expected to 
intertwine the corresponding $u$-connections, i.e. 
\[ OC^{hS^1}\Grad^{GGM}_{q \partial_q} = \Grad^{QH}_{u q \partial/\partial q} OC^{hS^1}_{\Lambda_\Z}
\]
This is stated as an assumption in \cite{ganatra2015mirror}, and has been verified in many settings
for monotone symplectic manifolds \cite{hugtenburg2024cyclic, pomerleano2023quantum} and for the
class of Greene-Plesser mirrors by \cite{ganatra2025cyclic}.

\textbf{Assumption D}: The symplectic manifold $M$ satisfies the above comparison at the level of
cohomology for $\tilde{\Lambda} = \Lambda^\Q$. 

\begin{remark}
\label{rk:quantum-connection-pole}
    If $M$ is not a symplectically aspherical symplectic manifold, one generally cannot take
    $\tilde{\Lambda}$ to be a Novikov \emph{ring} (i.e. $\Lambda^{univ}_R$ or $\Lambda_0$) while
    satisfying Assumption $C$. Indeed, this would \emph{contradict} Assumption $D$: the quantum
    connection manifestly has a pole at $q=0$, while the Getzler-Gauss-Manin connection of a smooth
    proper category over a ring has no poles over that ring (by construction). 
\end{remark}

\subsection{Comparison with equivariant cap products.}
Quantum cohomology is an algebra under the quantum product; and the closed-open map is a map of
algebras \cite{ganatra2013symplectic} (when it is defined). Moreover, due to the Calabi-Yau
structure of the Fukaya category (ibid.), one expects that the cap product action of Hochschild
cohomology on Hochschild homology is compatible with the quantum product in the sense that the
following diagram commutes:
\begin{equation}
    \label{eq:cap-product-comparison}
    \begin{tikzcd}
        HH_*(Fuk(X, \tilde{\Lambda})) \ar[r, "OC"]
        \ar[d, "CO(b) \cap"] & QH^{*+n}(X, \tilde{\Lambda})[[u]] \langle \theta \rangle \ar[d, "b*"]
        \\
        HH^{C_p}_{*+p|b|}(Fuk(X, \tilde{\Lambda})) \ar[r, "OC"] &
        QH^{*+p|b|+n}(X, \tilde{\Lambda})\langle \theta \rangle. 
    \end{tikzcd}
\end{equation}

This has been established in a variety of cases, see \cite{ganatra2013symplectic, ganatra2019cyclic,
ganatra2025cyclic}.

One moreover expects equivariant enhancements of the above fact, e.g. that for all $b \in QH(X,
\tilde{\Lambda})$, the diagram commutes
\begin{equation}
    \label{eq:equivariant-cap-product-comparison}
    \begin{tikzcd}
        HH_*^{C_p, symp}(Fuk(X, \tilde{\Lambda})) \ar[r, "OC^{hC_p}"]
        \ar[d, "CO(b) \cap^{\Z_p}_\Lambda"] & QH^{*+n}(X, \tilde{\Lambda})[[u]] \langle \theta
        \rangle \ar[d, "Q\Sigma_b"] \\
        HH^{C_p, symp}_{*+p|b|}(Fuk(X, \tilde{\Lambda})) \ar[r, "OC^{hC_p}"] &
        QH^{*+p|b|+n}(X, \tilde{\Lambda})[[u]]\langle \theta \rangle. 
    \end{tikzcd}
\end{equation}
In the monotone setting, this is established in \cite[Theorem~1.2]{chen2024quantum}.

\textbf{Assumption E:} The diagram \eqref{eq:cap-product-comparison} commutes for $\tilde{\Lambda}
= \Lambda$ (and thus under Assumption $C$ for $\tilde{\Lambda} = \Lambda^F$ for any $F$), and the
diagram \eqref{eq:equivariant-cap-product-comparison} commutes for $\tilde{\Lambda} =
\Lambda^{\F_p}$. 

\begin{remark}
    One expects that assumptions $D$ and $E$ follow for free whenever the Fukaya category is defined
    over $\tilde{\Lambda}$, as the only technical difficulty is transversality or virtual class
    techniques, given that the combinatorics of the relevant moduli spaces is explained in the
    earlier-cited papers. 
\end{remark}

We finally include the technical

\textbf{Assumption F}: The cohomology of $M$ has no $p$-torsion.
\begin{remark}
    Thus, all assumptions have been verified in the literature for Greene-Plesser mirror pairs,
    except for the commutativity of \eqref{eq:equivariant-cap-product-comparison}; this is expected
    to follow from a straightforward combination  of \cite[Theorem~1.2]{chen2024quantum} with the
    techniques of \cite{sheridan2025constructing}.
\end{remark}

\subsection{Proof of Theorem \ref{thm:symplectic-p-curvature-for-calabi-yau}}
Suppose that $M$ satisfies assumptions $A-E$. 

Then $Fuk(M)$ is smooth and proper over $\Z[1/N]((q))$, and satisfies the assumptions of Theorem
\ref{thm:algebraic-p-curvature}. Now assumption $D$ and \eqref{eq:cap-product-comparison} implies
that the action of the cap product by the Kodaira-Spencer class agrees with quantum multiplication
by the symplectic form under the open closed map. Since the closed open map is a map of algebras,
this implies that the symplectic form is sent to the Kodaira-Spencer class. Then by Theorem
\ref{thm:algebraic-p-curvature}, the commutativity of \eqref{eq:p-fold-cap-product-comparison}, and
the equivariant open closed map comparison \eqref{eq:equivariant-cap-product-comparison}, we
conclude that the $p$-curvature of the Getzler connection in the $q \partial_q$-direction agrees
with $Q\Sigma_b$. We have proven the theorem.

\appendix
\section{Cyclotomic Deligne Conjecture}
\label{app:cyclotomic-deligne-conjecture}

In this appendix we state the cyclotomic Deligne conjecture. Let $\mathcal{O}^\cup(n)$ be the $n$-th
space of the little $2$-disks operad, i.e. the space of isometric disjoint embeddings of $n$ closed
disks (of arbitrary radii $1 > r_i > 0$, $i=1, \ldots n$) into the closed unit disk. Let
$\mathcal{O}^{\cap}(n, L)'$ be the space of isometric disjoint embeddings of $n$ disks (of arbitrary
radii $0 < r_i < \min(2, L)$) into the cylinder $S^1 \times [0,L]$ with the flat metric, and let
\begin{equation}
    \label{eq:define-cap-space}
    \mathcal{O}^{\cap}(n) = \cup_{L \in (0, \infty)}S^1 \times \mathcal{O}^\cap(n, L)'.
\end{equation}
topologized such that the map given by projection to $L \in(0, \infty)$ is a fibration. Note that
there is an $S^1$ action on $\mathcal{O}^{\cap}(n)$ which rotates the $S^1$ and the embeddings in
the second factor along the axis of the cylinder simultaneously.
There are natural maps 
\[ \mathcal{O}^{\cap}(n) \times \mathcal{O}^\cup(k_1) \times \ldots \times \mathcal{O}^{\cup}(k_n)
\to \mathcal{O}^{\cap}(k_1 + \ldots + k_n)\]
given by rescaling the embeddings corresponding to the coordinates in $\mathcal{O}^{\cup}(k_i)$ and
then plugging them into the corresponding disks in the embedding corresponding to the
$\mathcal{O}^{\cap}(n)$ factor. Similarly, there are natural maps 
\[ \mathcal{O}^{\cap}(n) \times \mathcal{O}^{\cap}(m) \to \mathcal{O}^{\cap}(n+m)\]
which \emph{aligns markings and concatenates cylinders}, i.e. this restricts to the map 
\[ S^1 \times \mathcal{O}^\cap(n, L_1)' \times S^1 \times \mathcal{O}^\cap(m, L_2)' \to S^1 \times
\mathcal{O}^{\cap}(n+m, L_1+L_2)'\]
one acts by the value of the first $S^1$-action on the second factor, then makes the resulting
second $S^1$-coordinate the final output of the $S^1$-coordinate, simultaneously one glues the
target $S^1\times[0,L_1]$ of the rotated embedding in the second factor with the target $S^1 \times
[0, L_2]$ of the embedding in the first factor to produce an embedding into $S^1 \times [0, L_1 +
L_2]$. 

These spaces can be extended to spaces defining a colored operad in spaces with two colors,
corresponding to $THC(A)$ and $THH(A)$; one simply sets the spaces associated to a number of inputs
from $THH(A)$ distinct from one to be zero. 

The first statement, which is a variant of the Deligne conjecture, is that if we take the
corresponding colored operad in spectra (given by applying $\Sigma^\infty_+$ to all the spaces
above), then $THC(A)$ and $THH(A)$ are naturally an algebra over this colored operad. Unfortunately,
the author does not know of a reference that establishes this statement, although it is a natural
extension of the Deligne conjecture \cite{mcclure1999solution} and proofs of more complex extensions
like the Deligne conjecture for Calabi-Yau algebras \cite{brav2023cyclic} are known. 

In any case, assuming the statement above, we can state the Cyclotomic Deligne Conjecture. We choose
to state the strictest possible variant, which is a statement in the category of genuine equivariant
spectra; however, weaker (or $\infty$-categorical) variants of this statement are easy to derive
from the statement below, by replacing geometric fixed points with Tate fixed points and making the
diagrams commute up to coherent homotopy only. 
\begin{conjecture}[Strict Cyclotomic Deligne Conjecture]
\label{conj:strict-cyclotomic-deligne}
For every natural number $n$, we have a commutative diagram of genuine $S^1$-spectra
\begin{equation}
    \begin{tikzcd}
        \Sigma^\infty_+\mathcal{O}^{\cap}(n) \sma THC(A)^{\sma n} \sma THH(A) \ar[r] \ar[d,
        "\Sigma^\infty_+\phi^{\cap}_m \sma \Delta_m \sma \phi_m"]&  THH(A)\ar[d, "\phi_m"] \\
        (\Sigma^\infty_+\mathcal{O}^{\cap}(nm) \sma THC(A)^{\sma nm} \sma THH(A))^{\Phi C_m} \ar[r]&
        THH(A)^{\Phi C_m}.
    \end{tikzcd}
\end{equation}
Here $\phi_m$ is the cyclotomic structure map, the $S^1$ action on $THC(A)^{\sma n}$ is trivial, and
we compose with the lax monoidal structure map for geometric fixed points on the left, and 
\[ \phi^\cap_m: \mathcal{O}^\cap(n) \to (\mathcal{O}^{\cap}(nm))^{C_p} \]
is defined as follows. There is an $S^1$-action on $\mathcal{O}^\cap(n)$ given by acting by the
$S^1$-action in the second factor but \emph{not} in the first factor. On the codomain of
$\phi^\cap_m$, we are using \emph{this} $S^1$-action to take fixed points; this defines the $S^1=
S^1/C_p$ actions on the domain and codomain of $\phi^\cap_m$. The map $\phi^\cap_m$ is the
$S^1$-equivariant homeomorphism which simply takes the $m$-fold unbranched cover of the
cylinder in the second factor of \eqref{eq:define-cap-space}, and then conformally rescales so that
the $S^1$ factor is again of length $1$. 
\end{conjecture}

Section \ref{sec:topological-cap-product-proof} of this paper proves Conjecture
\ref{conj:strict-cyclotomic-deligne} for one particular point in $\mathcal{O}^\cap(p)$, namely one
corresponding to the picture on the left of Figure \ref{fig:bk-formula-picture-proof}.  

\section{The prismatic subdivision.}
\label{sec:prismatic-subdivision}
In this section, we prove Theorem \ref{thm:generalized-cap}. 
\subsection{Prismatic subdivision of a simplex into product simplices}
The prismatic-subdivision maps are 
\begin{equation}
    \label{eq:prismatic-maps}
    \iota^u_{n, m}: \Delta^n \times \Delta^m \to \Delta^{n+m}
\end{equation}
given in barycentric coordinates by 
\[ \iota_{n, m}((s_0, \ldots, s_n), (t_0, \ldots, t_m)) = (u s_0, \ldots, u s_{n-1}, us_n +
(1-u)t_0, (1-u)t_1, \ldots, (1-u)t_m)\]
where $u$ is any number in $[0,1]$. These maps are used to define the cap product action of $THC(A)$
on $THH(A)$ \cite{malm2010string}, as well as the $E_2$ algebra structure on $THC(A)$
\cite{mcclure1999solution}.  For $u \in (0, 1)$, these maps have the following properties (where we
drop $u$ from the notation):
\begin{itemize}
    \item The map $\bigsqcup_{0 \leq k \leq n} \iota_{k, n-k}$ factors through the quotient by the
    relation 
    \[ \Delta^{k} \times \Delta^{n-k} \ni (\delta^k_ks,t) \sim (s, \delta^{n-k+1}_0t) \ni
    \Delta^{k-1} \times \Delta^{n-k+1}, \]
    and after quotienting $\bigsqcup_{0 \leq k \leq n} \Delta^k \times \Delta^{n-k}$ by this
    relation, the resulting map is a homeomorphism onto $\Delta^n$. 
    \item These maps satisfy the identities 
    \begin{equation}
    \label{eq:prismatic-a}
        \iota_{n+1, m} (\delta^{n+1}_i \times 1) = \delta^{n+m+1}_i \iota_{n, m}, i<n
    \end{equation}
    \begin{equation}
    \label{eq:prismatic b}
        \iota_{n, m+1}(1 \times \delta^{m+1}_i) = \delta^{n+m+1}_{i+n}\iota_{n, m}, i >0 
    \end{equation}
    \begin{equation}
        \label{eq:prismatic c}
        \iota_{n+1, m}(\delta^{n+1}_n \times 1) = \iota_{n, m+1}(1 \times \delta^{m+1}_0).  
    \end{equation}
    \begin{equation}
    \label{eq:prismatic d}
    \iota_{n-1, m} (\sigma^{n-1}_i \times 1) = \sigma^{n+m-1}_i \iota_{n, m}, i=0, \ldots, n-1
    \end{equation}
    \begin{equation}
    \label{eq:prismatic e}
    \iota_{n, m-1} (1 \times \sigma^{m-1}_i) = \sigma^{n+m-1}_{n+i} \iota_{n, m}, i=0, \ldots, m-1
    \end{equation}
\end{itemize}
We will suppress $u$ from the notation; typically we will take $u=1/2$, although we vary $u$ in
Section \ref{sec:relative-bk-formula}. 

\subsection{Maps from cap pairings.}
Given a cap pairing 
\[ \cap_{p,q}: X^p \sma Y_q \to Z_{q-p} \]
we can define the associated map 
\[ mev^m: \prod_{n \geq 0} F(\Delta^n, X^n) \sma \bigvee_{n \geq 0} \Delta^n \times \Delta^m \times
Y_{n+m} \rightarrow \Delta^m \times Z_m,\]
\[ \prod_{n \geq 0} \left((f_n \sma \vee_{\ell \geq 0} (a^\ell_1, a_2, y)) \mapsto (a_2, \cap_{n,
n+\ell}(f(a_1), y)) \right). \]
Taking the direct sum of these maps for $m \geq 0$ defines 
\[  \widetilde{mev}: \prod_{n \geq 0} F(\Delta^n, X^n) \sma \bigvee_{n \geq 0} \bigvee_{m \geq 0}
\Delta^n \times \Delta^m \times Y_{n+m} \rightarrow \bigvee_{m \geq 0} \Delta^m \times Z_m.\]
Now $|X|$ is an equalizer of two maps with domain the first factor of the domain of
$\widetilde{mev}$, while $|Y|$ is the coequalizer of two maps to the codomain. Composing with the
inclusion map from the coequalizer on the first factor of the domain and the map to the equalizer on
the codomain we get a map
\[ |X| \sma \bigvee_{n \geq 0} \bigvee_{m \geq 0} \Delta^n \times \Delta^m \times Y_{n+m}
\rightarrow |Z|. \]
The first claim is that this map factors first through the map 
\[ 1 \sma \iota':  |X| \sma \bigvee_{m \geq 0} \bigvee_{n\geq 0} \Delta^n \times \Delta^m \times
Y_{n+m} \to |X| \sma \bigvee_{k \geq 0} \Delta^k \times Y_k\]
where $\iota'$ is defined via wedge sums of the maps $\iota_{n, m} \times id_{Y_{n+m}}$. This
follows if the outer square of 
\[
\begin{tikzcd}
    {|X|} \sma \bigvee_{m \geq 0} \bigvee_{n \geq 0} \Delta^n \times \Delta^m \times Y_{n+m}
    \ar[r] & \bigvee_{m \geq 0} \Delta^m \times Z_m \ar[rd] & \\
    {|X|} \sma\bigvee_{m \geq 1} \bigvee_{n \geq 0} \Delta^n \times \Delta^{m-1} \times Y_{n+m}
    \ar[r] \ar[u, "1 \sma \vee (1 \times \delta^m_0 \times 1)"] \ar[d, "1 \sma \vee
    (\delta^{n+1}_{n+1} \times 1 \times 1)" ] &
    \bigvee_{m} \Delta^{m-1} \times Z_m \ar[u, "\vee \delta^n_0 \times 1"] \ar[d, "\vee 1 \times
    d^{m+1}_0 "] & {|Z|} \\
    {|X|} \sma \bigvee_{m \geq 1} \bigvee_{n \geq -1} \Delta^{n+1} \times\Delta^{m-1} \wedge Y_{n+m} 
    \ar[r] & 
     \bigvee_m \Delta^{m-1} \times Z_{m-1}  \ar[ru] & \\
    {|X|} \sma \bigvee_{m \geq 0} \bigvee_{n \geq 0} \Delta^n \times \Delta^m \times Y_{n+m} 
    \arrow[equal, u]&&
\end{tikzcd}
\]
commutes.  Here the left horizontal maps are just compositions of inclusions with $1 \sma \iota'$
followed by $\widetilde{mev}$, and the right-most horizontal and diagonal maps are all collapse maps
to the coequalizer. The triangles on the right commutes by the coequalizer formula; the top square
commutes due to \eqref{eq:prismatic c}, and the bottom left square commutes precisely because of the
definition of $|X|$ as an equalizer; symbolically this corresponds to the identification of the
first term of \eqref{eq:d of f cap h} with the last term of \eqref{eq:df of h}. 

The next factorization through the projection to $|X| \sma |Y|$ follows similarly. The structure of
the argument corresponds exactly to the identification of terms described in the review of the
argument that the cap product on Hochschild (co)homology of associative algebras is a chain map. The
formulae for the latter are reviewed in Appendix \ref{sec:hh-appendix}. To check that the given maps
factorize through $|X| \sma |Y|$, we check that the following diagrams commute. 

First, the diagram 
\begin{equation}
\label{diagram1}
\begin{tikzcd}
    {|X|} \sma \bigvee_{n \geq 1} \bigvee_{m \geq 0} \Delta^{n+1} \times \Delta^m \sma Y_{n+m+1}
    \ar[rd] & \\
    {|X|} \sma \bigvee_{n \geq 0} \bigvee_{m \geq 0} \Delta^n \times \Delta^m \sma Y_{n+m+1} \ar[u,
    "1 \sma \vee (\delta^{n+1}_{i+1} \times 1 \sma 1 \sma 1)"] \ar[d, "1 \sma \vee (1 \times 1 \sma
    d^{n+m+1}_i)"] & {|Z|}
    \\
    {|X|} \sma \bigvee_{n \geq 0} \bigvee_{m \geq 0} \Delta^n \times \Delta^m \sma Y_{n+m} \ar[ru] &
\end{tikzcd}
\end{equation}
for $i=0, \ldots, n$, commutes due to \eqref{eq:prismatic-a} and the definition of the equalizer in
${|X|}$.

Second, the diagram
\begin{equation}
\label{diagram2}
    \begin{tikzcd}
    {|X|} \sma \bigvee_{n \geq 0} \bigvee_{m \geq 0} \Delta^{n} \times \Delta^{m+1} \sma Y_{n+m+1} 
    \ar[rd] & \\
    {|X|} \sma \bigvee_{n \geq 0} \bigvee_{m \geq 0} \Delta^n \times \Delta^m \sma Y_{n+m+1} 
    \ar[u, "1 \sma \vee (1 \times \delta^{m+1}_{i-m} \sma 1)"] \ar[d, "1 \sma \vee (1 \times 1 \sma
    d^{n+m+1}_i)"] & {|Z|}
    \\
    {|X|} \sma \bigvee_{n \geq 0} \bigvee_{m \geq 0} \Delta^n \times \Delta^m \sma Y_{n+m} \ar[ru] &
\end{tikzcd}
\end{equation}
commutes for $n+m+1\geq i > n$ due to \eqref{eq:prismatic b} and the coequalizer defining the
codomain $|Z|$.

Third, the diagram
\begin{equation}
\label{diagram3}
\begin{tikzcd}
    {|X|} \sma \bigvee_{n \geq 0} \bigvee_{m \geq 0} \Delta^{n} \times \Delta^{m} \sma Y_{n+m} 
    \ar[rd] & \\
    {|X|} \sma \bigvee_{n \geq 0} \bigvee_{m \geq 0} \Delta^{n+1} \times \Delta^m \sma Y_{n+m} 
    \ar[u, "1 \sma \vee (\sigma^{n}_i \times 1 \sma 1)"] 
    \ar[d, "1 \sma \vee (1 \times 1 \sma s^{n+m}_i)"] & {|Z|}
    \\
    {|X|} \sma \bigvee_{n \geq 0} \bigvee_{m \geq 0} \Delta^{n+1} \times \Delta^m \sma Y_{n+m+1}
    \ar[ru] &
\end{tikzcd}
\end{equation}
commutes for $i=0, \ldots, n-1$ due to \eqref{eq:prismatic d} and the equalizer defining $|X|$.

Finally, the diagram 
\begin{equation}
\label{diagram4}
\begin{tikzcd}
    {|X|} \sma \bigvee_{n \geq 0} \bigvee_{m \geq 0} \Delta^{n} \times \Delta^{m} \sma Y_{n+m} 
    \ar[rd] & \\
    {|X|} \sma \bigvee_{n \geq 0} \bigvee_{m \geq 0} \Delta^{n} \times \Delta^{m+1} \sma Y_{n+m}
    \ar[u, "1 \sma \vee (1 \times \sigma_{i-n}^m \sma 1)"] 
    \ar[d, "1 \sma \vee (1 \times 1 \sma s^{n+m}_i)"] & {|Z|}
    \\
    {|X|} \sma \bigvee_{n \geq 0} \bigvee_{m \geq 0} \Delta^{n} \times \Delta^{m+1} \sma Y_{n+m+1}
    \ar[ru] &
\end{tikzcd}
\end{equation}
commutes for $i=n, \ldots, m$ due to \eqref{eq:prismatic e} and the coequalizer defining the
codomain $|Z|$.

This concludes the proof of the theorem, as the functoriality under maps of cap products is clear.

\section{Homotopies of $THH(R)$-module structures}
\label{sec:thh-r-homotopies}
In this section we explain the homotopies of $THH(R)$-module structures needed for the proof of Theorem
\ref{thm:relative-bk-formula}. This involves carefully unwinding the definitions. 

Let $R \in Comm(Sp_G)$. Then $THH(R)$ is naturally an object in $Comm(Sp_G)$, since the underlying
simplicial object \eqref{eq:simplicial-object-thh} is naturally a simplicial object in commutative
ring spectra via levelwise multiplication in each smash-product factor, and because geometric
realization commutes with smash product. Similarly, if $A \in Alg(Mod_R)$, $THH(A)$ is a
module over $THH(R)$ since $THH^\Delta(A)$ is a module over $THH^\Delta(R)$ by, for each level,
applying the tensor product of the module structures. Let us make this explicit.
\begin{figure}[t]
\label{fig:cut-and-paste}
\includegraphics[width=\textwidth]{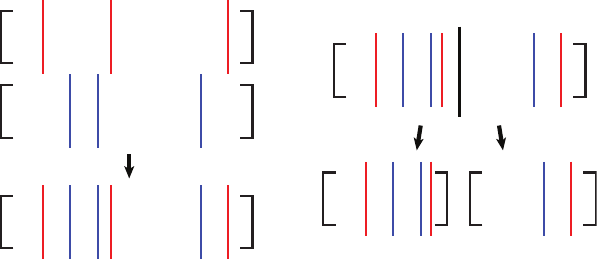}
\caption{\emph{Cut and paste interpretations of decompositions of simplices.} On the left, we
describe the canonical decomposition of $\Delta^a \times \Delta^b$ into simplices $\Delta^{a+b}$
indexed by shuffles. On the right, we describe the prismatic subdivision on the simplex into
products of simplices. On both sides we identify $\Delta^n$ with decompositions of an interval into
$n+1$-subintervals. Taken together, these diagrams show that the cap product is $THH(R)$-linear.  }
\centering
\end{figure}

\subsection{Decomposition of a product simplex into simplices.}
The fact that the smash product of simplicial orthogonal spectra commutes with geometric realization
comes from the corresponding fact for simplicial spaces \cite{milnor1957geometric,
may1992simplicial}. Recall that an \emph{$(n, k)$-shuffle} $(\mu,\eta)$ is a partition of the
totally ordered set $\{1, \ldots n+k\} = \mu \sqcup \eta$ into a pair of subsets, with $|\mu| = n$
and $|\eta|=k$. We write $(\mu_1, \ldots, \mu_n)$ and $(\eta_1, \ldots, \eta_k)$ for the elements of
$\mu$ and $\eta$ in increasing order, respectively. Any $(n, k)$-shuffle $(\mu, \eta)$ defines two
maps 

\begin{equation}
    \label{eq:degeneracy-maps-associated-to-shuffles}
    \sigma^\mu_{\mu, \eta}: [n+k] \to [n], \sigma^\mu_{\mu, \eta}:[n+k] \to [k]
\end{equation}

by dropping the elements of $\eta$, $\mu$, respectively, from $[n+k] = \{0\} \cup \{1, \ldots, n\}$,
and applying the unique isomorphism of totally ordered sets from the resulting set to the codomain.
We write $\Sigma_{n, k}$ for the set of $(n, k)$-shuffles. 

We recall that there is a decomposition 
\begin{equation}
    \label{eq:decomposition-of-product-of-simplices-into-simplices}
    \Delta^n \times \Delta^k = \left(\bigsqcup_{(\mu, \sigma) \in \Sigma_{n, k}}
    \Delta^{n+k}\right)/\sim
\end{equation}
where the interiors of the simplicies on the right are non-overlapping. Geometrically, this
decomposition is depicted in Figure \ref{fig:cut-and-paste} on the left. Given a point $(t_0,
\ldots, t_a) \in \Delta^a$, one can identify this point with the decomposition of the interval with
$a$ `partitions':
\begin{equation}
\label{eq:decompose-interval-into-partitions} 
[0, 1] = [0, t_0] \cup [t_0, t_0 + t_1] \cup \ldots \cup [\sum_{i=0}^{a-1} t_i, 1]. 
\end{equation}
From this perspective, the degeneracy maps $\sigma$ act on the simplex by `dropping partitions',
i.e. $\sigma^n_i$ drops the $i$-th partition; and the face maps duplicate partitions, i.e.
$\delta^n_i$ creates a new partition at $t=t_i$. 

From this perspective, it is clear that if $a = n+k$, then, given an $(n, k)$-shuffle $(\mu, \nu)$,
we get a corresponding point of $\Delta^n$ by dropping the `partitions' corresponding to elements of
$\mu$, and a point of $\Delta^k$ by dropping the `partitions' corresponding to elements of $\nu$.
Explicitly, we have
\[ f_{\mu, \nu}: \Delta^{n+k} \to \Delta^n \times \Delta^k, \;f_{\mu, \nu}(t_0, \ldots, t_a) =
\left(\left(\sum_{j=\mu_r}^{\mu_{r+1}-1} t_j\right)_{r=0}^n, \left(\sum_{j=\eta_r}^{\eta_{r+1}-1}
t_j\right)_{r=0}^k\right)\]
where $\mu_0 = \eta_0 = -1$, $\mu_{n+1}-1 = \eta_{k+1}-1 = n+k$, and $t_{-1} = 0$. The
interpretation in terms of partitions
implies that this is simply
\[ f_{\mu, \nu}(t) = (\trianglespace(\sigma^\mu_{\mu,\nu})(t), \trianglespace(\sigma^\nu_{\mu,
\nu})(t)).\]

\subsection{Explicit module structure.}
\label{sec:explicit-module-structure}

With this notation, we can write the module structure of $THH(R)$ on $THH(A)$ explicitly. Given
$(p_x, p_y)\in \Delta^n \times \Delta^k$, write $(\mu, \eta)(p_x, p_y) \subset \Sigma_{n, k}$ for
the set of $(n, k)$-shuffles $(\mu, \eta)$ such that $(p_x, p_y) \in Im(f_{\mu, \eta})$.  The
module structure is a map of the form

\begin{equation}
    \label{eq:module-structure-map}
    \left[\left(\bigvee_n \Delta^n \times THH^\Delta_n(R)\right)/\sim\right]  \sma
    \left[\left(\bigvee_k \Delta^k \times THH^\Delta_k(A)\right)/\sim \right]\to
    \left[\left(\bigvee_j \Delta^j \times THH^\Delta_j(A)\right)/\sim\right].
\end{equation}

The left hand side can be written as 
\[\left(\bigvee_{n, k} \Delta^n \times \Delta^k \times THH^\Delta_n(R) \sma
THH^\Delta_k(A)\right)/\sim.\]
The universal property of the Day convolution \cite{EKMM} implies that to define the module
structure, it suffices to define maps
\[\left(\bigvee_{n, k} \Delta^n \times \Delta^k \times THH^\Delta_n(R)(V) \sma
THH^\Delta_k(A)(W)\right)/\sim \to \left(\bigvee_j \Delta^j \times THH^\Delta_j(A)(V \oplus
W)\right)/\sim.\]
which are natural in $V$ and $W$ in the appropriate sense. These maps are given by the map which
acts as follows. Write $X = THH^\Delta(R)$ and $Y = THH^\Delta(A)$. For $(p_x, p_y)$ as above and
any $\mu, \eta \in (\mu, \eta)(p_x, p_y)$, let the map act via
\begin{equation}
    \label{eq:module-structure-explicit-formula}
    \begin{gathered}
    \Delta^n \times \Delta^k \times THH^\Delta_n(R)(V) \sma THH^\Delta_k(A)(W) \ni (p_x, p_y, x, y)
    \mapsto \\
    (f^{-1}_{\mu, \eta}(p_x, p_y), m_{n+k} \circ \tensor \circ X_n(s^\mu_{(\mu, \eta)(p_x, p_y)})(x)
    \sma Y_k(s^\eta_{(\mu, \eta)(p_x, p_y)})(y))
    \end{gathered}
\end{equation}
where 
\[ \tensor: X(V) \sma Y(W) \to (X\sma Y)(V \oplus W)\]
is the universal map associated to the smash product of orthogonal spectra, 
\[ m_{n+k}: R^{\sma n+k+1} \sma A^{\sma n+k+1} \to A^{\sma n+k+1} \]
is the structure map for the module structure of $THH^\Delta(A)$ as a simplicial module over the
simplicial orthogonal ring spectrum $THH^\Delta(R)$, and the ambiguity of the possible choices of
$(\mu, \eta) \in (\mu, \eta)^{-1}(p_x, p_y)$ are made irrelevant by the quotient taken in the
codomain of \eqref{eq:module-structure-map}.

\subsection{A family of automorphisms of $THH(R)$.}
Recall that the inverse to collapse map from the two-sided bar complex over $\SS$ to the
two-sided bar complex over $R$ gives a map 
\[ g_R: THC(A/R) \to THC(A). \]
In the next section we will study the interaction of the map
\[ \cap'_R: THH(A) \sma THC(A/R), \cap'_R = \cap \circ (1 \sma g_R)\]
participating in the top line of \eqref{eq:relative-bk-formula} with the $THH(R)$-module structure
on $THH(A)$ elucidated above. Unfortunately, $\cap'_R$ is not strictly $THH(R)$-linear as a map of
orthogonal spectra. However, it turns out to be $THH(R)$-linear if we compose with a natural
self-equivalence of $THH(R)$, with the precise relation given in the next section.

Recall that the orientation-preserving homeomorphisms of the interval act on the geometric
realization of any simplicial orthogonal spectrum \cite{nikolaus-scholze}. Concretely,  a
homeomorphism $h$ of the interval acts on $\Delta^k$ via 
\[ h(t_0, \ldots, t_k) = (h(t'_0)-0, \ldots, h(t'_j) - h(t'_{j-1}) \ldots) \text{ where } t'_j =
\sum_{i=0}^j t_i;\]
in other words, one acts by $h$ on the corresponding partition of the interval
\eqref{eq:decompose-interval-into-partitions}. 
In terms of the formula \eqref{eq:geometric-realization}, the action of $h$
on the geometric realization is induced from the action of $h$ on all the simplex factors in the
coequalizer. 

For $a \in (0,1)$, let $f_a: [0,1] \to [0,1]$ be the homeomorphism which is affine-linear on
$[0,1/2]$ and on $[1/2, 1]$, sends $[0,1/2]$ onto $[0,a]$ and sends $[1/2, 1]$ onto $[a, 1]$. Then
$f_{1/2} = id$ and the limit $f_0 = \lim_{a \to 0} f_a$ is a map of the interval which collapses
$[0,1/2]$ to $0$. 

\begin{lemma}
    The family of maps $f_a: THH(R) \to THH(R)$, $a \in (0, 1/2]$ extend uniquely to a continuous
    family of maps over $[0,1/2]$. Moreover, the map $f_0$ is described as follows: writing 
    \begin{equation}
    \label{eq:prismatic-subdivision-equivalence}
    THH(R) = \left\{\left(\bigvee_{n \geq 0} \Delta^n \sma R^{\sma n+1}\right)/\sim\right\}\; \;
    \simeq \; \; \left\{\left(\bigvee_{n, m} \Delta^n \times \Delta^m \sma R^{\sma
    n+m+1}\right)/\sim\right\}
    \end{equation}
    via the prismatic subdivision, $f_0$ acts on the image of $\Delta^n \times \Delta^m \sma R^{\sma
    n+1}$ via multiplying the first $n+1$ factors of $R^{\sma n+1}$ and then projecting to $THH(R)$. 
\end{lemma}
\begin{proof}
    In fact, the isomorphism in \eqref{eq:prismatic-subdivision-equivalence} depends on the
    parameter $u$ used in the definition of the prismatic subdivision maps
    \eqref{eq:prismatic-maps}. Composing this isomorphism for $u=1/2$ with the inverse isomorphism
    for $u=a$ gives exactly the action of $f_a$. The limit as $a\to 0$ is thus the map 
    \[  \begin{gathered}\left\{\left(\bigvee_{n \geq 0} \Delta^n \sma R^{\sma n+1}\right)/\sim
    \right\}\simeq \left\{\left(\bigvee_{n, m} \Delta^n \times \Delta^m \sma R^{\sma
    n+m+1}\right)/\sim\right\} \\
    \to \left\{\left(\bigvee_{n,m} 0 \times \Delta^m \sma R^{\sma n+m+1}\right)/\sim\right\} \subset
    \left\{\left(\bigvee_{j} \Delta^j \sma R^{\sma j+1}\right)/\sim\right\}
    \end{gathered}\]
    where the first isomorphism uses $u=1/2$ and the composition of the map from the second to the
    fourth quantity is the map is induced from the maps \eqref{eq:prismatic-maps} for $u=0$. Thus
    the continuous extension to $a=0$ is clear; its description as in the lemma follows from the
    relation imposed when defining $THH(R)$ associated to the face map corresponding to
    multiplication in the first $n$ copies of $R$.  
\end{proof}

\begin{lemma}
    The maps $f_a: THH(R) \to THH(R)$ for $a \in [0, 1/2]$ are weak equivalences of algebras. 
\end{lemma}
\begin{proof}
    It suffices to show that they are maps of algebras, because $a=1/2$ is the identity, and all the
    maps are homotopic to an equivalence, and are thus equivalences themselves. Moreover, since
    being a map of algebras is a closed condition, it suffices to check this for $a \in (0, 1/2]$.
    For such $a$, we note that the algebra structure of $THH(R)$ has a description essentially
    identical to the module structure of $THH(A)$ over $THH(R)$ as in Section
    \ref{sec:explicit-module-structure}, but with $A$ replaced with $R$ everywhere. But then this
    follows from the fact that acting by $h$ on a pair of partitions of the interval and then taking
    the intersection of the partitions (as in Figure \ref{fig:cut-and-paste}) is the same as acting
    by $h$ on the intersection of the partitions. 
\end{proof}

\begin{remark}
    A very similar trick is used to compare two definitions of the cup product on $THH(R)$ in
    \cite{mcclure1999solution}. 
\end{remark}
\subsection{$THH(R)$-linearity of the cap product.}
Write $m$ for the left-multiplication of $THH(R)$ on $THH(A)$, and write $m_0 = m(f_0 \sma 1)$; this
defines a new $THH(R)$-module structure on $THH(A)$ which is homotopic (and thus equivalent) in
$\mathcal{D}(THH(R)-mod)$ to the usual module structure.

In this section, we establish 
\begin{proposition}
\label{prop:THH(R)-linearity-of-cap-product-strict}
The following diagram strictly commutes. 
\begin{equation}
\label{eq:THH(R)-linearity-of-cap-product-strict}
\begin{tikzcd}
    THH(R) \sma THH(A) \sma THC(A/R) \ar[r, "1 \sma \cap'_R"] 
    \ar[d, "m \sma 1"] &
    THH(R) \sma THH(A) \ar[d, "m_0"]\\
    THH(A) \sma THC(A/R) \ar[r, "\cap'_R"] & THH(A)
\end{tikzcd}
\end{equation}
\end{proposition}

Taking $R$ and $A$ to be cofibrant-fibrant in the model structure of Proposition
\ref{prop:convenient-model-structure}, we see that all the terms in the diagram
\eqref{eq:THH(R)-linearity-of-cap-product-strict} compute all the corresponding derived functors,
since $THH(A)$ is cofibrant. Thus, we have that
\begin{lemma}
\label{lemma:cap_R'-derived}
    The map $\cap_R'$ defines a corresponding map in $\mathcal{D}(THH(R))$.
\end{lemma}

\begin{proof}[Proof of Proposition \ref{prop:THH(R)-linearity-of-cap-product-strict}]
    Note first that there is a commutative diagram 
    \begin{equation}
    \begin{tikzcd}
        \Delta^{a + a'} \times \Delta^{b + b'} &
        \Delta^{a + a'} \times \Delta^b \times \Delta^{b'}
        \ar[l, "1 \times i_{b, b'}"]
        &
        \Delta^a \times \Delta^{a'} \times \Delta^b \times \Delta^{b'} 
        \ar[l, "i_{a, a'} \times 1"]\\
        \Delta^{a + a' + b + b'} 
        \ar[u, "f_{\mu, \eta}"] & &
        \Delta^{a + b} \times \Delta^{a'+b'}
        \ar[ll, "\iota_{a+b, a'+b'}"]
        \ar[u, "r \circ (f_{\mu_0, \eta_0} \times f_{\mu_1, \eta_1})"]
    \end{tikzcd}
\end{equation}
Here $(\mu_i, \eta_i)$ are uniquely characterized by $(\mu, \eta)$ to get commutativity of the
diagram. (This is \emph{obvious} from the geometric interpretations of these maps in Figure
\ref{fig:cut-and-paste} -- we can 'smush' and then 'cut in half', or we can 'cut each piece in half'
and then 'smush the pairs of pieces'. )
This implies that the diagram in Figure \ref{fig:cap-product-linearity-figure} commutes. But the
commutation of this latter diagram proves Proposition
\ref{prop:THH(R)-linearity-of-cap-product-strict}, since the maps in this diagram are obtained as
quotients of wedge sums of maps in Figure \ref{fig:cap-product-linearity-figure}.
\end{proof}

\begin{figure}[h]
\begin{equation}
    \begin{tikzcd}[cells={font=\small}]
        \Delta^a \times \Delta^{a'} \times \Delta^b \times \Delta^{b'} \times R \sma R^a \sma R^{a'}
        \sma A \sma A^b \sma A^{b'} \sma THC(A/R)\ar[d, dotted, "m"] \ar[r, "f_0 \sma id"]
        & \Delta^a \times \Delta^b \times \Delta^{b'} \times R \sma R^{a'} \sma A \sma A^b \sma
        A^{b'} \ar[d, "1 \sma \cap_R'"]\\
        \Delta^{a+b} \times \Delta^{a'+b'} \sma A \sma A^{a+b} \sma A^{a'+b'} \sma THC(A/R) \ar[rd,
        "\cap'_R"] &
        \Delta^a \times \Delta^{b'} \times R \sma R^{a'} \sma A \sma A^{b'} \ar[d, dotted, "m"]\\
        &\Delta^{a'+b'} \sma A \sma A^{a'+b'}
    \end{tikzcd}
\end{equation}
\caption{Commutative diagram proving $THH(R)$-linearity of the cap product $\cap'_R$. Here $\cap'_R$
is actually the map used to define $\cap'_R$ on the corresponding piece of the prismatic
subdivision, and $f_0$ is induced from the corresponding automorphism of $THH(R)$. The dotted arrows
for $m$ actually denote maps defined in terms of $f_{\mu,\nu}$ going the other way; these maps are
used to define the module structure $m$ of $THH(A)$ over $THH(R)$. }
    \label{fig:cap-product-linearity-figure}
\end{figure}

\subsection{The homotopy is compatible with the proof of Theorem \ref{thm:relative-bk-formula}}
\begin{proof}
We freely use the convenient model structures of Proposition \ref{prop:convenient-model-structure},
and assume that $A$ and $R$ are  cofibrant-fibrant. 
The desired diagram already commutes as a diagram of orthogonal spectra; we need to prove that it
can be lifted to a diagram in $\mathcal{D}(THH(R))$. The lift of the top horizontal arrow is given
in the construction of Lemma \ref{lemma:cap_R'-derived}. The corresponding left vertical arrow is
manifestly $THH(R)$-linear. The bottom horizontal arrow is manifestly $R^{\Phi C_p}$-linear, and
since we give the top-right quantity the $THH(R)$-module structure given by composition with $f_0$,
we will need to give the bottom right the $THH(R)$-module structure induced by composition with the
arrow which goes around the bottom and then to the right of
\eqref{eq:f_a-does-nothing-after-collapse}, for $a=0$:
\begin{equation}
\label{eq:f_a-does-nothing-after-collapse}
    \begin{tikzcd}
        THH(R) \ar[r] \ar[d, "f_a"] & THH(R)^{\Phi C_p} \ar[r] &  R^{\Phi C_p} \ar[d, equals] \\
        THH(R) \ar[r] & THH(R)^{\Phi C_p} \ar[r] &R^{\Phi C_p}
    \end{tikzcd}
\end{equation}
Thus we can conclude if the diagram \eqref{eq:f_a-does-nothing-after-collapse} commutes for $a=0$.
But this follows by continuity from the same result for nonzero $a$, which is clear because
the cyclotomic structure is defined levelwise in the simplicial object defining $THH(R)$, as is the
collapse map, and $f_a$ acts by the identity on the realization of the constant simplicial object on
$R^{\Phi C_p}$. \end{proof}

\section{Multisimplicial objects and subdivisions.}

\subsection{Brief recollection of Dold-Kan.}
Given a simplicial object $X_\bullet$ in any abelian category $\mathcal{A}$ there are functors to
chain complexes in $\mathcal{A}$ and natural transformations between them
\[ N(X_\bullet) \to \|X_\bullet\| \to M(X_\bullet)/D(X_\bullet) \]
where $\|X_\bullet\|_k = X_k$ with differential the alternating sum of the face maps, $D(X_\bullet)
\subset M(X_\bullet)$ the subset given by the span of the images of degeneracy maps,
$N(X_\bullet)_k$ the intersection of kernels of all face maps except $d_n$ with differential $(-1)^n
d_n$. Moreover, each natural transformation is homotopy equivalence for every $X_\bullet$ and the
composition is an isomorphism. In particular these functors make sense when $\mathcal{A}$ is chain
complexes in another abelian category; we use the same notation to denote applying the functors
above and then taking direct-sum total complexes. We then see that $N(X_\bullet)$ is the geometric
realization and $\|X_\bullet\|$ is the fat geometric realization of $X_\bullet$ with respect to the
cosimplicial object given by $[n] \mapsto N(\Z[\Delta_{sset}^n])$, where $\Z[Y^\bullet]$ is the
levelwise free abelian group on a cosimplicial set $Y^\bullet$ and $\Delta_{sset}$ is the standard
cosimplicial set defined as the Yoneda embedding functor of $\Delta$. The same construction then
lets us make sense of (fat) realizations of cosimplicial objects in abelian categories, and their
semi-(co)simplicial counterparts.

\begin{remark}\label{rk:dold-kan}The functors $N$ and $\| \cdot \|$ are lax symmetric monoidal via
the Eilenberg-Zilber map. Moreover, there is a right adjoint $D$ to $N$, which is oplax monoidal via
the Alexander-Whitney map, but is not oplax \emph{symmetric} monoidal, because the
Alexander-Whitney map is based on an approximation to the diagonal via a union of faces of the
square of a simplex, and there are multiple possible such choices. This same ambiguity in the choice
of approximation to the diagonal Steenrod squares in Steenrod's original treatment of the subject
\cite{steenrod1947products}. Ignoring the commutativity issue, the functors $N$ and $D$ give
Quillen equivalences of monoidal model categories \cite{schwede-shipley}.
\end{remark}

\subsubsection{Multisimplicial objects.}
A multi-(co)simplicial object in a monoidal category $\mathcal{C}$
enriched in spaces with  monoidal product $\sma$ is a
contra/co-variant functor out of $\Delta \times \ldots \times \Delta$;
we write $\Delta_n$ when the latter category has $n$ $\Delta$ factors,
and we write $\mathbf{\Delta}_n$ for the multi-cosimplicial space
assigning $\times_i \Delta^{a_i}$ to $([a_1], \ldots, [a_n])$; the
latter defines (fat) realizations of multi-(co)-simplicial
objects. Similarly, the multi-cosimplicial abelian group
$N(\Z[(\Delta_{sset}^\bullet)^{\boxtimes n}])$ lets us define (fat)
realizations of multi-(co)simplicial chain complexes. The category
$\Delta_n$ is generated by $n$ variants of coface and codegeneracy
maps.  

\subsection{Ordinal sum.}
There is a functor called the \emph{ordinal sum functor}
\[ \oplus_p: \Delta^p \to \Delta, ([k_1], \ldots, [k_p]) \mapsto [k_1 + \ldots + k_p +p-1]\]
which arises by identifying  $[k_1 + \ldots + k_p +p-1]$ with the ordered set 
$\{0_1 < \ldots < (k_1)_1 < 0_2 < \ldots < (k_2)_2 < \ldots < (k_p)_p\}$ where we think of the
$i$-th element of $([k_1], \ldots, [k_p])$ as $\{0_i < \ldots < (k_i)_i\}$. Pulling back a
simplicial chain complex by the ordinal sum functor defines a corresponding multisimplicial chain
complex. The ordinal sum is related to the prismatic subdivision of a simplex, and this relationship
lets one show results like the following via a generic method which we explain below. 
\begin{proposition}
\label{prop:decalage}
    There is a natural transformation between the functors 
    \[ F \mapsto \| F\|, F \mapsto \|F \circ \oplus_p\|\]
    from simplicial chain complexes to chain complexes, which is a homotopy equivalence on every
    cofibrant simplicial chain complex. 
\end{proposition}
\begin{proof}
    This follows from the fact that for cofibrant simplicial chain complexes, fat geometric
    realization agrees with geometric realization, together with the very general Proposition
    \ref{prop:generic-subdivision-proposition} below. 
\end{proof}

\subsection{Subdivisions.}
\label{sec:ordinal-sums-etc}
We have the well-known identity of functors $\Delta \to Spaces$ \cite{EHLERS2008489} 
\[ \mathbf{\Delta}([n]) = \int^{([a], [b]) \in \Delta \times \Delta} 
 \cosimptri_2([a], [b]) \times \Delta([a] \oplus [b], [n]). \]
This in fact is precisely the prismatic subdivision of the simplex into products of simplices.
Recall that a map in $\Delta$ is degenerate when it is not injective (and thus, when factorizing it
uniquely into a composition of an injective map and a surjective map in $\Delta$, the surjective map
is not the identity.) The essential observation is that every map $[n+1] \to [n]$ is degenerate, and
when $a+b=n$, exactly one of the degeneracy maps with domain $[n+1] \simeq [a] \oplus [b]$ is not
in the image of a degeneracy map in $\Delta \times \Delta$ under an application of $\oplus$.

There is also a map 
\[ \diag: \Delta \to \Delta \times \Delta, \diag([n]) = ([n], [n]), \]
and we have a corresponding identity of functors 
\begin{equation}
\label{eq:ez-space-level}	
 \cosimptri_2([a], [b]) = \int^{[x] \in \Delta } 
 \cosimptri([x]) \times (\Delta \times \Delta)(\diag([x]), [a], [b]). 
\end{equation}

 This identity is better known, and underlies the \emph{Eilenberg-Zilber subdivision} of a simplex
 into a product of simplices. 

We now state a completely general result connecting such identities to general types of
subdivisions: 
\begin{proposition}
\label{prop:generic-subdivision-proposition}
   Let 
   \[ G: D \to C \]
   be a functor, and let 
   \[X: C^{op} \to Spaces, Y: C \to Spaces, Z: D \to Spaces\] 
   be functors, such that 
   \begin{equation}
   \label{eq:subdivision-equation}	
   Y(c) = \int^{d \in D} Z(d) \times C(G(d), c).
   \end{equation}
     Then 
   \begin{equation}
   \label{eq:subdivision-result}	
      \int^{c \in C} Y(c) \times X(c) = \int^{d \in D} Z(d) \times X(G(d)).
   \end{equation}

   The same result holds with spaces replaced by orthogonal spectra. Similarly, writing 
\[ \cosimptri_{ch}: \Delta  \to Ch_*, \cosimptri_{ch}(n) = N(\Z[\Delta^n_{sset}]) \]
where $N: SSet\to Ch$ is the normalized chain complex functor (which is lax symmetric monoidal)  
we have analogous results with spaces replaced by chain complexes and $\phi$ replaced with
$\tilde{\phi}$ everywhere, with the modification that in \eqref{eq:subdivision-equation} the
condition is that we have natural transformation of functors from the right to the left which is
levelwise a weak-equivalence, and the conclusion is modified so that there is a canonical weak
equivalence from the right to the left in \eqref{eq:subdivision-result}.
\end{proposition}
\begin{proof}
    This is an elementary consequence of the Fubini theorem: 
    \[ \int^{c \in C} Y(c) \times X(c) = \int^{c \in C} \int^{d \in D} Z(d) \times C(G(d), c) \times
    X(G(d)) = \int^{d \in D} Z(d) \times X(G(d))\]
    where we have used 
    \[ \int^{c \in C} C(c', c) \times X(c) = X(c')\]
    which is a result that has many names \cite[Theorem 1.3]{meyer-hrm}, which holds in all the
    categories mentioned earlier. 
\end{proof}

Taking normalized chain complexes everywhere for the Eilenberg-Zilber subdivision
\eqref{eq:ez-space-level}, Proposition \ref{prop:generic-subdivision-proposition} thus recovers the
Eilenberg-Zilber map comparing the tensor product of geometric realizations of simplicial chain
complexes with the geometric realization of the levelwise tensor product of such. 

Finally, we observe one more elementary identity: 
\begin{lemma}
\label{lemma:relate-3-subdivision-functors}
    The composition $\oplus \circ \diag$ is the edgewise subdivision functor $\sd_2$, and more
    generally, writing 
    \[ \diag_k: \Delta \to \Delta^k, \diag([n]) = ([n], \ldots, [n])\]
    we have
    \[ \sd_k = \oplus_k \circ \diag_k,\]
    and the resulting formula 
    \[ \int^{[x] \in \Delta} \cosimptri([x]) \times \Delta(\sd_k([x]), [y]) = \cosimptri([y])\]
    realizes the $k$-fold edgewise subdivision of the simplex. 
\end{lemma}
\begin{proof}
    Obvious. (See Lemma 1.1 of \cite{bokstedt1993cyclotomic}.) 
\end{proof}

\begin{remark}
    There are two conventions for edgewise subdivision, one involving the $op$ functor due to Segal \cite{Segal1973-qs}, and the one used in cyclic homology,
    which does not involve the $op$ functor (see Figure 1 of
    \cite{bokstedt1993cyclotomic}). 
\end{remark}

\section{Hochschild (co)homology and other constructions with chain complexes}
\label{sec:hh-appendix}

In this appendix, we recall several standard formulae related to Hochschild (co)homology and explain
their relation to various simplicial constructions of this paper.

\subsection{Explicit Hochschild (co)chain complexes.} Let $R$ be a ring, and let $A_\bullet$ be a
(homologically graded) unital differential graded algebra that is free as an $R$-module. 

The standard Hochschild complex of $A$ is the chain complex

\begin{equation}
    \label{eq:algebraic-hochschild-chain-complex}
    CC^{alg}_*(A_\bullet) = \oplus_{k \geq 0} A \tensor_R A^{\tensor_R k}[k]
\end{equation}
with the convention being that for a chain complex $M_\bullet$ we have $M[k]_j = M_{j-k}$ (so that
$M \mapsto M[1]$ takes the CW chain complex of a CW complex $X$ to the CW chain complex of $\Sigma
X$). This has differential \cite[Section 5.3]{loday2013cyclic}
\[ d = b + \delta.\]
Here,  $\delta$ is the differential on the right hand side of
\eqref{eq:algebraic-hochschild-chain-complex} thought of as a chain complex, where we recall that
the shifts $M \mapsto M[k]$ of chain complexes modify differentials via $d \mapsto (-1)^k d$. The
operator $b$ is the \emph{bar differential}: writing 
\[  \oplus_{k \geq 0} A \tensor_R A^{\tensor_R k}[k] \ni a \tensor a_1 \tensor \ldots \tensor a_k
\text{ as } (a | a_1, \ldots, a_n),  \]
we have
\[ b(a | a_1, \ldots, a_n) = \sum_{i=0}^n (-1)^id^n_i = (aa_1| a_2, \ldots, a_n) + \ldots (a|a_1
\ldots , a_ia_{i+1}, \ldots, a_n) + (-1)^{n} (a_n a | a_1, \ldots, a_{n-1}).\]
Here the maps $d^n_i$ are the face maps of the simplicial object 
\[ HH^\Delta_\bullet(A) = B(A,A,A)_\bullet \tensor_{A \tensor_R A^{op}} A\]
where $B(A,A,A)$ is the simplicial $A \tensor_R A^{op}$-module

\begin{equation}
    \label{eq:two-sided-bar-complex-R-mod}
    B(A,A,A), B(A,A,A)_n = A  \tensor_R A^{ \tensor_R n}  \tensor_R A, d^n_i = id^{\sma i}  \tensor
    m \tensor id^{n-i}, s^{n}_i = (id)^{\sma i+1}\tensor 1_A \tensor id^{n-i};
\end{equation} 
Thus we see that $CC^{alg}_*(A_\bullet)$ is the fat geometric realization of $HH^\Delta_\bullet(A)$.
Similarly, the Hochschild cochain complex is the cochain complex 

\[CC^*_{alg}(A_\bullet) = Hom_{A \tensor_R A^{op}}(\|B(A,A,A)_\bullet\|, A)_{-*},\]
or equivalently, the fat geometric realization of the cosimplicial cochain complex 
\[ HH^\bullet_\Delta(A), \text{ where } HH^k_\Delta(A) = Hom_{A \tensor_R A^{op}}(B(A,A,A)_k, A). \] 

Write $CC_*^{alg, red}(A_\bullet)$ and $CC^*_{alg, red}(A_\bullet)$ for the corresponding (non-fat)
geometric realizations. 

As discussed in Section \ref{sec:comparing-conventions}, to convert from these conventions to
conventions in symplectic geometry, one notes that if one looks at the cohomologically graded dga
$A^\bullet$ corresponding to $A_\bullet$ one has 
\[ CC_*^{symp}(A^\bullet) = CC^{alg}_{-*}(A), CC_{symp}^*(A^\bullet) = CC_{alg}^*(A_\bullet). \]

\begin{remark}
    The above complexes make sense even when $A$ is a non-unital dga, in which case the
    corresponding (co)-simplicial objects are replaced with their semi-(co)-simplicial counterparts.
\end{remark}

\subsection{The cap product.}
In these conventions, the \emph{cap product} is a map of complexes 
\[ \cap_R:  CC_{alg}^i(A_\bullet) \tensor CC^{alg}_j(A_\bullet) \to CC^{alg}_{j-i}(A_\bullet).\]
The map takes the form 
\begin{equation}
    \label{eq:cap-not-bdd-below}
     \prod_k Hom_R(A^{\tensor_R k}, A) \ni (f_k), (f_k) \tensor (a | a_1, \ldots, a_\ell) \mapsto
     \sum_k (a f_k(a_1, \ldots, a_k) | a_{k+1}, \ldots, a_\ell)
\end{equation}
where the terms on the right hand side are interpreted to be zero if $k > \ell$, so the sum is in
fact finite.

To see that the cap product $\cap_R$ is a map of complexes, let us imagine that the differential on
$A_\bullet$ is zero, e.g. that $A$ is a unital associative algebra in degree zero. Neglecting signs,
we have that
\begin{equation}
    \label{eq:d of f cap h}
    \begin{gathered}
     d (f \cap a | a_1, \ldots, a_m) = (af(a_1, \ldots, a_n)a_{n+1} | a_{n+2}, \ldots, a_m)
     \\+\ldots + (af(a_1,\ldots a_n)| a_{n+1} \ldots, a_{n+i}a_{n+i+1}, \ldots, a_n) +\ldots \\+
     (a_m f(a_1, \ldots, a_n) | a_{n+1}, \ldots, a_{m-1}). 
     \end{gathered}
\end{equation}
Computing further we have that 
\begin{equation}
    \label{eq:f cap dhh}
    \begin{gathered}
    f \cap d(a | a_1, \ldots, a_m) = f \cap (aa_1 | a_2, \ldots, a_m) \\+ \ldots + f \cap (a | a_1,
    \ldots, a_{i}a_{i+1}, \ldots a_m) + \ldots \\
    + f \cap (a_m a | a_1, \ldots a_{m-1}).
    \end{gathered}
\end{equation}
We see that the last term of \eqref{eq:f cap dhh} corresponds to the last term of \eqref{eq:d of f
cap h}, and that all except for the first term of \eqref{eq:d of f cap h} correspond to terms of the
second kind in \eqref{eq:f cap dhh}. Finally, we compute 
\begin{equation}
    \label{eq:df of h}
    \begin{gathered}
    (df) \cap (a | a_1, \ldots, a_n) = (aa_1 f(a_2, \ldots, a_{n+1}) | a_{n+2} , \ldots a_m) \\+
    \ldots + af(a_1, \ldots, a_ia_{i+1}, \ldots, a_{n+1}) | a_{n+2}, \ldots, a_m) + \ldots \\
    +(af(a_1, \ldots, a_n)a_{n+1}| a_{n+2}, \ldots, a_m). 
    \end{gathered}
\end{equation}
We have that the first term of \eqref{eq:df of h} corresponds to the first term of \eqref{eq:f cap
dhh}, the last term of \eqref{eq:df of h} corresponds to the first term of \eqref{eq:d of f cap h},
and the terms of the second kind of \eqref{eq:df of h} correspond to the remainder of the terms of
the second kind of \eqref{eq:f cap dhh} that did not already arise from \eqref{eq:d of f cap h}.
Keeping track of all the signs proves the claim when the differential on $A$ is zero, and the
generalization to the case where the differential is nonzero is straightforward. 

\begin{remark}
    These cancellations described above correspond precisely to the diagrams \eqref{diagram1}
    \eqref{diagram2} \eqref{diagram3} \eqref{diagram4} that needed to be checked to show that the
    spectral cap product is well defined.
\end{remark}

\subsection{Cyclic objects and subdivision.}
Now, the simplicial object $HH_\bullet(A_\bullet)$ actually arises from a cyclic object, and so we
can define its $p$-fold subdivision $\sd_p HH_\bullet(A_\bullet)$, which is a simplicial object
valued in $R$-chain complexes with $C_p$-action. We will not review the corresponding theory here in
full \cite[Section 1]{bokstedt1993cyclotomic}, but simply write the result of the construction. We
have that 
\[ \sd_pHH^\Delta_n(A) =HH^\Delta_n(A)^{\tensor_R p},\]
with the $C_p$-action cyclically permuting the $THH^\Delta_n(A)$ factors, and
\begin{equation}
    \label{eq:modified-simplicial-structure-sd_p}
    (d_{\sd_p})^n_i = (d^n_i)^{\tensor_R p} \text{ for } 0 \leq i < n, (d_{\sd_p})^n_n =
    e_{C_p}(d^n_n)^{\tensor_R p}
\end{equation}
where $e_{C_p}$ is the action of the generator of $C_p$. Finally, we have that
\[ (s_{\sd_p})^n_i = (s^n_i)^{\tensor_R p}, 0 \leq i \leq n.\]

The complex $\sd_p CC^{alg}(A_\bullet)$ is the fat geometric realization of the above simplicial
object, the differential is the sum of the usual differential and a `bar differential', and the bar
differential is quite simple: 
\begin{align}
\label{eq:bar-differential-subdivision}
    d(a^1 | \ldots | \ldots| a^k_k) = & (a^1a^1_1 | \ldots |\ldots) + (\ldots | a^2a^2_2 | \ldots) + \ldots +\\ 
     & \ldots \sum_{j=1}^p (a^1 | \ldots a^j_ia^j_{i+1} \ldots ) +
    \ldots\\
    & + (-1)^{kp} (a^p_ka^1 | \ldots ) + (\ldots | a^1_ka^2 | \ldots ) + \ldots + (\ldots|
    a^{p-1}_ka^p | \ldots). 
\end{align}
The $p$-fold cap product in this setting is the map 
\begin{equation}
    \label{eq:cap-p-hochschild-homology}
    \cap^p: (HH^\bullet(A))^{\tensor p} \tensor \sd_pHH_\bullet(A) \to \sd_p HH_\bullet(A)
\end{equation}
given by
\begin{equation}
    \label{eq:p-fold-cap-product-formula}
    \left(\prod_{r_1, \ldots, r_p} f^{r_1}_1 \tensor \ldots \tensor f^{r_p}_p \right)\tensor (a^1 |
    \ldots |\ldots a^p_k) \mapsto \sum_{r_1, \ldots, r_p} (\ldots | a^jf_j(a^j_1, \ldots a^j_{r_1})
    | a^j_{r_1+1}, \ldots a^j_{k} | \ldots ). 
\end{equation}
when $f_i^{r_i} \in (HH^\bullet(A))^{r_i}$ and if for any $j=1, \ldots, p$, $r_j > k$, then the
map is the zero map. The proof that this map is a chain map is almost identical to the proof that
the usual cap product is a chain map; the only potentially interesting subtlety is to note that
terms of the last kind (in the sense of \eqref{eq:bar-differential-subdivision}) in $d (f^{r_1}_1
\tensor \ldots \tensor f^{r_p}_p) \cap (a^1 | \ldots |\ldots a^p_k)$ still match up with terms of
the last kind (in the same sense) in $(f^{r_1}_1 \tensor \ldots \tensor f^{r_p}_p) \cap d (a^1 |
\ldots |\ldots a^p_k)$, and this allows the rest of the argument to go through essentially
unchanged. 

\subsection{Multisimplicial variants of the Hochschild complex.}
\newcommand{\pCC}{{{}_pCC}}

We write $\pCC^{alg}_*(A_\bullet)$ for the latter functor in Proposition \ref{prop:decalage} applied
to $CC^\Delta_\bullet(A)$. Concretely we have that as a graded $R$-module, 

\begin{equation}
\label{eq:p-decalage-hochschild-homology}
\pCC^{alg}_*(A_\bullet) = \bigoplus_{n_1, \ldots, n_p \geq 0 }  A^{\tensor_R n_1 + 1} [n_1]
\tensor_R \cdots \tensor_R A^{\tensor_R n_p + 1} [n_p] \ni (a^1 \tensor a^1_1 \tensor \cdots \tensor
a^1_{n_1} \tensor a^2 \tensor \cdots \tensor a^p_{n_p}).
\end{equation}
In `bar notation', the element on the right hand side of the above equation is written 
\[ (a^1 |a^1_1, \ldots , a^1_{n_1} | a^2 |\ldots, a^p_{n_p}) \text{ with degree } \sum_i |a^i| +
\sum_{i, j} |a^i_j|+1,\]
where $|x|$ denotes the homological degree of an element of pure degree $x \in A_\bullet$. The
differential on $\pCC$ is given by the sum of the internal differential on the right hand side of
$\pCC$ and the bar differential, which takes the form (ignoring signs) 
\begin{equation}
    \label{eq:p-fold-bar-differential}
     b(a^1 |a^1_1, \ldots , a^1_{n_1} | a^2 |\ldots, a^p_{n_p}) = 
(a^1a^1_1 | a^1_2, \ldots) + (a^1 | a^1_1 a^1_2, \ldots ) + \ldots + (\ldots, a^1_{n_1}a^2 | \ldots
) + \ldots + (a^p_{n_p}a^1 | a^1_1, \ldots, ). 
\end{equation}
We see that this has essentially the same form as the differential
\eqref{eq:bar-differential-subdivision} on $\sd_p CC^{alg}_*(A_\bullet)$, but now arbitrary sequence
of `lengths' of tensor products of copies of $A$ are allowed, not just sequences where all lengths
of tensor products are the same. Furthermore, we see that if we give $\pCC^{alg}_*(A_\bullet)$ the
$C_p$ action given by cyclically rotating blocks 
\[ e_{C_p}(a^1 |a^1_1, \ldots , a^1_{n_1} | a^2 |\ldots, a^p_{n_p}) = (-1)^*(a^p | a^p_1 \ldots
a^p_{n_p} | a^1 | \ldots | a^{p-1}_{n_{p-1}} ) \]
then $\pCC^{alg}_*(A_\bullet)$ is a $C_p$-equivariant chain complex.

One can define a $dg$ $p$-fold cap-product, which is a 
$C_p$-equivariant chain map 
\begin{equation}
\label{eq:zihong-cap-product}
	\cap^p_{dg}: CC^\bullet(A)^{\tensor p} \tensor \pCC_\bullet(A) \to \pCC_\bullet(A)
\end{equation}

via the formula \eqref{eq:p-fold-cap-product-formula} but now viewed in this new context. 

\subsubsection{Homotopy Invariance}
\label{sec:homotopy-invariance-of-p-fold-cap-product}
The above formulae are special cases of certain formulae in the world of $A_\infty$-algebras. There
is a standard definition of the Hochschild complex $CC_*(A)$ of an $A_\infty$-algebra $A$; the
definition of $\pCC_*(A)$ in the setting where $A$ is an $A_\infty$-algebra is given in
\cite[Definition 3.16]{chen2024quantum}. There are corresponding Hochschild cochain complexes; the
cap product, and the $p$-fold cap product in that setting is written down in \cite[Eq 2.37,
2.40]{chen2024quantum}; see \cite[Eq 2.18]{chen2024quantum} to verify that these products agree with
ours as written above in the case where $A$ is actually a $dg$-algebra. The equivalent formula
\cite[Eq. 2.30]{chen2024quantum} (corresponding to Lemma \ref{lemma:subdivision} below in the world
of spectra) makes it manifestly clear that the $p$-fold cap product is invariant under
$A_\infty$-equivalences under the induced maps of $p$-fold Hochschild chain complexes, once one
verifies that the usual cap product is invariant under $A_\infty$ equivalences; and this holds
because it is simply the action of $RHom_{A-A-mod}(A,A)$ on $A \tensor^L_{A \tensor A} A$ by acting
on the left factor, as exhibited explicitly for example in \cite[2.19]{chen2024quantum}.

\section{Comparisons between cap products in the spectral setting.}
\subsection{Multi-cap products.}
Fixing $n$, given a multi-cosimplicial object $X^\bullet$ and a pair of multisimplicial objects
$Y_\bullet, Z_\bullet$, a cap pairing on $(X^\bullet, Y_\bullet, Z_\bullet)$ is the data of maps 
\[ \cap_{\vec{p}, \vec{q}}: \; X^{p_1, \ldots, p_n}  \sma Y_{q_1, \ldots, q_n} \to Z_{q_1-p_1,
\ldots, q_n - p_n} \]
whenever $p_i \geq q_i$ for all $i$, satisfying the identities of a cap pairing of simplicial object
in each index separately with respect to the corresponding (co)face and (co)degeneracy maps. An
identical argument to Theorem \ref{thm:generalized-cap} based on applying the prismatic subdivision
to each simplex in $\mathbf{\Delta}_n^{p_1, \ldots, p_n}$ separately establishes 
\begin{proposition}
	\label{prop:generalized-multicap}
Theorem \ref{thm:generalized-cap} holds with (co)simplicial objects replaced by multi-(co)simplicial
objects and the corresponding notion of a cap pairing.
\end{proposition}

The diagonal map $\diag_n: \Delta \to \Delta^n, [n] \to [n]^n$ defines, for every triple of
multi-(co)-simplicial objects $(X^\bullet, Y_\bullet, Z_\bullet)$ and a cap pairing $\tilde{\cap}$
of such, a corresponding triple of (co)simplicial objects $(X^\bullet, Y_\bullet, Z_\bullet)$ via
pullback, as well as a cap pairing of such by setting $\cap_{p, q} = \tilde{\cap}_{\{p\}^n,
\{q\}^n}$. 

\begin{lemma}
\label{lemma:diagonal-pullback-of-multisimplicial-cap-products}
There is a commutative diagram
\[ 
\begin{tikzcd}
Tot(X^\bullet) \sma {|Y_\bullet|} \ar[r] \ar[d] & {|Z_\bullet|} \ar[d, equal]\\
Tot(\diag_n^*X^\bullet) \sma {|\diag_{n}^*Y_\bullet|} \ar[r] & {|\diag_{n}^*Z_\bullet|}\\
\end{tikzcd}
\]
where the horizontal maps are the maps associated to $\tilde{\cap}_{\vec{p}, \vec{q}}$ and $\cap_{p,
q}$, and the vertical maps are defined using the isomorphisms 
\[ |Y_\bullet| = |\diag_n^*Y_\bullet|, 	|Z_\bullet| = |\diag_n^*Z_\bullet| \]
induced by the Eilenberg-Zilber map (Appendix \ref{sec:ordinal-sums-etc}), and the map 
\begin{equation}
\label{eq:diagonal-pullback-multisimplicial-totalizations}	
 Tot(X^\bullet) = Hom_{\Delta_n}(\mathbf{\Delta}_n, X) \to Tot(\diag_n^*X^\bullet) =
 Hom_{\Delta}(\mathbf{\Delta}, \diag_n^*X)
 \end{equation}

given by pullback along the natural transformation 
\[ \mathbf{\Delta} \to \diag_n^* \mathbf{\Delta}_n \]
covering the $\diag_n$ functor which levelwise sends $\Delta^n \to (\Delta^n)^{\times n}$ via the
diagonal map on spaces. 

\end{lemma}
\begin{proof}(Sketch.) We explain the essential commutative digram underlying the claim. 

The quantity $Tot(X^\bullet)$ is defined by an equalizer, the domain of which consists of spectra 
\[ F = \prod_{\vec{a}} f_{\vec{a}} \in \prod_{\vec{a}} Hom(\Delta^{\vec{a}}, X^{\vec{a}}) \]
where  $f_{\vec{a}}$ simply means the projection of the quantity on the right to the corresponding
factor. Write $\tilde{a} = \sum_i a_i, \tilde{b} = \sum_i b_i$. 
Consider the diagram 
\[ 
\begin{tikzcd}
\bigwedge_{i=1}^n \Delta^{a_i + b_i} \times Y^{a_i + b_i} \ar[r] & \bigwedge_{i=1}^n \Delta^{b_i}
\times Z^{b_i} \\
\bigvee_{i=1}^n \Delta^{a_i} \times \Delta^{b_i} \times Y^{a_i + b_i}\ar[u, "i_1"] \ar[ur,
"f_{\vec{a}} \tilde{\cap}_{\vec{a}, \vec{b}}"] & \\
\Delta^{\sum_i a_i} \times \Delta^{\sum_i b_i} \times \bigwedge_{i=1}^n Y^{a_i + b_i} \ar[u] \ar[d]
&
\Delta^{\sum_i b_i} \times \bigwedge_{i=1}^n Z^{b_i} \ar[uu] \ar[dd] \\
\Delta^{\sum_i a_i} \times \Delta^{\sum_i b_i} \times Y^{\sum_i a_i + b_i} \ar[d, "i_2"] \ar[dr,
"f_{\sum_i a_i} \cap_{\tilde{a}, \tilde{a} + \tilde{b}}" ] & \\
\Delta^{\sum_i a_i + b_i} \times Y^{\sum_i a_i + b_i} \ar[r] & \Delta^{\sum_i  b_i} \times Z^{\sum_i
b_i}
\end{tikzcd}
\]
The quantities at the top left and the bottom left are summands of the codomain of the coequalizer
defining the geometric realization of $Y^\bullet$ and $\diag^*_n Y^\bullet$,
respectively. The maps $i_1$ and $i_2$ are the inclusions defined by the prismatic subdivision maps.
The diagonal maps are defined in the natural way composing the action of the corresponding
projection of $F$ with the action of the (multi) cap product, so the horizontal maps descend to the
maps induced by the (multi) cap products defined in Theorem \ref{thm:generalized-cap} and
Proposition \ref{prop:generalized-multicap}. The middle two left vertical maps are associated to
Eilenberg-Zilber subdivisions, as are the right two vertical maps; thus the action on simplices is a
product of degeneracy maps associated to a shuffle as in
\eqref{eq:degeneracy-maps-associated-to-shuffles}, as is the action on the element of the
corresponding multisimplicial object $Y_\bullet$ or $Z_\bullet$. Then the analog of equation 
\eqref{eq:degeneracy-supernatural} in the multisimplicial setting implies that the diagram above
commutes. The rest follows from the (co)equalizer identities defining the geometric realizations and
totalizations of the (multi)-(co)simplicial objects at hand.
\end{proof}
\newcommand{\boxwedge}{\text{\fboxsep=0pt\fbox{$\sma$}}}
Now, given any cosimplicial object $X^\bullet$ we can get a multicosimplicial object $(X^{\boxwedge
n})^\bullet$ which is levelwise the monoidal product of $n$ levels of $X^\bullet$. There is an
induced map of nonequivariant spectra
\begin{equation}
\label{eq:exterior-smash-product-on-totalizations}	
 N^{C_p}_e Tot(X^\bullet) \to Tot((X^{\boxwedge} p)^\bullet) 
 \end{equation}

as in the first map $\alpha$ in the equation in Remark \ref{rk:lax-monoidal-structure-map}.

Moreover, given a cyclic object $Y_\bullet$, the geometric realization of $\oplus_p^* Y_\bullet$ is
canonically a $C_p$-object, with the $C_p$ action specified by the commutativity of the diagrams 
\[ 
\begin{tikzcd}
\Delta^{a_1} \times \ldots \times \Delta^{a_n} \sma Y_{a_1 + \ldots + a_n+p-1} \ar[r] \ar[d] &
{|\oplus_p^* Y_\bullet|} \ar[d] \\
\Delta^{a_2} \times \ldots \times \Delta^{a_n} \times \Delta^{a_1} \sma Y_{a_1 + \ldots + a_n+p-1}
\ar [r] & {|\oplus_p^* Y_\bullet|}
\end{tikzcd}
\]
where the right vertical arrow is the action of a generator of $C_p$, the left vertical arrow
rotates the $\Delta$ factors cyclically and is the action of $a_1+1$ times the generator of the
$C_{a_1 + \ldots + a_n+p-1+1}$ action on $Y_{\oplus_p(\vec{a})}$, and the horizontal maps are the
inclusions into the defining coequalizer.  In the case we are interested in, the above diagram is
simply
\[ 
\begin{tikzcd}
\Delta^{a_1} \times \ldots \times \Delta^{a_n} \sma A^{\sma a_1} \sma \ldots \sma A^{\sma a_p}
\ar[r] \ar[d] & {|\oplus_p^* N^\cyc_\bullet(A)|} \ar[d] \\
\Delta^{a_2} \times \ldots \times \Delta^{a_n} \times \Delta^{a_1} \sma A^{\sma a_2} \sma \ldots
\sma A^{\sma a_p} \sma A^{\sma a_1} \ar[r] & {|\oplus_p^* N^\cyc_\bullet(A)|}.
\end{tikzcd}
\]

Under the isomorphism with the geometric realization of the edgewise
subdivision of $Y_\bullet$, this gives the same $C_p$-action as the
action induced by the levelwise action on $\sd_p Y_\bullet$.  

More simply, rotating the coordinates of the category $\Delta_p$ gives an \emph{isomorphism} of
categories $F: \Delta_p \to \Delta_p$. We then have notions of an $F$-multi-(co)-simplicial object,
which is simply a multi-co(simplicial) object $X$ together with a natural transformation $F^*X \to
X$ such that the $p$-th power of this natural transformation is \emph{equal to} the identity natural
transformation. The geometric realization and totalization of $F$-multi-(co)-simplicial objects is
manifestly a $C_p$-object in the target category. An $F$-cap pairing of multisimplicial objects is a
cap pairing of such which is $F$-equivariant in the obvious manner. The levelwise definition of the
map associated to a cap pairing immediately implies then that 
\begin{lemma}
	The map associated to an $F$-cap pairing of multisimplicial objects is $C_p$-equivariant. 
\end{lemma}

\begin{lemma}
\label{lemma:multi-cap-pairing-comparison}Let $\CC = R-mod$ for a commutative orthogonal ring
spectrum $R$. We will write $\sma$ for $\sma_R$. 
Letting $X^\bullet = C^\bullet_{\cyc, R}(A)^{\boxwedge p}$, and $Y_\bullet = Z_\bullet = \oplus^*_p
C_\bullet^{\cyc, R}(A)$ and defining the $F$-cap product on $(X^\bullet, Y_\bullet, Z_\bullet)$ via
the maps 
\[ \tilde{\cap}_{\vec{a}, \vec{a} + \vec{b}}: C^{a_1}_{\cyc}(A) \sma \ldots C^{a_p}_{\cyc}(A) \sma
A^{\sma a_1 + b_1} \sma A^{\sma a_2 + b_2} \sma \ldots \sma A^{\sma a_p + b_p} \to A^{\sma b_1} \sma
\ldots \sma A^{\sma b_p} \]
so that $\tilde{\cap}_{\vec{a}, \vec{a} + \vec{b}} = \sma_i \cap_{a_i, a_i+b_i}$ up to rearrangement
of terms in the domain and codomain, we have that the two $C_p$-maps 
\[ N^{C_p} Tot(C^\bullet_\cyc(A)) \sma |\sd_p C^{\cyc}_\bullet(A)| \to |\sd_p C^{\cyc}_\bullet(A)| \]
with the first being the map $\cap^p$ of \eqref{eq:first-def-of-spectral-cap-p} defined via the cap
pairing $\cap^p_{m, n}$ of \eqref{eq:cap-pairing-defining-spectral-cap-p}, and the second defined by
composing  \eqref{eq:exterior-smash-product-on-totalizations} acting on the first factor with the
map associated to $\tilde{\cap}_{\vec{p}, \vec{q}}$, agree, under the identifications 
\[ |\sd_p C^{\cyc}_\bullet(A)| = |\diag^*_p \oplus^*_p C^{\cyc}_\bullet(A)| \]
associated to Proposition \ref{prop:generic-subdivision-proposition} and Lemma
\ref{lemma:relate-3-subdivision-functors}.
\end{lemma}
\begin{proof}
This follows immediately by verifying that the conditions of Lemma
\ref{lemma:diagonal-pullback-of-multisimplicial-cap-products} are satisfied. 
\end{proof}

\subsection{Multi-cap products and endomorphisms.}
The following lemma is immediate from the definition of geometric realization, and will be left without proof. 
\begin{lemma}
\label{lemma:subdivision}
    Let $A$ be a cofibrant orthogonal ring spectrum or a semifree differential graded algebra over a
    commutative ring $R$. Then as $(A,A)$-bimodules we have that
    \[ |\oplus^*B(A,A,A)| = |B(A,A,A)|\tensor_A |B(A,A,A)|, \]
    and more generally 
    \[ |\oplus_k^*B(A,A,A)| = |B(A,A,A)| \tensor_A |B(A,A,A)| \tensor_A \ldots \tensor_A
    |B(A,A,A)|\]
    where there are $k$ factors on the right hand side. 
\end{lemma}

\begin{proof}[Proof of Lemma \ref{lemma:cap-product-two-models-comparison}] We give an explicit
diagram at the point-set level realizing the comparison in the homotopy category. 

Assume that $A$ is cofibrant-fibrant as an $R$-algebra and $R$ is cofibrant-fibrant as well. Write
$BA = |B(A,A,A)|$.

Notationally, given an $(A,A)$-bimodule $M$ we write $M \sma_A \cdots$ to mean $M \sma_{A \sma
A^{op}} A$, since we `tensor the left and the right module structures', and we write $Hom_{A,A}$ for
$Hom_{A \sma_R A^{op}}$. Below, when smash products do not have a subscript they are taken over $R$.
The diagram giving the comparison is
\[ \begin{tikzcd}
    Hom_{A,A}(BA, BA) \sma (BA \sma_A \cdots) \ar[r] & BA \sma_A \cdots \\
    Hom_{A,A}(BA, BA) \sma (BA \sma_A BA \sma_A \cdots) \ar[r] \ar[u, "1 \sma p_u"] \ar[d] & (BA
    \sma_A BA \sma_A \cdots) \ar[u, "p_u"] \ar[d] \\
    Hom_{A,A}(BA, A) \sma (BA \sma_A BA \sma_A \cdots) \ar[r] &A \sma_A BA \sma_A \cdots
\end{tikzcd}\]
On the top is the map of \eqref{eq:derived-cap}. The map $p_u$ is the map of Lemma
\ref{lemma:subdivision} defined by the prismatic subdivision with parameter $u$; for any $u\neq 0,1$
this is a homeomorphism, and it is always a homotopy equivalence (it is a weak homotopy equivalence
by continuity in $u$, and it is a homotopy equivalence of $R$-modules because its domain is fibrant
and its codomain is cofibrant as $R$-modules).  In the middle horizontal map one acts by the
bimodule endomorphism on the first copy of $BA$, and similarly with the bottom horizontal map; the
maps going down simply are compositions with the collapse map $BA\to A$ in the appropriate places.
We see that for $u=1$ the top square commutes since for $u=1$ the degeneracy relations mean that
$p_1$ is the map that collapses $BA \to A$ in the second factor. On the other hand, for $u=1/2$,
going from the top left to the bottom right gives precisely the composition of the map $\alpha:
Hom_{A,A}(BA, BA) \to Hom_{A,A}(BA, A)$ with the map associated to the cap pairing. Since $A$ is
fibrant, $\alpha$ is a weak equivalence, as are the maps going down (since $BA\sma_A BA \sma_A
\cdots$ and $A \sma_A BA \sma_A \cdots$ are fibrant-cofibrant as $R$-modules). Thus concludes the
argument; if we wish to articulate the conclusions in terms of zig-zags of weak equivalences, we map to
the bottom square from the analog of the bottom square with $BA \sma_A BA \sma_A \cdots$ replaced by
$BA \sma_A \cdots$ everywhere by acting on this factor via $p_{u'}$ for $u'=1/2$, and note that for
$u=1, u'={1/2}$ the resulting diagram commutes on the nose and gives a zig-zag of weak equivalences
realizing the desired relation between the maps.
\end{proof}

\begin{remark}
The following somewhat informal discussion may be helpful for readers more familiar with chain
complexes than with simplices. Let $A$ be an associative $R$-algebra and $HA$ the corresponding
$HR$-algebra. Write 
\[ \begin{gathered} (a_0 | a_1, \ldots, a_n) \in CC_*^{alg, red}(A) \\
 	(a_0 | a_1, \ldots, a_n | b_0 | b_1, \ldots b_m) \in {}_2CC_*^{alg, red}(A) \\
 	(a_0| a_1, \ldots, a_n | a'_n) \in A \tensor \bar{A}^{\tensor n} \tensor A \subset |B(A,A,A)| =:
 	BA
 \end{gathered}\]
for typical pure tensor elements of the corresponding geometric realizations. 
 A Hochschild cochain $\bar{f}_n$ with $n$ inputs corresponds to an $A,A$-bimodule map $f$ from
 $\bar{B}(A,A,A)_n$ to $A$ via $\bar{f}(\vec{a}) = f(1,\vec{a},1)$. The prismatic subdivision then
 corresponds to picking the preimage of $(a_0 | a_1, \ldots, a_R)$ under the homotopy equivalence
 ${}_2CC_*^{alg, red}(A) \to CC_*^{alg, red}(A)$ via 
\[ (a_0 | a_1, \ldots, a_R) \rightsquigarrow \sum_{n+m = R} (a_0 | a_1, \ldots, a_n | 1 | a_{n+1},
\ldots a_R) \]
where the image under the identification 
\[{}_2CC_*^{alg, red}(A) =  BA \tensor_A BA \tensor_A \cdots \]
can be written as 
\begin{equation}
\label{eq:prismatic-subdivision-chain-complex}	
 \sum_{n+m=R} (a_0 | a_1, \ldots, a_n | 1 )\tensor_A (1 | a_{n+1}, \ldots, a_R | 1) \tensor_A
 \cdots. 
 \end{equation}
Then the cap product 
\[ (a_0 | a_1, \ldots, a_R) \mapsto \sum_{n=0}^R(a_0 \bar{f}_n(a_1, \ldots, a_n)| a_{n+1}, \ldots,
a_R) \]
corresponds to picking the preimage of $(a_0| \ldots, a_R)$ as above and then applying $f = \prod_i
f_i$ to the first factor in \eqref{eq:prismatic-subdivision-chain-complex}. In particular one sees
that action of the last coface map on $f$ agrees with the action of the first face map in the second
copy of $BA$ in ${}_2CC_*^{alg, red}(A)$. 
\end{remark}

Let us keep the notation from the previous lemma here.
\begin{lemma}
\label{lemma:p-cap-spectra-comparison-endo-to-multi}
Let $\cap^p_{endo}$ be the $C_p$-equivariant map 
\[ Hom_{A,A}(BA, BA)^{\sma p} \sma (\underbrace{BA \sma_A \ldots BA}_p \sma_A \cdots) \to
(\underbrace{BA \sma_A \ldots A}_p \sma_A \cdots) \]
where one acts by the $k$-th bimodule map on the $k$-th copy of the bimodule in the tensor product,
and let 
\[\cap^p_{multi}: Hom_{A,A}(BA, A)^{\sma p} \sma (\underbrace{BA \sma_A \ldots BA}_p \sma_A \cdots)
\to (\underbrace{BA \sma_A \ldots BA}_p \sma_A \cdots) \]
 be the map with the same domain and codomain defined in Lemma
 \ref{lemma:multi-cap-pairing-comparison} associated to the cap pairing of corresponding
 multisimplicial objects. Then the resulting maps, viewed as maps in the category of
 $C_p$-equivariant $R$-modules, are related by a commutative zig-zag of weak equivalences.
\end{lemma}
\begin{proof}
This follows identically to the proof of Lemma \ref{lemma:cap-product-two-models-comparison} by
applying the prismatic subdivision to \emph{each} copy of $BA$ in $BA \sma_A \ldots \sma_A \cdots$. 
\end{proof}

\section{Cap product comparison between spectral and chain complex settings.}

Let $R$ be a discrete commutative ring, and let $HR$ be the
corresponding orthogonal ring spectrum. Let $B$ be a $dga$ over $R$,
and let $A$ be a corresponding orthogonal $HR$-algebra spectrum. Let
$\tilde{A}$, $\tilde{B}$ be the cofibrant-fibrant replacement of $A$,
$B$ in the corresponding bimodule category. 

\begin{proposition}
\label{prop:compare-spectral-and-dg-p-cap-products-module-action-variants}
The two maps 
\[
Hom_{A}(\tilde{A}, \tilde{A})^{\sma_{HR} p} \sma_{HR}
(\tilde{A}^{\sma_{A} p})\sma_{A \sma_{HR} A^{op}}A \to
(\tilde{A}^{\sma_{A} p})\sma_{A \sma_{HR} A^{op}}A
\]
and 
\[
Hom_B(\tilde{B}, \tilde{B})^{\tensor_R p} \tensor_R
(\tilde{B}^{\tensor_B p})\tensor_{B \tensor_R B^{op}} B \to
(\tilde{B}^{\tensor_B p})\tensor_{B \tensor_R B^{op}} B
\]
agree as $C_p$-equivariant maps in the $\infty$-categories $HR$-mod
and $Mod_R$ under the equivalence between these $\infty$-categories.  
\end{proposition}

\begin{proof}
First, observe that a monoidal Quillen equivalence of symmetric monoidal
categories $\aC$ and $\aD$ induces a Quillen equivalence on categories
of $C_p$-objects in $\aC$ and $\aD$, where we use the projective model
structure with weak equivalences and fibrations detected on the
underlying object.  In our case, since the underlying comparison
of~\cite{schwede2003stable} is lax symmetric monoidal, we see that for
a cofibrant object $X$ with image $Y$ under this comparison, the $C_p$-object $X^{\sma_{HR} p}$ corresponds
to the $C_p$-object $Y^{\tensor_R p}$.

Thus the objects 
\[(\tilde{A}^{\sma_{A} p})\sma_{A \sma_{HR} A^{op}}A \text{ and }
(\tilde{B}^{\tensor_B p})\tensor_{B \tensor_R B^{op}} B\] 
correspond as $C_p$-objects, since the first is equal to the object
\[
\tilde{A}^{\sma_{HR} p}\sma_{(A \sma_{HR} A^{op})^{\sma_{HR} p}}
A^{\sma_{HR} p}
\]
where the bimodule structure on $\tilde{A}^{\sma_{HR}}$ is multiplying
on the 'left of each factor' or on the 'right of each factor' in the
appropriate manner, and similarly for the corresponding analog with
$A$ replaced by $B$. 
\end{proof}

\begin{proposition}
\label{prop:zihong-ours-comparison}
We have that the $dg$ $p$-fold cap product map $\cap^p_{dg}$ of
\eqref{eq:zihong-cap-product} agrees on homotopy groups with action of
$\cap^p$ defined via the cap pairing, under the natural comparisons on
homotopy groups $THH_*(A/HR) \simeq HH_*(B/R), THC^*(A/HR) \simeq
THH(B/R)$. 
\end{proposition}

\begin{proof}
This follows from Proposition
\ref{prop:compare-spectral-and-dg-p-cap-products-module-action-variants},
together with the comparisons recalled in
Appendix~\ref{sec:homotopy-invariance-of-p-fold-cap-product} on the dg
side and Lemma~\ref{lemma:p-cap-spectra-comparison-endo-to-multi} and Lemma
\ref{lemma:multi-cap-pairing-comparison} on the spectral side.
\end{proof}

\section{Completed Spectra}

In this section we prove the remaining theorems claimed in the text. 

\subsection{Power series rings (and variants) are cyclotomic bases.}
\label{sec:cyclotomic-bases-exist}

In this section, we prove
Proposition~\ref{prop:weak-cyclotomic-bases}. The discussion in
Section~\ref{sec:spectral-to-classical} has established the result for 
$\tilde{R} = \SS[t], \SS[t, t^{-1}]$.  

\begin{lemma}\label{lemma:weak-cyclotomic-base-for-truncated-polynomial}
The $\EE_\infty$ ring $\SS[x]/x^n$ is a standard cyclotomic base. 	
\end{lemma}

\begin{proof}
This follows from Hesselholt's computation in Lemma 3.1.6 of
\cite{hesselholt-p-typical-curves} that  
\begin{equation}
    \label{eq:THH-truncated-polynomial-ring}
    THH(\SS[x]/x^a) = \Sigma^\infty_+ N^{cy}_{\wedge}(\Pi_a) = \Sigma^\infty_+ \left\{
    \bigvee_{k=1}^{a-1} \Sigma^\infty_+(S^1_k) \vee  \bigvee_{k=a}^\infty  N^{cy}_{\wedge}(\Pi_a,
    k)\right\}.
\end{equation}
Here, $\Pi_a$ is the pointed monoid $\{X^0, X^1, \ldots, X^{a-1}, *\}$. The space
$N^{cy}_{\wedge}(\Pi_a, k)$ is the union of the $m$-simplices $X^{i_0} \wedge \ldots \wedge
X^{i_m}$ with $i_0 + \ldots + i_m = k$ (where here we are summing in $\mathbb{N}$ in the cyclic bar
construction), and for $0 \leq k \leq a-1$ this is $S^1_k$, as before. For all $k$, the equivariant
homotopy types of the $N^{cy}(\Pi_a, k)$ are explicitly known \cite{Hesselholt-cyclic-polytopes},
but we only need the computations of \cite[Lemma 3.1.6]{hesselholt-p-typical-curves}.  The above
characterization shows that the $\mathcal{I} = \bigvee_{k=a}^\infty  N^{cy}_{\wedge}(\Pi_a, k)$ form
a submonoid of $N^{cy}_{\wedge}(\Pi_a)$. The cyclotomic structure map is induced from the maps 
\[ N^{cy}_{\wedge}(\Pi_a, k) \to N^{cy}_{\wedge}(\Pi_a, pk)^{C_p} \]
which sends the simplex $X^{i_0} \wedge \ldots \wedge X^{i_k}$ to $((X^{i_0})^{\wedge p} \wedge
\ldots \wedge (X^{i_k})^{\wedge p})^{C_p}$. The collapse map $THH(\SS[x]/x^a) \to \SS[x]$ is induced
from the map of topological abelian $S^1$-spaces which, on simplices, is given by the map 
\[X^{i_0} \wedge \ldots \wedge X^{i_k} \mapsto \prod_{j=0}^k X^{i_j}, \] 
where the product is taken in the monoid $\Pi_a$. Thus we see that the space-level cyclotomic
structure map restricted to $\mathcal{I}$ factors through $*$.

Define $THH(\SS[x]/x^a)'$ to be the spectrum where one collapses $\mathcal{I}$ in
\eqref{eq:THH-truncated-polynomial-ring}. The discussion above immediately shows that the collapse
maps $THH(\SS[x]/x^a)$ factorize through the projection to $THH(\SS[x]/x^a)'$, and the restriction
of the cyclotomic structure map to $THH(\SS[x]/x^a)'$ fits into the corresponding commutative square
with the endomorphism of $\SS[x]/x^a$ induced by $x \mapsto x^p$. This concludes the proof. 
\end{proof}

The computation above also shows
\begin{lemma} \label{lemma:alternative-filtration}
The map
\[THH^\completed(\SS[[x]]) := \varprojlim_a THH(\SS[x]/x^a) \to \varprojlim_a THH(\SS[x]/x^a)'\] is
an equivalence (of $\EE_\infty-S^1$-ring spectra).
\end{lemma}

\begin{proof}(Proof of Proposition \ref{prop:weak-cyclotomic-bases} for $\tilde{R} = \SS[[x]]$)
This follows from Lemma \ref{lemma:weak-cyclotomic-base-for-truncated-polynomial} by taking inverse
limits and using the naturality of the diagrams 

\[
\begin{tikzcd}
    THH(\SS[[x]]) \ar[r] \ar[d]  & \SS[[x]] \ar[d]\\
    THH(\SS[x]/x^n) \ar[r] & \SS[x]/x^n
\end{tikzcd}
\]
together with the fact that the Tate construction commutes with these inverse limits (since it is an
inverse limit of smash products with stunted lens spaces, see, e.g. \cite[Lemma
3.2]{burklund2024note}). 
	
\end{proof}

\begin{proof}(Proof of Proposition \ref{prop:weak-cyclotomic-bases} for $\tilde{R} = \SS((x))$)

Consider the defining diagram 
\begin{equation}
\label{eq:cyclotomic-structure-diagram-for-formal-power-series}	
 \begin{tikzcd}
 THH(\SS[[x]]) 	\ar[r] \ar[d] & \SS[[x]] \ar[d]\\
 THH(\SS[[x]])^{tC_p} \ar[r] & \SS[[x]]^{tC_p} = \SS_p[[x]]
 \end{tikzcd}
 \end{equation}
There is a map of $\EE_\infty-S^1$-rings $f: THH(\SS[x]) \to THH(\SS[[x]])$ and also $g: THH(\SS[x])
\to THH(\SS[x,x^{-1}])$. We use $f$ to take the pushout of the diagram above with respect to $g$;
the resulting diagram, after forgetting the right vertical map, maps to
\begin{equation}
\label{eq:cyclotomic-structure-diagram-for-formal-laurent-series}
 \begin{tikzcd}
 THH(\SS((x))) 	\ar[r] \ar[d] & \SS((x)) \\
 THH(\SS((x)))^{tC_p} \ar[r] & \SS((x))^{tC_p} = \SS_p[[x]].
 \end{tikzcd}
 \end{equation}
 with all squares commuting. 
Now $THH(\SS[x]) \to THH(\SS[x,x^{-1}])$ is, as a map of $\EE_\infty$-rings, the localization given by
inverting $x \in \pi_0(THH(\SS[x]))$ \cite{mathew2017thh}. With this in mind, we can verify that the
maps from the pushouts
\[ THH(\SS[[x]])[x^{-1}] \to THH(\SS((x))) \]
\[ \SS[[x]][x^{-1}] \to \SS((x)) \]
\begin{equation}
\label{eq:localization-A}	
 \SS[[x]]^{tC_p}[(x^p)^{-1}] \to \SS((x))^{tC_p} 
\end{equation}
are all equivalences, and that the diagram 
\[
\begin{tikzcd}
THH(\SS[[x]])^{tC_p}[\phi(x)^{-1}] \ar[d] \ar[rd] & \\
THH(\SS((x)))^{tC_p} \ar[r] & \SS((x))^{tC_p} 	= \SS_p((x))
\end{tikzcd}\]
commutes, with \eqref{eq:localization-A} being an equivalence because the Tate construction agrees
with $p$-completion for these spectra, and the diagram above commuting because the image of
$\phi(x)$ in $\pi_0(\SS((x))^{tC_p})$ is $x^p$. 
The fact that \eqref{eq:cyclotomic-structure-diagram-for-formal-power-series} maps to
\eqref{eq:cyclotomic-structure-diagram-for-formal-laurent-series} in a commutative manner implies,
by the universal properties of the homotopy pushouts, the left vertical and the top horizontal map
of \eqref{eq:cyclotomic-structure-diagram-for-formal-laurent-series} agrees with the corresponding
maps on pushouts under the equivalences above. We have produced the desired map $\SS((x)) \to
\SS((x))^{tC_p}$ by localizing the corresponding map for $\SS[[x]]$. 
\end{proof}

\subsection{Proof of Theorem \ref{thm:general-calculus-diagram} for $\kk[[x]]$}
We first note that there is a spherical lift of the diagram \eqref{eq:general-calculus-diagram} in
this case. Indeed, we have a diagram 
\begin{equation}
	\label{eq:general-calculus-diagram-spherical-formal-power-series}
	\begin{tikzcd}
	THH(\SS[[x]]) \ar[r] \ar[d] & THH^\completed(\SS[[x]]) \ar[r] \ar[d, "\phi^{\completed}"]  &
	\SS[[x]] \ar[r] \ar[d, "F"] & \hat{\Omega}^1_{\SS[[x]]/[[x]]}[2] \ar[d, "F'_{\SS}"]\\
	THH(\SS[[x]])^{tC_p} \ar[r] & THH^\completed(\SS[[x]])^{tC_p} \ar[r] & \SS[[x]]^{tC_p} \ar[r] &
	(\hat{\Omega}^1_{\SS[[x]]/[[x]]})^{tC_p}[2]
	\end{tikzcd}
\end{equation}
by taking the inverse limit the cyclotomic structure maps for $THH(\SS[x]/x^n)$, using the fact that 
\begin{equation}
\label{eq:tate-construction-for-formal-power-series-computation}
THH^\completed(\SS[[x]])^{tC_p} = \varprojlim (THH(\SS[x])/x^n)^{tC_p}	
\end{equation}
which defines the left two squares of $\EE_\infty$-$S^1$-rings in
\eqref{eq:general-calculus-diagram-spherical-formal-power-series}, and then taking the cone on the
middle maps to define the map on the right. To see the equivalence
\eqref{eq:tate-construction-for-formal-power-series-computation}, we see that there is a map from
the left hand side to the right hand side induced by the colimit-limit exchange map 
\[ (\varprojlim THH(\SS[x]/x^n))_{hC_p} \to \varprojlim  (THH(\SS[x]/x^n)_{hC_p}), \]
which itself is an equivalence of bounded below spectra on integral homology by Lemma
\ref{lemma:lawson-lemma} and then computing using Lemma \ref{lemma:alternative-filtration}
(essentially using the fact that the resulting $C_p$-complexes are uniformly bounded in degree from
above and below.). 

We can then take the smash product of all vertical maps in
\eqref{eq:general-calculus-diagram-spherical-formal-power-series} with respect to the canonical map
$\kk \to \kk^{tC_p}$. Again, Lemma \ref{lemma:alternative-filtration}, via geometric argument of
Section \ref{eq:p-fold-covers-on-free-loop-space}, allows us to compute the homotopy groups of all
resulting objects,  and the action of $\phi^{\completed} \sma can$ on homotopy groups, and thus the
action of $\pi_*(F')$ on the generator of $\widehat{\Omega}^1_{\kk[x]/\kk} :=
\hat{\Omega}^1_{\SS[[x]]/[[x]]}[2] \sma \kk$. We conclude the desired factorization of $F'$
as in the proof of Theorem \ref{thm:general-calculus-diagram} in the earlier cases. 

\subsection{Proof of Theorem \ref{thm:general-calculus-diagram} for $\kk((x))$}
We note now that the complex $HH^\completed(\kk[[x]]/\kk)$ introduced earlier actually computes
$THH^\completed(\SS[[x]]) \sma \kk$ precisely because the strict map of chain complexes
\[ HH^{\completed}(\kk[[x]]) \to \varprojlim HH(\kk[x]/x^n) \]
satisfies the Mittag-Leffler condition, and so computes the homotopy inverse limit. Let us focus on
the smash product of the diagram \eqref{eq:general-calculus-diagram-spherical-formal-power-series}
with respect to $\kk \to \kk^{tC_p}$, and then push out this resulting diagram with respect to
$HH(\kk[x]) \to HH(\kk[x,x^{-1}])$  as in the proof of  Proposition \ref{prop:weak-cyclotomic-bases}
for $\tilde{R} = \SS((x))$. Given the pushout equivalences in the proof of that theorem, as well as
those in \eqref{eq:localization-equivalences}, all that remains is to show that the pushout map
 	\begin{equation}
 	\label{eq:localization-final}
 	 HH(\kk[x,x^{-1}])^{tC_p} \tensor_{HH(\kk[x])^{tC_p}} HH^\completed(\kk[[x]])^{tC_p} \to
 	 HH^\completed(\kk((x)))^{tC_p}
 	\end{equation}

is an equivalence. 

But the map $f: HH(\Z[x]) \to HH(\Z[x,x^{-1}])$, as a map of $\EE_\infty$-rings, can also be thought
of as the localization map $HH(\Z[x]) \to HH(\Z[x])[x^{-p}] \xrightarrow{\sim} HH(\Z[x,x^{-1}])$ of
$\EE_\infty$ rings, by the corresponding computation on homotopy groups. In fact, let us take the
$p$-fold subdivision of the corresponding cyclic objects, and view the map $f$ as arising from a map
of simplicial commutative rings in $C_p$-equivariant abelian groups. The element in level zero
$x^{\tensor p} \in \underline{HH}_0(\Z[x])$ is strictly $C_p$-invariant, and the map $f$ factors on
the nose through the $C_p$-equivariant map $B: HH(\Z[x])[(x^{\tensor p})^{-1}] \to HH(\Z[x,x^{-1}])$;
the element $x^{\tensor p}$ is seen to be sent to $x^p \in HH_0(\Z[x,x^{-1}])$, so we see that $B$ is
an equivalence of $C_p$-equivariant $\EE_\infty$ rings.

 	We conclude that \eqref{eq:localization-final} is an equivalence since the map
 	$HH(\Z[x])^{tC_p} \to HH(\Z[x,x^{-1}])^{tC_p}$ is equivalent to $B^{tC_p}$, and the Tate
 	construction commutes with colimits of $C_p$-equivariant maps of chain complexes which are
 	homologically bounded above. 

\section{A review of solid spectra.}
\label{sec:solid-spectra}

The arguments of the previous section can be reinterpreted using the language of \emph{solid spectra},
first described in the literature in \cite{clausen-duality-and-linearization}. This theory extends
some of the basic ideas of condensed mathematics \cite{condensed-lectures} to the setting of
spectra; the essential purpose of the method is to give a straightforward language for talking about
homological-algebra invariants of \emph{completed} or \emph{analytic} objects (in the sense of
``analytic geometry'') such as $\SS((x))$. 

    As we have seen, the methods of condensed mathematics are \textbf{not required} for any of the
    final claims of the paper; however, believing in the existence of a reasonable `completed variant
    of $THH(\SS((x)))$' was very helpful for the author to find the relevant arguments. While we in
    the end are able to `define' this object without appealing to any condensed mathematics, it is
    really most naturally seen in that light.
    
\subsection{Sheaves on profinite sets.}
Recall that a topological space $X$ is called \emph{extremally disconnected} if the closure of every
open subset of $X$ is open. 

Let $\kappa$ be an uncountable strong limit cardinal. 
Let $ExtDisc^\kappa$ be the category of compact Hausdorff extremally disconnected spaces which are
$\kappa$-small (i.e. have cardinality strictly less than $\kappa$) and continuous maps between them.
Given an extremally disconnected set which is a disjoint union of a finite collection of subspaces
\[ S = \sqcup_{i=1}^k S_i\]
which are necessarily extremally disconnected themselves, we say that $\{S_i \to S\}_{i=1}^k$ are a
covering in $ExtDisc$. This defines a Grothendieck topology  on $ExtDisc^\kappa$. 

Let $\CC$ be a pointed $\infty$-category admitting finite products.  Let $Cond^\kappa\CC$ be the full
subcategory of $Fun^{op}(ExtDisc^\kappa, \CC)$ on the objects $F$ such that 
\begin{itemize}
    \item $F(pt) = *$, 
    \item The map $F(S_1 \sqcup S_2) \to F(S_1) \times F(S_2)$ is an equivalence for all $S_1, S_2$. 
\end{itemize}
Equivalently, this is 
\begin{itemize}
    \item The $\infty$-category of hypercomplete objects of the $\infty$-topos of sheaves on
    $ExtDisc^\kappa$, or 
    \item The $\infty$-category of hypercomplete objects of the $\infty$-topos of sheaves on the
    category $ProFin^\kappa$ of $\kappa$-small profinite sets, equipped with the Grothendieck
    topology with finite jointly surjective families of maps as covers (this follows \cite{HTT} and
    a variant of \cite[Proposition 2.7]{condensed-lectures}). 
\end{itemize}

\begin{remark}
    Recall that a profinite set is a \emph{cofiltered limit} of finite sets.
\end{remark}

\begin{remark}The choice of $\kappa$ does not matter in the following sense. 
For every pair of cardinals $\kappa' > \kappa$, there is the inclusion functor $ProFin^\kappa \to
ProFin^{\kappa'}$, giving rise to a forgetful functor $Cond^{\kappa'}\CC \to Cond^\kappa\CC$ and its
(fully faithful) left adjoint \cite[Proposition 2.6]{condensed-lectures}. We summarize the argument:
Indeed this construction, by the adjoint functor theorem and the formula for left adjoints in terms
of left Kan extensions, is the functor 
\[ Cond^\kappa\CC \to Cond^{\kappa'}\CC, T \mapsto (\tilde{S} \mapsto \varinjlim_{\tilde{S} \to S}
T(S))^\#\]
where the colimit is taken over the category of all $\kappa$-small profinite sets admitting a map
from $\tilde{S}$, and the $\#$ denotes sheafification. This left adjoint is fully faithful because
the unit of the adjunction is the identity by direct inspection of the formula above.  The above
colimit is $\kappa$-filtered and so also commutes with $\kappa$-small limits. One can verify that
the sheafification function at the end is unnecessary (because finite-disjoint-union decompositions
on $\tilde{S}$ arise from the same on some $S$). 
\end{remark}

\begin{remark}
There is also the category $ProFin^{light}$ consisting of the profinite sets which are sequential
inverse limits of finite sets. This category, with the Grothendieck topology in which covers are
finite jointly surjective maps, gives rise to the category $Cond^{light}\CC$ of hypercomplete
objects in the $\infty$-topos of sheaves on $ProFin^{light}$, which is introduced in
\cite{clausen-duality-and-linearization}. The same arguments as earlier show that the left adjoints
to the forgetful functors $Cond^\kappa\CC \to Cond^{light}\CC$ exist and are fully faithful. 
\end{remark}

\begin{remark}
    By definition, sheafification of a presheaf on $ProFin^\kappa$ does not change its value at a
    point.
\end{remark}

Let $\tilde{R}$ be a condensed $\EE_1$-ring spectrum. We define $THH^{cond}(\tilde{R})$ to be the
sheafification of the presheaf $S \mapsto THH(\tilde{R}(S))$. We note, however, that this latter
presheaf is already a sheaf; indeed, we can check the sheaf condition on extremally disconnected
sets, and the fact that the map 
\[ THH(\tilde{R}(S \sqcup T)) = THH(\tilde{R}(S) \times \tilde{R}(T)) \to THH(\tilde{R}(S)) \times
THH(\tilde{R}(T)) \]
(using the fact that direct sums agree with direct products in this case) is an equivalence follows
from the fact that $THH$ is a localizing invariant. Since limits and colimits are computed levelwise
and the Tate construction commutes with finite products, we conclude that
$THH^{cond}(\tilde{R})^{tC_p}$ defined levelwise is also already a sheaf, and thus, 
\begin{lemma}
    The condensed spectrum $THH^{cond}(\tilde{R})$ lifts to an $\EE_\infty$-algebra in cyclotomic
    condensed spectra.
\end{lemma}

\begin{remark}
    If we worked with light condensed spectra, we would not be able to use the above precise
    argument, because we could not utilize the trick of restricting to extremally disconnected sets
    throughout. However, we expect that this lemma, and the remaining lemmata of this section,
    remain true in the setting of light condensed spectra. 
\end{remark}

\subsection{Solid spectra}

For any profinite set $S = \varprojlim S_i$, write $\Z[S]^\solid = \varprojlim \Z[S_i]$ for the
corresponding condensed abelian group. A condensed abelian group $A$ is \emph{solid} if for all
profinite sets $S$ and all maps $f:S \to A$, there exists a unique map $\Z[S]^\solid \to A$
extending $f$.

A condensed spectrum is solid if all of its homotopy groups are solid.

\begin{proposition}
    Solid spectra are closed under all limits, colimits, and internal mapping objects in condensed
    spectra. The free solid spectrum on a profinite set $S = \varprojlim S_i$ is $\varprojlim
    \SS[S_i].$
\end{proposition}
\begin{proof}
    This follows exactly as in the proof of \cite[Theorem 13.3]{clausen-duality-and-linearization},
    with the modifications that we need to show that 
    \begin{itemize}
        \item arbitrary products in condensed spectra are exact, and that 
        \item arbitrary products of hypercovers are hypercovers. 
    \end{itemize}
    We first note that cardinality assumption in \cite[Lemma 4.9]{clausen-duality-and-linearization}
    is unnecessary, in the sense that given any cofiltered limit of condensed anima with surjective
    transition maps, the map from the limit to each of the objects in the limit diagram is surjective,
    since the same holds in the category of profinite sets. Thus, the argument in \cite[Lemma
    4.10]{clausen-duality-and-linearization} for $(1) \Leftrightarrow (4)$ implies that arbitrary
    products of $d$-connected maps in condensed anima are $d$-connected. Thus the same holds for
    condensed spectra by applying $\Omega^\infty$ to all shifts of the given condensed spectra. 

    The same trick works to prove that an arbitrary product of hypercovers is a hypercover in
    condensed anima, which reduces to proving that an arbitrary product of surjections of condensed
    anima is a surjection, which follows by testing on all $S$-valued points for all extremally
    disconnected $S$, and using the fact that this statement is true in ordinary anima. Thus the
    same holds for all hypercovers in condensed spectra by applying levelwise $\Omega^\infty$ to all
    levelwise shifts of the hypercovers.
\end{proof}

In particular, there is a solidification functor given by the left adjoint to the forgetful functor
from condensed spectra to solid spectra, and given a condensed spectrum $C = \colim \SS[S_i]$ where
the $S_i \in ProFin^\kappa$, its solidification is $C^\solid = \colim \SS[S_i]^\solid$.

As in \cite[p. 115]{clausen-duality-and-linearization}, the fact that $M^\solid = 0$ implies
that $(M \tensor_\SS N)^\solid = 0$ for any condensed spectra $M, N$ implies that there is a
unique symmetric monoidal structure on solid spectra making solidification into a symmetric monoidal
functor, given by $M \tensor^\solid_\SS N = (M \tensor_\SS N)^\solid$. In particular, we have that

\begin{equation}
    \label{eq:solid-tensor-product-calculation}
    \SS[S]^\solid \tensor^\solid_\SS \SS[T]^\solid = \SS[S \times T]^\solid
\end{equation}
for any $S, T \in ProFin^\kappa$. 

\subsection{Solid $THH$ is cyclotomic and some computations.}
Using the arguments above, we can immediately see
\begin{lemma}
    Given a solid $\EE_1$-algebra $A$, we have a condensed $\EE_1$-algebra structure on the
    underlying condensed spectrum $A^{cond}$ of $A$ due to the symmetric monoidality of
    solidification. The solid $THH^\solid(A)$, defined as the solidification of
    $THH^{cond}(A^{cond})$, is cyclotomic, and receives a map of cyclotomic spectra from $A$, such
    that when $A$ is $\EE_\infty$, the underlying map of condensed $\EE_\infty-S^1$-spectra $THH(A)
    \to A$ factors through $THH(A) \to THH^\solid(A)$.
\end{lemma}
\begin{remark}   
    Note that sheafification of a functor from $ProFin^\kappa \to Sp$ does not change its value on a
    point, since there are no sheaf conditions to check. Similarly, the symmetric monoidal structure
    on profinite spectra is performed levelwise. However, the solidification functor can change the
    value of a condensed spectrum on the point in a radical manner. In particular, 
    \[ THH^{cond}(A)(pt) = THH(A), \SS[[x]](pt) = \SS[[x]],\]
    while $THH^{\solid}(A)(pt)$ is more challenging to access, although comprehensible in
    certain cases due to \eqref{eq:solid-tensor-product-calculation}.
\end{remark}
\begin{proof}

    The first statement is elementary. Given a cyclotomic condensed spectrum $X \to X^{tC_p}$, we
    can compose the map to the solidification inside the Tate construction, and then apply
    solidification to the resulting map to get a map $X^\solid \to
    ((X^\solid)^{tC_p})^\solid \simeq (X^\solid)^{tC_p}$, with the last
    equivalence holding because solid spectra are closed under limits and colimits; this proves the
    second statement. 

    Since solidification is a left adjoint, it commutes with colimits, so by the symmetric
    monoidality of solidification, $THH^\solid(A)$ is the geometric realization of the usual
    cyclic object defined by $A$ in solid spectra. This shows that $THH^\solid(A)$ is an
    $\EE_\infty$-$S^1$-algebra, and the statement about the factorization. Finally, to see that
    $THH^\solid(A)$ is a cyclotomic $\EE_\infty$-algebra, we use the fact that solidification
    and the Tate construction are each (lax) symmetric monoidal, so the functor from cyclotomic
    condensed spectra to cyclotomic spectra constructed in the first paragraph is lax symmetric
    monoidal as well. 
\end{proof}

 \begin{lemma}
 \label{lemma:compute-completion}
 The map
\begin{equation}
\label{eq:completed-THH-computation}	
 THH^\solid(\SS[[x]]) \to \varprojlim THH(\SS[x]/x^n) 
\end{equation}
 is an equivalence (in condensed spectra).
 \end{lemma}
 \begin{proof}
 To see this, we first note that $THH^\solid(\SS[[x]])$ is a colimit of connective condensed
 spectra (by induction on the $n$-skeleta of the corresponding simplicial object and using
 \eqref{eq:solid-tensor-product-calculation}). Then we note that the Hurewicz theorem continues to
 hold in solid spectra, i.e. that it suffices to show that the map
 \eqref{eq:completed-THH-computation} is an equivalence after taking its solid tensor product with
 $\Z$. But the latter lands us into a computation in the derived category of solid abelian groups
 \cite{condensed-lectures}. We then compute $HH^\solid(\Z[[x]])$ via the standard resolution
 of the diagonal map $\Z[[x,y]] \to \Z[[x]] = A$ given by 
 \[ 0 \to A \hat{\tensor} A \xrightarrow{ \cdot (x \mapsto 1 \tensor x - x \tensor 1)} A
 \hat{\tensor} A \xrightarrow{m} A \to 0\]
and conclude by the earlier $\lim^1$ computations for $\varprojlim HH(\Z[x]/x^n)$. 
  \end{proof}

Then by smashing with $\Z$ in solid spectra, the same resolution as above shows that 
\[ THH^{\solid}(\SS[[x]])[x^{-1}] \to  THH^{\solid}(\SS((x))) \]
is an equivalence. 

\begin{remark}
We would be able to show that \eqref{eq:general-calculus-diagram-spherical-formal-power-series} has
a natural lift to solid spectra, if we were able to show some analogs of statements like Burklund's
result for solid spectra. We do not pursue this here. 	
\end{remark}

\subsection{A straightforward interpretation of completed Hochschild homology.}

By the above, we see that for a $dg$ algebra $A$ over a Novikov field like $\kk((x))$, 
\[ HH^\completed(A/\kk) = HH^{\solid}(A/\kk)(pt) = A \tensor^{\solid}_{A \tensor^{\solid}_{\kk} A}
A, \]
since we are just taking cyclic bar constructions in an abelian category, which have the usual
relations to two-sided bar resolutions. 

\section{Some notes on completed cyclic HKR. }
A certain strengthening of Proposition \ref{prop:completed-u-connection} holds, namely an analog of
the theorem of \cite{toen} in the adically completed setting. We sketch proofs of two variants of
this statement, one via computation and one via condensed mathematics. Neither of these are required
for the main theorems of the text.
\subsection{Proof via cyclic shuffles.}
\begin{lemma}
The left two maps of \eqref{eq:first-exact-sequence-completed}, when viewed as maps of
$\EE_1$-algebras in $\kk$-modules, can be promoted to maps of $\EE_1$-algebras in
$\kk[\epsilon]$-modules, and the HKR map extends to a map (and thus an equivalence) of
$\EE_1$-algebras in $\kk[\epsilon]$-modules. 
\end{lemma}

\begin{proof}
	We first note that the terms of a map on negative cyclic homology between mixed complexes are
	exactly the data of a map of $A_\infty$-modules over $k[\epsilon]$ \cite[Section
	2.1]{ganatra2019cyclic}; in particular, an $S$-morphism \cite[2.5.14, see also
	4.3.8]{loday2013cyclic} is the same thing as a map of $A_\infty$-modules over $k[\epsilon]$. Thus, 
    \cite[Theorem~4.4]{getzler1990algebras}
    is equivalent to the statement that
	the cyclic shuffle maps  $\tilde{m}_k$ of that theorem give the data of an $A_\infty$-algebra structure on
	$HH^\completed_\bullet(A/\kk)$ and $HH(A/\kk)$ in the $dg$-category of
	$A_\infty$-modules over $k[\epsilon]$.

 To enhance the HKR map to an equivalence of $A_\infty$-algebras over $k[\epsilon]$, we let the
 higher terms of the $A_\infty$-morphism be given by the $B_k$ operators of \cite[Lemma
 4.3]{getzler1990algebras}; the fact that they all land in the subcomplex of degenerate chains and
 also all such operations vanish on that subcomplex makes it straightforward to verify that the
 equations for an $A_\infty$-morphism are exactly implied by the equation of \cite[Lemma
 4.3]{getzler1990algebras}.
\end{proof}

\subsection{Proof via condensed methods.}
In this argument we freely use the foundations summarized in Appendix \ref{sec:solid-spectra}. We
note that there is a symmetric monoidal abelian category $(Mod^\solid_\kk, \tensor^\solid_\kk)$ of
solid $\kk$-modules; this category is equipped with functors of abelian categories
\[ Mod_\kk \to Mod^\solid_\kk \to Mod_\kk, A \mapsto A^\solid \mapsto A^\solid(*) \]
which factorize the identity functor. The first functor preserves all colimits and finite limits,
and the second preserves all limits and colimits. The first functor above is moreover symmetric
monoidal. Thus there are objects $R^\solid \in Mod^{\solid}_\kk$ lifting $R \in Mod_\kk$ for $R =
\kk[[x]], \kk((x))$ which satisfy \eqref{eq:elementary-tensor-products} with respect to
$\tensor^\solid_\kk$ and such that the images of these identities in $Mod_\kk$ are precisely
\eqref{eq:elementary-tensor-products}; these are defined by looking at the images of the ordinary
rings $\kk[x]/x^n$ and taking the appropriate inverse limits and localizations in $Mod^\solid_\kk$. 

From this it follows that
\[ HH^\completed(R/\kk) = HH^\solid(R/\kk)(*) \]
where the latter is a cdga in $Mod^\solid_\kk$. The formulae defining the solid analog of the HKR
map continue to make sense in this setting, and the earlier argument shows that this map is also an
equivalence. Thus the diagram \eqref{eq:first-exact-sequence-completed} lifts to a corresponding
diagram in cdgas in $Mod^\bullet_\kk$. 

We now recall the strategy of the proof of the cyclic HKR theorem of \cite{toen}. The
idea is that the left adjoints to the forgetful functors: 
\[ S^1 \tensor - : cdga_\kk \to cdga(\kk[S^1]-mod): F \]
\[ \widetilde{dR}: cdga_\kk \to cdga(\kk[\epsilon]-mod): F\]
exist since the $\infty$-categories on the right hand side are cocomplete; so by the equivalence of
dg-bialgebras $k[\epsilon] \simeq k[S^1]$, their images on a given objects must be the same. It then
suffices to verify, via a calculation involving some model category theory, that $S^1 \tensor A$ is
computed by $HH(A/\kk)$, and that $\widetilde{dR}(A)$ is computed by
$\tilde{\Omega}^\solid_{A/\kk}$. 

\begin{remark}
	There are (injective and projective) model structures on cdgas internal to an abelian category
	enriched in modules over a field of characteristic zero \cite[Prop. 4.5.4.6, 7.1.4.7]{HA}, and the
	arguments that $\infty$-categories $\EE_\infty$-algebras in the derived $\infty$-categories of
	abelian categories are equivalent to the $\infty$-categories of CDGAs in the abelian categories go
	through in that generality. 
\end{remark}

To verify the final claims, one is required to do some computation.  
\begin{enumerate}[(a)]
\item Commutative differential graded algebras are tensored over simplicial sets, and one
can show that there is an equality of chain complexes
\[ HH(R/\kk) = S^1 \tensor R\]	
(taken in the strict sense) when $R$ is a commutative algebra, when one takes $S^1$ to be the
simplicial circle;
\item There is an explicit equivalence of simplicial sets $S^1 \simeq B\Z$, where the latter is also
a simplicial abelian group; this induces an equivalence 
\[ S^1 \tensor R \simeq B\Z \tensor R, \]
with the latter now a cdga internal to the category of $\kk[B\Z]$-algebras. Now when $R$ is a
cofibrant cdga, $B\Z \tensor R$ is also cofibrant; writing $c$ for the cofibrant replacement functor
in cdgas,  the collapse map $B\Z \tensor cR \to B\Z \tensor R$ is manifestly a levelwise
quasi-isomorphism whenever the underlying $\kk$-module of $R$ is cofibrant. Thus, under this latter
condition, the Hochschild chain complex computes the $\infty$-categorical functor $S^1 \tensor R$. 
\item Similarly, the functor $dR$ on the level of categories is manifestly $R \mapsto
\tilde{\Omega}^\bullet_{R/\kk}$. The question is whether the left derived functor of this functor
agrees  with the original functor when $R$ is smooth over $\kk$. This is true whenever 
\begin{equation}
\label{eq:cotangent-complex-and-kahler-differentials}
	\mathbb{L}\Omega^1_{R/\kk} = \Omega^1_{R/\kk}, 
\end{equation} 
since taking exterior powers does not need to be derived when $\Q \subset \kk$. But K\"ahler
differentials compute the cotangent complex when $R$ is smooth over $\kk$. 
\end{enumerate}

Now, essentially all points in this argument straightforwardly carry over to the solid context. The
equivalence between $\kk[S^1]$ and $\kk[\epsilon]$ is a fact in classical (as opposed to condensed)
homotopy theory, thus applying the functor from cdgas to solid cdgas gives us an equivalence of
$\infty$-categories
\[cdga(\kk[S^1]^\solid-mod) \simeq cdga(\kk[\epsilon]^\solid-mod).\]
Items (a) and (b) go through as well for formal reasons; it happens to be that $\kk[x]$ is a
cofibrant cdga in $\kk$-modules, since it is the free commutative algebra on the extremally
disconnected set given by a single point; and $\kk((x))$ and $\kk[[x]]$ are cofibrant as
$\kk$-modules since infinite products of free $\kk$-modules are projective generators of the latter. 
The argument in (c) similarly goes through once we have verified the analog of
\eqref{eq:cotangent-complex-and-kahler-differentials} for our algebra at hand. To see this, we note
first that 
\begin{equation}
\label{eq:idempotency}	
R \tensor^L_{\kk[x]} R = R 
\end{equation}

in solid $\kk$-modules, because this is equivalently
\[ (R \tensor^L_\kk R) \tensor^L_{\kk[x] \tensor^L_\kk \kk[x]} \kk[x] \]
(where all tensor products are implicitly solid); by the earlier arguments, only the tensor product
in the middle has to be derived, and the two term resolution of the $\kk[x] \tensor_\kk \kk[x]$
module $\kk[x]$ allows us to conclude \eqref{eq:idempotency}.

Now, because $\kk[x]$ is the free solid $cdga$ on one variable, we have that its cotangent complex
is computed by its K\"ahler differentials. Now, the cotangent complex of $A/\kk$ corepresents the
functor from $A$-modules $M$ to $A$-algebra maps $A \to A \oplus M$; \eqref{eq:idempotency} thus
shows that 
\[\mathbb{L}\Omega^1_{R/\kk} = \Omega^1_{\kk[x]/\kk} \tensor_{\kk[x]}^L R = \hat{\Omega}^1_{R/\kk}
\]
as desired.

\printbibliography %Prints bibliography

\end{document}